%% file: non-linear-bv-2d.tex
\documentclass[11pt,a4paper,reqno]{article}
\usepackage{amsmath,amsthm,amssymb}
\usepackage[margin=1in]{geometry}
\usepackage{authblk}
\usepackage{mathtools}
\usepackage{esint}

\newcommand{\ats}[3][2]{\alpha_{#3}^{#2}}
\newcommand{\up}[3][2]{u_{#3}^{#2\,+}}
\newcommand{\wats}[3][2]{\widehat{\alpha}_{#3}^{#2}}
\newcommand{\um}[3][2]{u_{#3}^{#2\,-}}
\newcommand{\vm}[3][2]{v_{#3}^{#2\,-}}
\newcommand{\vp}[3][2]{v_{#3}^{#2\,+}}

\newcommand{\halfarrow}{\rightharpoonup}

\usepackage{multirow}
\usepackage[table,xcdraw]{xcolor}

\usepackage{mathrsfs}
\numberwithin{equation}{section}
\usepackage[labelformat=simple]{subcaption}

\usepackage{enumitem}
\usepackage{mathtools}  

\theoremstyle{plain}
\newtheorem{theorem}{Theorem}[section]

\newtheorem{proposition}[theorem]{Proposition}
\newtheorem{definition}[theorem]{Definition}
\newtheorem{example}[theorem]{Example}
\newtheorem{lemma}[theorem]{Lemma}
\newtheorem{remark}[theorem]{Remark}

\newtheorem{corollary}{Corollary}[section]

\newcolumntype{H}{>{\setbox0=\hbox\bgroup}c<{\egroup}@{}}

\definecolor{refkey}{rgb}{0,0,1}
\definecolor{labelkey}{rgb}{0,0,1}

\usepackage[normalem]{ulem}
\allowdisplaybreaks

\begin{document}
	\title{Strong bounded variation estimates for the multi--dimensional finite volume approximation of scalar conservation laws and  application to a tumour growth model}
	\author[1]{Gopikrishnan Chirappurathu Remesan\thanks{mail: gopikrishnan.chirappurathuremesan@monash.edu}}
	\date{\today}
	\affil[1]{\small{IITB -- Monash Research Academy, Indian Institute of Technology Bombay, Mumbai, Maharashtra 400076, India}}
	\maketitle
	
	\begin{abstract}
		A uniform bounded variation estimate for finite volume approximations of the nonlinear scalar conservation law $\partial_t \alpha + \mathrm{div}(\boldsymbol{u}f(\alpha)) = 0$ in two and three spatial dimensions with an initial data of bounded variation is established.  We assume that the divergence of the velocity $\mathrm{div}(\boldsymbol{u})$ is of bounded variation instead of the classical assumption that $\mathrm{div}(\boldsymbol{u})$ is zero. The finite volume schemes analysed in this article are set on nonuniform Cartesian grids. A uniform bounded variation estimate for finite volume solutions of the conservation law $\partial_t \alpha + \mathrm{div}(\boldsymbol{F}(t,\boldsymbol{x},\alpha)) = 0$, where $\mathrm{div}_{\boldsymbol{x}}\boldsymbol{F} \not=0$ on nonuniform Cartesian grids is also proved. Such an estimate provides compactness for finite volume approximations in $L^p$ spaces, which is essential to prove the existence of a solution for a partial differential equation with nonlinear terms in $\alpha$, when the uniqueness of the solution is not available. This application is demonstrated by establishing the existence of a weak solution for a model that describes the evolution of initial stages of breast cancer proposed by S. J. Franks et al.~\cite{Franks2003424}. The model consists of four coupled variables: tumour cell concentration, tumour cell velocity--pressure, and nutrient concentration, which are governed by a hyperbolic conservation law, viscous Stokes system, and Poisson equation, respectively. 
		  Results from numerical tests are provided and they complement theoretical findings.  \vspace{0.5cm} 
		\vspace{0.5cm} 
	\end{abstract}

	\section{Introduction}
	\label{sec:intro}
	\noindent Consider the following scalar hyperbolic conservation law in $\mathbb{R}^2$ with a homogeneous source term and an initial data of bounded variation $(BV)$: 
	\begin{align}
	\label{eqn:cons_law}\left.
	\begin{array}{r l}
	\partial_t \alpha + \mathrm{div}(\boldsymbol{u}f(\alpha)) &={} 0\;\;\mathrm{ in }\; \Omega_T\;\;\textmd{and} \\
	\alpha(0,\cdot) &= \alpha_0\;\;\mathrm{ in }\; \Omega,
	\end{array}
	\right\}
	\end{align}
	where $\alpha$ is the unknown, $\alpha_0 : \Omega \rightarrow \mathbb{R}$ is  known \emph{a priori} function of $BV$, $\boldsymbol{u} = (u,v)$ is the advecting velocity, $\Omega_T := (0,T) \times \Omega$, $\Omega := \mathrm{I} \times \mathrm{J}$, $\mathrm{I} := (a,b) \subset \mathbb{R}$ and $\mathrm{J} := (c,d) \subset \mathbb{R}$ are intervals. For technical simplicity assume that $\boldsymbol{u} = \boldsymbol{0}$ on $\partial \Omega$. The function $f$ quantifies the amount of material advected with the velocity  $\boldsymbol{u}$ and is called the \emph{flux function}. We assume that $f$ is Lipschitz continuous with Lipschitz constant,  $\mathrm{Lip}(f)$, which is a classical assumption in literature~\cite{eymard}. Finite volume methods are extensively used to discretise and compute numerical solutions to~\eqref{eqn:cons_law} since such schemes respect the conservation of mass property associated with the underlying partial differential equation (p.d.e.). 
	
	\medskip
	\noindent \textbf{Motivation}\medskip\\
	Conservation laws of the form~\eqref{eqn:cons_law} are crucial in practical applications. Usually they model density or concentration of a conserved quantity in a coupled system, where the conservation law is strongly entangled with the equation that governs the advecting velocity, and with other governing equations, if present.
	
	A wide class tumour growth models based on multiphase mixture theory~\cite{multi-helen} contain a coupled system of a conserved variable and corresponding advecting velocity.  For instance consider a model developed by S. J. Franks et al.~\cite{Franks2003424} that depicts \emph{ductal carcinoma in situ} -- the initial stage of breast cancer. In two spatial dimensions, the model  describes the evolution of an advancing tissue in a cylindrical domain with rigid walls, see Figure~\ref{fig:domain_adv}.  
	\begin{figure}[htp]
		\centering
		\includegraphics[scale=1]{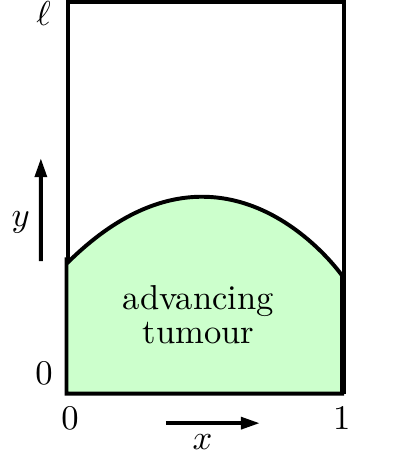}
		\caption{Advancing tumour in the duct $(0,1) \times (0,\ell)$}
		\label{fig:domain_adv}	
	\end{figure}
	
	To keep the discussion simple, we consider the model with simplified kinetics, wherein the viscosity, denoted by $\mu$, inside and outside the tumour is assumed to be uniform and  divergence of the velocity field is assumed to depend only on nutrient concentration.  The domain of tumour growth is denoted by $\Omega = \{ \boldsymbol{x} := (x,y)\,:\, 0 \le x \le 1,\, 0 \le y \le \ell \}$. Here, $x$ is radial distance, $y$ is the axial distance, and $\ell > 2$ is the duct length. For $T < \infty$, time--space domain is denoted by $\Omega_T = (0,T) \times \Omega$ and $t \in (0,T)$ is the time variable. The model variables are concentration of the tumour cells $\alpha(t,\boldsymbol{x})$, velocity of the tumour cells $\boldsymbol{u}(t,\boldsymbol{x}) := (u(t,\boldsymbol{x}),v(t,\boldsymbol{x}))$, pressure inside the tumour $p(t,\boldsymbol{x})$, and nutrient concentration  $c(t,\boldsymbol{x})$. The model seeks a four tuple $(\alpha,p,\boldsymbol{u},c)$ such that, in $\Omega_{T}$ it holds
	
	\begin{subequations}
		\begin{align}
		\mathrm{tumour\;cell\;concetration} &\left\{ \begin{array}{c}
		\dfrac{\partial \alpha}{\partial t} + \mathrm{div}(\boldsymbol{u}\alpha) ={} \gamma \alpha (1 - c),
		\end{array} \label{eqn:cell_vf} \right.\\
		\mathrm{velocity-pressure\;system} &\left\{
		\begin{array}{r l}
		-\mu \left( \Delta \boldsymbol{u} + \dfrac{1}{3} \nabla (\mathrm{div}(\boldsymbol{u})) \right) + \nabla p&={} \boldsymbol{0}, \\
		\mathrm{div}(\boldsymbol{u})&={} \gamma (1 - c),  \text{ and} 
		\end{array} \label{eqn:pv_sys} \right. \\
		\mathrm{nutrient\;concentration} &\left\{
		\begin{array}{c}
		-\Delta c = Q\alpha, 
		\end{array} \label{eqn:nutr} \right.
		\end{align}
		\label{eqn:ductal_tum}
			\end{subequations}
		with appropriate boundary conditions. In~\eqref{eqn:cell_vf}, $\gamma$ is a positive constant that controls the rate of cell division and in~\eqref{eqn:nutr}, $Q$ is a positive constant that controls the nutrient intake by the cells.

	 Another example is the two--phase tumour spheroid growth problem~\cite{droniou2019convergence}, where velocity of the tumour cells $\boldsymbol{u}$ is governed by 
	\begin{align}
	\left.
	\begin{array}{r l}
	-\text{div} \left( \mu \alpha (\nabla {\boldsymbol u} + (\nabla {\boldsymbol u})^T)+ \lambda \alpha \text{div}({\boldsymbol u})\mathbb{I}_{2}\right) + \nabla p &= -\nabla\left( \dfrac{(\alpha - \alpha^{\ast})^{+}}{(1 - \alpha)^2} \right) \textrm{ and }\\
	- \text{div} \left(\dfrac{1-\alpha}{k\alpha} \nabla p \right) + \text{div} (\boldsymbol{u}) &= 0,
	\end{array} \right\}
	\label{eqn:cellphase}
	\end{align}
	where $\mu$ and $\lambda$ are the viscosity coefficients, $k$ is the traction coefficient, $\alpha^{\ast}$ is a positive parameter that controls intra--cellular attraction, $p$ is the pressure, $\mathbb{I}_2$ is the $2 \times 2$ identity tensor, and $\alpha$ is evolves with respect to~\eqref{eqn:cons_law} with a nonlinear source function in $\alpha$.
	
	To show that a possible limit of discrete solutions obtained from a finite volume scheme applied to~\eqref{eqn:cons_law} or~\eqref{eqn:cell_vf} satisfies~\eqref{eqn:ductal_tum} or~\eqref{eqn:cellphase}, respectively and hence to prove the existence of a solution, we need to establish that the discrete solutions converge to the limit in strong $L^p$--norm, where $p \ge 1$. Otherwise, it becomes challenging to apply pass to the limit arguments on functions of $\alpha$ appearing in~\eqref{eqn:ductal_tum} and~\eqref{eqn:cellphase}. A feasible way to obtain strong $L^p$--norm convergence is to show that the discrete solutions have uniform $BV$ and invoke Helly's selection theorem (Theorem~\ref{appen_id.c}) to extract a strongly converging subsequence. 
	 Moreover, the velocity vector field $\boldsymbol{u}$ is not necessarily divergence--free of which~\eqref{eqn:pv_sys} is a direct example. The divergence of the velocity field manifests as a source term in~\eqref{eqn:cons_law}. Hence, while attempting to obtain a uniform $BV$ estimate on discrete solutions of~\eqref{eqn:cons_law}, we need to account for divergent velocity vector fields also.
	
	\medskip
	\noindent \textbf{Literature}\medskip\\
	Total variation properties of weak and entropy solutions of~\eqref{eqn:cons_law} are rather classical results. E. Conway and J. Smoller~\cite{Conway196695} studied conservation laws of the form 
	\begin{equation}
	\partial_t \alpha + \sum_{j=1}^d \partial_{ x_j}f_j(\alpha) = 0,
	\label{eqn:bv-class}
	\end{equation}
	where $BV$ initial data and $(f_j)_{j=1,\ldots,d}$ are assumed to be in $\mathscr{C}^1(\mathbb{R};\mathbb{R})$. They studied
	a finite difference scheme on a uniform Cartesian grid (see Definition~\ref{defn:admis_grid}) and showed that discrete solutions have uniform $BV$. The limit solution obtained from a strongly convergent subsequence is then showed to be a weak solution and is a function with $BV$. N. Kuznetsov~\cite{Kuznetsov1976105} provided early results on $BV$ properties of entropy solutions of~\eqref{eqn:bv-class}. This article~\cite{Kuznetsov1976105} establishes that the $BV$ seminorm of the entropy solution to~\eqref{eqn:bv-class} at any time is bounded by the $BV$ seminorm of the initial data. M. G. Crandall and A. Majda~\cite{crandall} considered monotone finite difference approximations of~\eqref{eqn:bv-class} with $BV$ initial data on uniform Cartesian meshes and established uniform $BV$ estimate for discrete solutions. This estimate is used to prove the convergence of the discrete solutions to the unique entropy solution in strong $L^1$--norm and to prove that the entropy solution also inherits the $BV$ property of the discrete solutions. Later, this work was extended to nonuniform Cartesian meshes by R. Sanders~\cite{Sanders198391}. B. Merlet and J. Vovelle~\cite{Merlet2,Merlet1} considered linear advection equations of the form~\eqref{eqn:cons_law} with $f(\alpha) = \alpha$, $\boldsymbol{u} \in W^{1,\infty}(\mathbb{R}^{+}\times\mathbb{R}^d;\mathbb{R}^d)$, and $\mathrm{div}(\boldsymbol{u}(t,\cdot)) = 0$. The $BV$ seminorm of the unique weak solution of this problem, constructed using the characteristic method, is bounded and the bound depends on the $BV$ seminorm of the initial data. However, discrete solutions corresponding to this problem obtained by using finite volume schemes on general polygonal meshes are not proved to satisfy a uniform $BV$ estimate (see the Remark 1.5 in~\cite[p.~7]{Merlet1}). In fact, to show that the finite volume solutions converge to the entropy solution, whose existence is known \emph{a priori}, it is enough to have a weak $BV$ estimate~\cite[p.~143]{Champier1993139}\cite[p.~161]{eymard} of the following form 
	\begin{align}
	\sum_{n = 0}^{N} \delta \sum_{\mathrm{e}} \vert f(\alpha_{\mathrm{e}}^{p}) - f(\alpha_{\mathrm{e}}^{n}) \vert \left \vert \int_{\mathrm{e}}\boldsymbol{u}(t_{n},\cdot)\cdot\boldsymbol{n}_{\mathrm{e}}\,\mathrm{d}s \right\vert \le \mathrm{C}h^{-1/2},
	\label{eqn:weak-bvrate}
	\end{align}
	where $\delta$ is the temporal discretisation factor, $h$ is the spatial discretisation factor, $\mathrm{e}$ is an edge of a polygon $K$ in the mesh, $\boldsymbol{n}_{\mathrm{e}}$ is the outward unit normal to $\mathrm{e}$ with respect to $K$, $\alpha_{\mathrm{e}}^{(p/n)}$ are the values of a discrete solution on the neighbouring polygons of $\mathrm{e}$. The weak $BV$ estimate ensures convergence in nonlinear weak--$\ast$ sense (see Definition 6.3 in~\cite[p.~100]{eymard}) to a Young measure, called a \emph{process solution}. It can be established that the process solution is indeed a function by proving the uniqueness of the process solution. In this scenario, the nonlinear weak--$\ast$ convergence actually becomes strong $L^p$ convergence (see Theorem 6.4 and 6.5 in~\cite[p.~187-188]{eymard}). Uniqueness of the process solution is crucial in this technique and hence, it is difficult to use it in the case of coupled systems like~\eqref{eqn:ductal_tum} and~\eqref{eqn:cellphase}. The relationship between process solution and function solution is not very clear in this case and an \emph{a priori} compactness result like a uniform $BV$ estimate is necessary to obtain strong $L^p$ convergence. 
	
	A recent uniform $BV$ estimate on finite volume solutions of conservation laws of the form~\eqref{eqn:bv-class} on uniform Cartesian grid is obtained by K. H. Karlsen and J. D. Towers~\cite{Karlsen2017515}. They consider~\eqref{eqn:bv-class} with an auxiliary boundary condition $\boldsymbol{f}\cdot\boldsymbol{n}_{\Omega} = 0$, where $\boldsymbol{n}_{\Omega}$ is the outward unit normal to $\partial \Omega$. C. Chainais-Hillairet~\cite{claire_1999} also provides a uniform $BV$ estimate on finite volume solutions of fully nonlinear conservations laws on uniform square Cartesian grids (see subsection~\ref{subsec:claire} for details).  
	
In~\cite[p.~153]{eymard}, it is stated that  weak $BV$ estimates may be extended to the case with $\mathrm{div}(\boldsymbol{u}) \not= 0$. It is also mentioned in~\cite[p.~154]{eymard} that $BV$ estimates in higher dimensional Cartesian grids reduces to a one dimensional discretisation. However, the corresponding proofs are not provided and we address this aspect.\medskip\\
	\noindent \textbf{Contributions}\medskip\\
	In all of the works reviewed above, either the advecting velocity vector is component--wise constant (see~\eqref{eqn:bv-class}) or the advecting velocity is assumed to be divergence--free. However, these may not be  realistic assumptions in applications as evident from~\eqref{eqn:ductal_tum} and~\eqref{eqn:cellphase}. While discretising physical models, it is imperative to refine the regions where discontinuities of the solution are expected and to retain other regions relatively coarse so that the scheme remains economical. A uniform $BV$ estimate is crucial in enabling the nonlinear terms to converge and hence to prove existence of a solution. 
	
	 The main contributions of this article are stated below.
	\begin{enumerate}[label= \Large${\bullet}$,ref=$\mathrm{(C.\arabic*)}$,leftmargin=\widthof{(C.4)}+3\labelsep]
		\item\label{c.it.1} In the conservation law~\eqref{eqn:cons_law}, the assumption that $\mathrm{div}(\boldsymbol{u}) = 0$ is relaxed. 
		\item\label{c.it.2} A finite volume scheme on nonuniform Cartesian grids in two and three spatial dimensions is considered, and the analysis holds in general for the class of monotone numerical fluxes. The nonuniformity of Cartesian grid can be used to refine the mesh adaptively and economically.
		\item\label{c.it.3} Finite volume solutions satisfy a uniform $BV$ estimate in space and time and this result is extended to the case of fully nonlinear conservation laws analysed by C. Chainais-Hillairet~\cite{claire_1999}.
		\item\label{c.it.4} The existence of a weak solution for~\eqref{eqn:ductal_tum} is shown by utilising the $BV$ estimates on Cartesian grids.  Compactness results rendered by uniform $BV$ present a convergent subsequence out of a family of discrete solutions constructed by applying a finite volume scheme to~\eqref{eqn:cell_vf}, whose limit is shown to be a weak solution of~\eqref{eqn:cell_vf}.
	\end{enumerate}
The uniform $BV$ estimate in space and time for linear and nonlinear conservation laws is obtained by computing the variation of the discrete solution along orthogonal Cartesian axes separately. This method has two major difficulties. Firstly, the term $\alpha\,\mathrm{div}(\boldsymbol{u})$ serves as an additional source function since divergence of the velocity field is not zero. The difference of $\alpha\,\mathrm{div}(\boldsymbol{u})$ at time step $t_{n+1}$ across neighbouring control volumes is estimated in terms of the $BV$ seminorm of $\mathrm{div}(\boldsymbol{u})$ and $L^\infty$ bound of $\alpha$ at time step $t_{n}$. Secondly, while estimating the difference of the discrete solution across two control volumes in $x$ direction, we obtain terms that contain differences of numerical fluxes across the other orthogonal direction and vice--versa. This is a potential obstacle to the standard technique of writing the variation of the discrete solution at $t_{n+1}$ across two control volumes as a convex linear combination of variations of the discrete solutions across neighbouring control volumes at $t_{n}$. We introduce the idea of an intermediate nodal (edge) flux in two (three) spatial dimensions, which is the numerical flux across control volumes sharing only a single vertex (edge), to transform the differences along $y$ and $z$ directions into that along $x$ direction and vice--versa. This helps to obtain a relation of the form
	\begin{align}
	BV(n+1) \le BV(n) + \int_{t_{n}}^{t_{n+1}}A(t)\,\mathrm{d}t,
	\label{eqn:int_rel}
	\end{align} 
	where $BV(n)$ is the $BV$ seminorm of the discrete solution at $t_n$, and $A(t)$ depends on $BV$ seminorm of $\mathrm{div}(\boldsymbol{u})$ and $\|\nabla \boldsymbol{u}(t,\cdot)\|_{L^\infty(\Omega)}$. Finally, an application of induction on~\eqref{eqn:int_rel} yields the $BV$ estimate on the discrete solution. 
	
	\medskip
\noindent	\textbf{Organisation}
	\medskip
	
	This article is organised in the following fashion. In Section~\ref{sec:ma-th}, we present the necessary notations, assumptions, function spaces, and the finite volume scheme. The main results of this article are also presented in Section~\ref{sec:ma-th}. The uniform $BV$ estimate of finite volume solutions of~\eqref{eqn:cons_law} is presented in Section~\ref{sec:b_var}. In Section~\ref{subsec:claire}, we show the uniform $BV$ estimate for conservation laws with fully nonlinear flux. The numerical results and discussion are presented in Section~\ref{sec:num_tests}. The uniform $BV$ estimate for scalar conservation laws in three spatial dimensions in presented in Section~\ref{sec:3d-bv}. The semi--discrete analysis that proves the existence of a weak solution of~\eqref{eqn:ductal_tum} is conducted in Section~\ref{sec:ductal_tumour}. The conclusions are presented  in Section~\ref{sec:con}.
		
	\section{Main results}
	\label{sec:ma-th}
	Four main results are presented in this article. The first three results establish uniform bounded variation estimates in space and time for 
	\begin{enumerate}[label= \Large${\bullet}$,ref=$\mathrm{(C.\arabic*)}$,leftmargin=\widthof{(C.4)}+3\labelsep]
		\item  conservation laws in two spatial dimensions of the form $\partial_t \alpha + \mathrm{div}(\boldsymbol{u}f(\alpha)) = 0$ in Theorem~\ref{thm:bv-alpha}.
		\item conservation laws in two spatial dimensions with fully nonlinear flux of the form $\partial_t \alpha + \mathrm{div}(\boldsymbol{F}(t,\boldsymbol{x},\alpha)) = 0$ in Theorem~\ref{thm:non-lin:flux}.
		\item  conservation laws in three spatial dimensions of the form $\partial_t \alpha + \mathrm{div}(\boldsymbol{u}f(\alpha)) = 0$ in Theorem~\ref{thm:3d-bv-alpha}.
	\end{enumerate}
	The fourth main result, see Theorem~\ref{thm:convergence}, presented in Section~\ref{sec:ductal_tumour} applies Theorem~\ref{thm:bv-alpha}
		 to establish the existence of a weak solution to the practical problem of interest~\eqref{eqn:ductal_tum}.
	
	\subsection{Preliminiaries}
	\label{sec:prelims}
	\begin{definition}
		A function $\beta \in L^1(\mathcal{A})$, where $\mathcal{A} \subset \mathbb{R}^d$, $d \ge 1$ is of $BV$ if $|\beta|_{BV_{\boldsymbol{x}}(\mathcal{A})} < \infty$, where 
		\begin{align}
		|\beta|_{BV_{\boldsymbol{x}}(\mathcal{A})} := \sup \left\{\int_{\mathcal{A}} \beta\, \mathrm{div}(\boldsymbol{\varphi})\,\mathrm{d}\boldsymbol{x}\,: \boldsymbol{\varphi} \in \mathscr{C}_c^1(\mathcal{A};\mathbb{R}^d),\, |\boldsymbol{\varphi}|_{L^\infty(\mathcal{A})} \le 1 \right\}.
		\end{align}
	\end{definition}
	
	The space $BV_{\boldsymbol{x}}(\mathcal{A})$ is the vector space of functions $\beta \in L^1(\mathcal{A})$ with $BV$. Recall that in this article $\Omega_T$ = $(0,T) \times \Omega$, where $\Omega = (a,b) \times (c,d)$. Then, define the following $BV$ seminorms for a function $\beta : \Omega_T \rightarrow \mathbb{R}$:
	\begin{gather}
	|\beta(t,\cdot)|_{L^1_yBV_x} :={} \int_{\mathrm{J}}|\beta(t,\cdot,y)|_{BV_x(\mathrm{I})}\,\mathrm{d}y,\qquad |\beta(t,\cdot)|_{L^1_xBV_y} :={} \int_{\mathrm{I}}|\beta(t,x,\cdot)|_{BV_y(\mathrm{J})}\,\mathrm{d}x, \\ 
	|\beta(t,\cdot)|_{BV_{x,y}} :={} |\beta(t,\cdot)|_{L^1_yBV_x} + |\beta(t,\cdot)|_{L^1_xBV_y},\\
	|\beta|_{L^1_tBV_{x,y}} :={} \int_{0}^T|\beta(t,\cdot)|_{BV_{x,y}}\,\mathrm{d}t,\qquad |\beta|_{L^1_{x,y}BV_t} := \int_{\Omega} |\beta(\cdot,x,y)|_{BV_t(0,T)}\,\mathrm{d}x\,\mathrm{d}y,\; \mathrm{ and } \\
	\label{eqn:ts-bv-norm}
	|\beta|_{BV_{x,y,t}} :={} |\beta|_{L^1_{x,y}BV_t} + |\beta|_{L^1_tBV_{x,y}}. 
	\end{gather}
	Also, define the following norms for a function $v : X_T \rightarrow \mathbb{R}^d$ ($d \ge 1$), where $X_T := (0,T) \times X$:
	\begin{align*}
	\|v\|_{L^1_tL^\infty(X_T)} := \int_{0}^T \|v(t,\cdot)\|_{L^\infty(X)}\,\mathrm{d}t\;\;\textrm{ and }\;\|v\|_{L^\infty_tL^1(X_T)} := \sup_{0 < t < T}\|v(t,\cdot)\|_{L^1(X)}.
	\end{align*}
	For a function $\beta : (a,b) \rightarrow \mathbb{R}$, define the total variation by $\mathrm{T\cdot V\cdot}(\beta) := \sup_{P} \sum_{i=0}^{I} |\beta(x_{i+1}) - \beta(x_i)|,$
	where $P := \left\{a = x_0,\ldots,x_{I+1} = b\right\}$ is a partition of $(a,b)$. It is a classical result that $|\beta|_{BV_x(a,b)} = \mathrm{T\cdot V\cdot}(\beta)$~\cite[Appendix A]{holden}.  
	\begin{definition}[two dimensional admissible grid]
		\label{defn:admis_grid}
		Let\;$\mathrm{X}_{k} := \left\{x_{-1/2},\ldots,x_{I+1/2}\right\}$ and $\mathrm{Y}_{h} := \left\{y_{-1/2},\ldots,y_{J+1/2}\right\}$, where $x_{-1/2} = a$, $x_{I+1/2} = b$, $y_{-1/2} = c$, $y_{J+1/2} = d$, $k_{i} = x_{i+1/2} - x_{i-1/2}$, $h_{j} = y_{j+1} - y_{j-1/2}$, $h = \max_{i} h_{i}$, and $k = \max_{j}k_{j}$. The Cartesian grid\;\;$\mathrm{X}_{k} \times \mathrm{Y}_{h}$ is said to be a two dimensional admissible grid if for a fixed constant $\widetilde{c} > 0$, it holds that $(\widetilde{c})^{-1} \le \frac{h_j}{k_i} \le \widetilde{c}\;\;\forall i,j$.
		If $k_i = k$ $\forall\,i$ and $h_j = h$ $\forall\,j$, then $\mathrm{X}_{k} \times \mathrm{Y}_{h}$  is called a uniform Cartesian grid and otherwise a nonuniform Cartesian grid, see Figure~\ref{fig:gen_grid}. 
	\end{definition}
Assume that~\ref{as.1}--\ref{as.3} below hold.
\begin{enumerate}[label= $\mathrm{(AS.\arabic*)}$,ref=$\mathrm{(AS.\arabic*)}$,leftmargin=\widthof{(AS.4)}+3\labelsep]
	\item\label{as.1} The flux $f: \mathbb{R} \rightarrow \mathbb{R}$ and the numerical flux $g: \mathbb{R}^2 \rightarrow \mathbb{R}$ are Lipschitz continuous with Lipschitz constants $\mathrm{Lip}(f)$ and $\mathrm{Lip}(g)$, respectively.
	\item \label{as.2} The numerical flux $g$ is monotonically nondecreasing in the first variable and nonincreasing in the second variable, and satisfies $g(a,a) = f(a)\;\forall a \in \mathbb{R}$.
	\item \label{as.3} There exists a constant $\mathscr{C} \ge 0$ such that $$\max\left( ||\boldsymbol{u}||_{L_t^1L^\infty(\Omega_T)}, ||\nabla \boldsymbol{u}||_{L_t^1L^\infty(\Omega_T)}, |\mathrm{div}(\boldsymbol{u})|_{L_t^1BV_{x,y}} \right) \le \mathscr{C}  < \infty.$$ 
\end{enumerate}
	\subsection{Presentation of the numerical scheme}
	\label{sec:presentation}
	Define the spatial discretisation factor $h_{\max}$ by $h_{\max} = \max_{i,j}\left\{k_i,h_j\right\}$, which quantifies the size of the Cartesian grid $\mathrm{X}_{k} \times \mathrm{Y}_{h}$. Let $\mathrm{T}_{\delta}$ defined by $\mathrm{T}_{\delta} := \left\{t_0,\ldots,T_{N} \right\}$ be a discretisation of $(0,T)$, where $t_0 = 0$ and $t_{N} = T$. Define the temporal discretisation factor by $\delta = \max_{n} \delta_n$, where $\delta_n = t_{n+1} - t_{n}$. For technical simplicity a uniform temporal discretisation is taken, wherein $\delta_{n} = \delta\;\;\forall\,n$. However, note that the results in this article hold with a nonuniform temporal discretisation also.
	
	\begin{figure}[htp]
		\centering
		\includegraphics[scale=0.9]{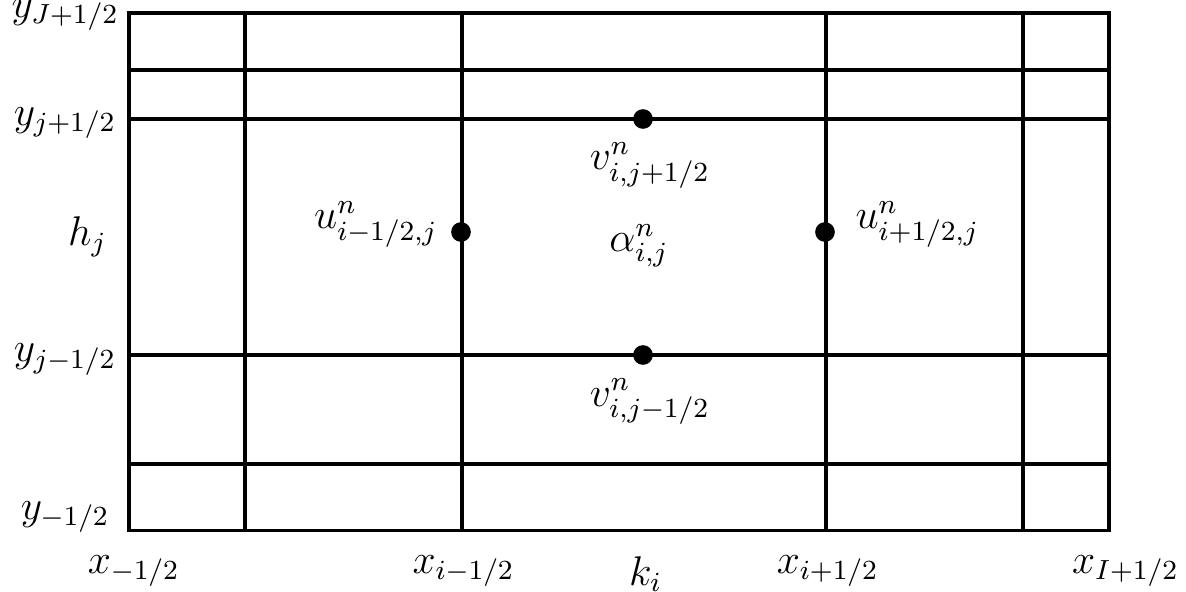}
		\caption{Rectangular grid and locations of the velocities and discrete unknowns $\ats{n}{i,j}$.}
		\label{fig:gen_grid}
	\end{figure}

	Integrate~\eqref{eqn:cons_law} on the time--space control volume $(t_{n+1},t_{n}) \times K_{ij}$, where $K_{ij} = (x_{i-1/2},x_{i+1/2}) \times (y_{j-1/2},y_{j+1/2})$ and apply the divergence theorem to obtain
	\begin{align}
	0 = \int_{t_{n}}^{t_{n+1}} \int_{K_{ij}} \partial_t \alpha \,\mathrm{d}\boldsymbol{x}\mathrm{d}t&+ \int_{t_{n}}^{t_{n+1}} \int_{\partial K_{ij}} f(\alpha) (u,v)\cdot\boldsymbol{n}_{ij} \,\mathrm{d}s\,\mathrm{d}t =: \mathrm{I}_1 + \mathrm{I}_2, 
	\label{eqn:disc_weak}
	\end{align}
	where $\boldsymbol{n}_{ij}$ is the outward unit normal to $\partial K_{ij}$ and $\boldsymbol{u} = (u,v)$. Replace $\mathrm{I}_1$ by the difference formula $k_{i}h_{j}(\ats{n+1}{i,j} - \ats{n}{i,j})$. Term $\mathrm{I}_2$ in~\eqref{eqn:disc_weak} is approximated by the numerical flux $g : \mathbb{R}^2 \rightarrow\mathbb{R}$ as $\delta h_{j}(\mathrm{F}_{i+1/2,j}^n - \mathrm{F}_{i-1/2,j}^n) + \delta k_{i}(\mathrm{G}_{i,j+1/2}^n - \mathrm{G}_{i,j-1/2}^n)$, where
	\begin{align}
	\mathrm{F}_{i-1/2,j}^n :={}& \left(u_{i-1/2,j}^{n\,+}g(\ats{n}{i-1,j},\ats{n}{i,j}) - u_{i-1/2,j}^{n\,-}g(\ats{n}{i,j},\ats{n}{i-1,j}) \right),\\
	\label{eqn:Gij-Jij}
	\mathrm{G}_{i,j-1/2}^n :={}& \left(v_{i,j-1/2}^{n+}g(\ats{n}{i,j-1},\ats{n}{i,j}) - v_{i,j-1/2}^{n-}g(\ats{n}{i,j},\ats{n}{i,j-1}) \right),
	\end{align}
	$a^{+} = \max(a,0)$, and $a^{-} = -\min(a,0)$ for $a \in \mathbb{R}$, 
	\begin{align}
	u_{i-1/2,j}^n = \fint_{t_{n}}^{t_{n+1}}\fint_{y_{j-1/2}}^{y_{j+1/2}} u(t,x_{i-1/2},s)\,\mathrm{d}s\,\mathrm{d}t\;\;\mathrm{ and }\;\;\,v_{i,j-1/2}^n = \fint_{t_{n}}^{t_{n+1}} \fint_{x_{i-1/2}}^{x_{i+1/2}} v(t,s,y_{j-1/2})\,\mathrm{d}s\,\mathrm{d}t.
	\end{align}
	Locations of the discrete unknowns $\ats{n}{i,j}$, velocities $u_{i-1/2,j}$ and $v_{i,j-1/2}$ in a two dimensional admissible grid is shown in Figure~\ref{fig:gen_grid}.  A substitution of approximations of $\mathrm{I}_1$ and $\mathrm{I}_2$ in~\eqref{eqn:disc_weak} leads to
	\begin{subequations}
		\begin{align}
		\ats{n+1}{i,j} = \ats{n}{i,j} - \mu_{i}(\mathrm{F}_{i+1/2,j}^n - \mathrm{F}_{i-1/2,j}^n) - \lambda_{j} (\mathrm{G}_{i,j+1/2}^n - \mathrm{G}_{i,j-1/2}^n),
		\label{eqn:fvm_scheme}
		\end{align}
		where $\mu_i = \delta/k_{i}$ and $\lambda_{j} = \delta/h_{j}$.  We set the discrete initial data as follows
		\begin{align}
		\alpha_{i,j}^0 := \fint_{K_{i,j}} \alpha_0(\boldsymbol{x})\,\mathrm{d}\boldsymbol{x}.
		\label{eqn:alpha_ini}
		\end{align}
	\end{subequations}
	The terms $\mathrm{F}_{i,j}^n$ and $\mathrm{G}_{i,j}^n$ can be expressed as, for $s \in \{-1,1\}$
	\begin{subequations}
		\begin{align}
		\label{eqn:flux_x}
		\mathrm{F}_{i+s/2,j}^n ={}& \mathrm{M}_{i +s/2,j}^{\,x}\left[ \dfrac{(1 - s)}{2}\left(\ats{n}{i-1,j} - \ats{n}{i,j}\right) + \dfrac{(1 + s)}{2}\left(\ats{n}{i,j} - \ats{n}{i+1,j}\right)  \right]  \\
		&\hspace{5cm}+ u_{i+s/2,j}^n f(\alpha_{i,j}^n)\;\;\textrm{and} 
		\\
		\label{eqn:flux_y}
		\mathrm{G}_{i,j+s/2}^n ={}& \mathrm{M}_{i,j+s/2}^{\,y}\left[\dfrac{(1 - s)}{2}\left(\ats{n}{i,j-1} - \ats{n}{i,j}\right) + \dfrac{(1 + s)}{2}\left(\ats{n}{i,j} - \ats{n}{i,j+1}\right) \right] \\
		&\hspace{5cm}+  v_{i,j+s/2}^n f(\alpha_{i,j}^n), 
		\end{align}
		where
		\begin{align}
		\mathrm{M}_{i-1/2,j}^{\,x} :={}& \left[ u_{i-1/2,j}^{n\,+} \,\mathrm{D}_{i,j}^n(\ats{n}{i-1,j},\ats{n}{i,j}) + u_{i-1/2,j}^{n\,-}\, \mathrm{D}_{i,j}^n(\ats{n}{i,j},\ats{n}{i-1,j}) \right], \\
		\mathrm{M}_{i,j-1/2}^{\,y} :={}&  \left[ v_{i,j-1/2}^{n\,+}\, \mathrm{D}_{i,j}^n(\ats{n}{i,j-1},\ats{n}{i,j}) + v_{i,j-1/2}^{n\,-}\, \mathrm{D}_{i,j}^n(\ats{n}{i,j},\ats{n}{i,j-1}) \right],
		\end{align}
		and the difference quotient $\mathrm{D}_{i,j}^n : \mathbb{R}^2 \rightarrow \mathbb{R}$ is defined by 
		\begin{align}
		\mathrm{D}_{i,j}^n(a,b) =  \left\{ \begin{array}{c l}
		\dfrac{g(a,b) - g(\ats{n}{i,j},\ats{n}{i,j})}{a-b} & \text{ if } a\not= b,\;\;\textrm{and}\\
		0 & \text{ if } a =  b.
		\end{array}
		\right.
		\label{eqn:diff_quo}
		\end{align}
	\end{subequations}
	Observe that $\mathrm{D}_{i,j}^n(\ats{n}{i-1,j},\ats{n}{i,j}),\,\mathrm{D}_{i,j}^n(\ats{n}{i,j},\ats{n}{i-1,j}),\,\mathrm{D}_{i,j}^n(\ats{n}{i,j},\ats{n}{i,j-1}),$ and $\mathrm{D}_{i,j}^n(\ats{n}{i,j-1},\ats{n}{i,j})$, (hence, $\mathrm{M}_{i-1/2,j}^{x}$ and $\mathrm{M}_{i,j-1/2}^{y}$) are nonnegative due to the monotonicity of $g$. Use~\eqref{eqn:flux_x} and~\eqref{eqn:flux_y} to transform the right hand side of~\eqref{eqn:fvm_scheme} into a convex linear combination of the terms $\ats{n}{l,m}$, where $(l,m) \in \{(i,j),(i-1,j),(i+1,j),(i,j+1),(i,j-1)\}$, and this yields an alternate form of the discrete scheme~\eqref{eqn:fvm_scheme}
	\begin{align}
	\ats{n+1}{i,j} ={}& \ats{n}{i,j}\left(1 - \mu_{i} \mathrm{M}_{i+1/2,j}^{\,x} - \lambda_{j} \mathrm{M}_{i,j+1/2}^{\,y} - \mu_{i} \mathrm{M}_{i-1/2,j}^{\,x} -  \lambda_{j} \mathrm{M}_{i,j-1/2}^{\,y}\right) \\ &+
	\ats{n}{i+1,j} \mu_{i} \mathrm{M}_{i+1/2,j}^{\,x} + \ats{n}{i,j+1} \lambda_{j} \mathrm{M}_{i,j+1/2}^{\,y} + \ats{n}{i-1,j} \mu_{i} \mathrm{M}_{i-1/2,j}^{\,x} + \ats{n}{i,j-1} \lambda_{j} \mathrm{M}_{i,j-1/2}^{\,y} \\
	&- f(\ats{n}{i,j}) \left(\int_{t_{n}}^{t_{n+1}}\fint_{K_{i,j}} \mathrm{div}(\boldsymbol{u})(t,\boldsymbol{x})\,\mathrm{d}t\,\mathrm{d}\boldsymbol{x} \right).
	\label{eqn:conv_comb}
	\end{align}

	\begin{definition}[Time--reconstruct]
		For a sequence of functions $(f_n)_{\{n \ge 0\}}$, where $f_{n} : X \rightarrow \mathbb{R}$, define the corresponding time--space reconstruct  $f_{h,\delta} : (0,T) \times X \rightarrow \mathbb{R}$ by, for every $t \in (t_{n},t_{n+1})$, $	f(t,\cdot) := f_{n}(\cdot)$
		\label{def:time_rec}
	\end{definition}
	The function $\alpha_{h,\delta} : \Omega_T \rightarrow \mathbb{R}$ is the time--space reconstruct corresponding to the sequence of functions $(\alpha_{h}^n)_{\{n \ge 0\}}$, where $\alpha_{h}^n(\boldsymbol{x}) := \alpha_{i,j}^{n}$ on $K_{i,j}$.

	\begin{subequations}
		\begin{theorem}[bounded variation]
			\label{thm:bv-alpha}
			Let $\mathrm{X}_{k} \times \mathrm{Y}_{h}$ be a two dimensional admissible grid, and assumptions~\ref{as.1}--\ref{as.3} and the Courant--Friedrichs--Lewy (CFL) condition $4 \delta \max_{i,j}( \frac{1}{k_{i}} + \frac{1}{h_j} ) \mathrm{Lip}(g) ||\boldsymbol{u}||_{L^\infty(\Omega_T)} \le 1$	hold. If $\alpha_0 \in L^\infty(\Omega)\cap BV_{\boldsymbol{x}}(\Omega)$, then $\alpha_{h,\delta}$ satisfies $|\alpha_{h,\delta}|_{BV_{x,y,t}} \le \mathscr{C}_{\mathrm{BV}}$,
			\label{eqn:alpha-bv}
		where $\mathscr{C}_{\mathrm{BV}}$ depends on $T$, $\alpha_0$, $f$, $g$, $||\nabla\boldsymbol{u}||_{L^1_tL^\infty(\Omega_T)}$, and $|\mathrm{div}(\boldsymbol{u})|_{L^1_tBV_{x,y}}$.
		\end{theorem}
	\end{subequations}
	\begin{remark}[Boundedness constant $\mathscr{C}_{\mathrm{BV}}$]
	\label{rem:constants}
	The exact dependency of $\mathscr{C}_{\mathrm{BV}}$ on the factors $T$, $\alpha_{0}$, and Lipschitz continuity of fluxes is obtained from the proofs of Proposition~\ref{prop:space_bv} and Proposition~\ref{prop:temp_bv}. The final expression for $\mathscr{C}_{\mathrm{BV}}$ is described by
	\begin{align}
	\mathscr{C}_{\mathrm{BV}} \le T \mathrm{B}_u \mathrm{B}_{\alpha,u}\left(1 + 4\mathrm{Lip}(g)\int_{0}^T ||\nabla \boldsymbol{u}||_{L^\infty(\Omega)}\,\mathrm{d}t\right)  + (\mathrm{Lip}(f) \alpha_M + f_{0} ) || \mathrm{div}(\boldsymbol{u})||_{L^1(\Omega_T)},
	\end{align}
	where $\mathscr{C} := \max\left(\mathrm{Lip}(f)\alpha_M + f_0,3\mathrm{Lip}(f) +  4\mathrm{Lip}(g)(\widetilde{c} + 1) + 1  \right)$, $\mathrm{B}_{\boldsymbol{u}}: = \exp\left(\mathscr{C}\|\nabla\boldsymbol{u}\|_{L^1_tL^\infty(\Omega_T)} \right)$, and $\mathrm{B}_{\alpha,\boldsymbol{u}}: =  |\alpha_0|_{BV_{x,y}} + \mathscr{C} ||\mathrm{div}(\boldsymbol{u})||_{L^1_tBV_{x,y}}$. However, the precise form of $\mathscr{C}_{\mathrm{BV}}$ has little impact on compactness arguments used to extract a strongly convergent subsequence from the family of time--space functions $\{\alpha_{h,\delta}\}$ -- for this purpose it is sufficient that $|\alpha_{h,\delta}|_{BV_{x,y,t}}$ is bounded by a global constant independent of the discretisation factors.
\end{remark}
 Assumptions~\ref{as.1},~\ref{as.2}, and boundedness of~$||\boldsymbol{u}||_{L_t^1L^\infty(\Omega_T)}$ described by~\ref{as.3} are classical in the literature (see~\cite[p.~153]{eymard} and~\cite[p.~130]{claire_1999} for more details). The crucial assumptions of Theorem~\ref{thm:bv-alpha} are the boundedness of (a) $||\nabla \boldsymbol{u}||_{L_t^1L^\infty(\Omega_T)}$ and (b) $|\mathrm{div}(\boldsymbol{u})|_{L_t^1BV_{x,y}}$ described by~\ref{as.3}. Condition (a) is not unexpected since a conventional assumption in estimating $BV$ seminorm of finite volume approximations of nonlinear conservation laws of the form~\eqref{eqn:cons_law} is that $\boldsymbol{u} \in \mathscr{C}^1(\mathbb{R}^d \times \mathbb{R}^{+})$~\cite{claire_1999,eymard}, which yields (a) on compact subsets of $\mathbb{R}^d \times \mathbb{R}^{+}$. Though (b) apparently seems to be restrictive, it is pivotal in bounding the  difference of $\mathrm{div}(\boldsymbol{u})$ between two control volumes (see~\eqref{eqn:kfg}). Indeed, we can relax this assumption to $\mathrm{div}(\boldsymbol{u}) \in L^1_tL^\infty(\Omega_T)$, which is the formally correct choice and is used in the seminal work~\cite{DiPerna1989511} by DiPerna and Lions. However, with this less restrictive assumption, Proposition II.1 in DiPerna and Lions~\cite{DiPerna1989511} only guarantees the existence of a weak solution $\alpha \in L^\infty_t L^1(\Omega_T)$. Therefore, (b) is justified for establishing a stronger convergence of the finite volume solutions  and the higher $BV$ regularity of the limiting solution.
	\section{Proof of Theorem~\ref{thm:bv-alpha}}
	\label{sec:b_var}
	We let the hypotheses of Theorem~\ref{thm:bv-alpha} to hold throughout the sequel of this article and recall that $\alpha_{h,\delta}$ is the time--reconstruct in the sense of Definition~\ref{def:time_rec}. The proof of Theorem~\ref{thm:bv-alpha} is accomplished through three steps: establish the
	\begin{enumerate}[label= \Large${\bullet}$,ref=$\mathrm{(C.\arabic*)}$,leftmargin=\widthof{(C.4)}+3\labelsep]
		\item  boundedness of $\alpha_{h,\delta}$ in Proposition~\ref{prop:boundedness},
		\vspace{-0.2cm}
		\item  spatial $BV$ estimate of $\alpha_{h,\delta}$ in Proposition~\ref{prop:space_bv}, and
		\vspace{-0.2cm}
		\item  temporal $BV$ estimate of $\alpha_{h,\delta}$ in Proposition~\ref{prop:temp_bv}.
	\end{enumerate}
	\begin{proposition}[boundedness]
		\label{prop:boundedness}
		The function $\alpha_{h,\delta}$ satisfies, for every $0 \le t \le T$,
		\begin{align}
		\label{eqn:sup_alpha}
		\left|\alpha_{h,\delta}(t,\cdot)\right|_{L^\infty(\Omega)} \le{}& \mathrm{B}_{f,\boldsymbol{u}} \left(a^0 + f_0 \|\mathrm{div}{(\boldsymbol{u}})\|_{L^1_tL^\infty(\Omega_T)} \right), 
		\end{align}
		where $\mathrm{B}_{f,\boldsymbol{u}} := \exp\left(\mathrm{Lip}(f) \|\mathrm{div}{(\boldsymbol{u}})\|_{L^1_tL^\infty(\Omega_T)} \right)$, $a^0 = \|\alpha_0\|_{L^\infty(\Omega)}$, and $f_0  = f(0)$.
	\end{proposition} 
	\noindent The proof of Proposition~\ref{prop:boundedness} is based on writing $\alpha_{i,j}^{n+1}$ as convex linear combination of values of $\alpha_{h,\delta}$ at the previous time step. 
	\begin{proof}
		The Lipschitz continuity of the function $g$ yields $|\mathrm{M}_{i-1/2,j}^{\,x}| \le \mathrm{Lip}(g) |u_{i-1/2,j}^n|$ and $|\mathrm{M}_{i,j-1/2}^{\,y}| \le \mathrm{Lip}(g) |v_{i,j-1/2}^n|$ for $i = 0,\ldots,I$ and $j = 0,\ldots,J$. 
		The CFL condition in Theorem~\ref{thm:bv-alpha} ensures that the coefficient of~$\ats{n}{i,j}$ in~\eqref{eqn:conv_comb} is nonnegative. Use the properties of convex linear combination of $\{\alpha^n_{l,m}\}$, where $(l,m) \in \{(i,j-1),(i,j+1),(i-1,j),(i+1,j) \}$ in~\eqref{eqn:conv_comb} and the Lipschitz continuity of $f$ to obtain 
		\begin{align}
		\sup_{i,j} \left|\ats{n+1}{i,j}\right| \le{}& \sup_{i,j} \left|\ats{n}{i,j}\right|\left[1 + \mathrm{Lip}(f) \int_{t_{n}}^{t_{n+1}}||\mathrm{div}{(\boldsymbol{u}})(t,\cdot)||_{L^\infty(\Omega)}\,\mathrm{d}t  \right]  \\
		&+ f_0 \int_{t_{n}}^{t_{n+1}}||\mathrm{div}{(\boldsymbol{u}})(t,\cdot)||_{L^\infty(\Omega)}\,\mathrm{d}t. 
		\label{eqn:alpha-sup}
		\end{align}
		An application of induction on~\eqref{eqn:alpha-sup} with $n$ as the index and~\eqref{eqn:alpha_ini} yield~\eqref{eqn:sup_alpha}.
	\end{proof}

	
	\begin{proposition}[spatial variation]
		\label{prop:space_bv}
		The function $\alpha_{h,\delta}$ satisfies $	|\alpha_{h,\delta}(t,\cdot)|_{BV_{x,y}} \le  \mathrm{B}_{\boldsymbol{u}} ( |\alpha_0|_{BV_{x,y}} + \mathscr{C}|\mathrm{div}(\boldsymbol{u})|_{L^1_tBV_{x,y}} )$ for every $0 \le t \le T$,
		where
		$\mathrm{B}_{\boldsymbol{u}}: = \exp\left(\mathscr{C}\|\nabla\boldsymbol{u}\|_{L^1_tL^\infty(\Omega_T)} \right)$ and $\mathscr{C}$ is defined in  Remark~\ref{rem:constants}.
	\end{proposition}
	\noindent The proof of Proposition~\ref{prop:space_bv} is achieved in five intermediate steps, which are as follows.
	\begin{enumerate}[label= $\mathbf{Step\;\;\arabic*}$,ref=$\mathbf{Step\;\;\arabic*}$,leftmargin=\widthof{(Step) 1}+3\labelsep]
		\item\label{ps.1} Write the difference $\ats{n+1}{i,j} - \ats{n}{i,j}$ as $\ats{n+1}{i,j} - \ats{n}{i,j} := \ats{n}{i,j} - \ats{n}{i-1,j} - \mathrm{H}_{i,j} - \mathrm{J}_{i,j}$, where $\mathrm{H}_{i,j}$ collects the variation of $\ats{n}{i,j}$ in $x$--direction and $\mathrm{J}_{i,j}$  the variation of $\ats{n}{i,j}$ in $y$--direction. 
		\item\label{ps.2} Use the intermediate nodal fluxes (see Figure~\ref{fig:nodal_flux}) to transform the vertical differences in $\mathrm{J}_{i,j}$ into horizontal differences. 
		\item\label{ps.3} Use~\ref{ps.1} and~\ref{ps.2} to write $\ats{n+1}{i,j} - \ats{n+1}{i-1,j}$ as a sum of (a) convex linear combinations of $\ats{n}{i,j} - \ats{n}{l,m}$, where $(l,m) \in \left\{(i,j-1),(i,j+1),(i-1,j),(i+1,j) \right\}$ and (b) the variation of $\partial_x u$ and $\partial_y v$ (recall that $\boldsymbol{u} = (u,v)$).
		\item\label{ps.4} Estimate variations of $\partial_x u$ and $\partial_y v$ in terms of the $BV$ seminorm of $\mathrm{div}(\boldsymbol{u})$.
		\item\label{ps.5} Combine the estimates from~\ref{ps.3} and~\ref{ps.4} to bound $|\alpha_{h,\delta}(t_{n+1},\cdot)|_{BV_{x,y}}$ in terms of  $|\alpha_{h,\delta}(t_{n},\cdot)|_{BV_{x,y}}$ and $|\mathrm{div}(\boldsymbol{u})|_{L_t^1BV_{x,y}}$  (see~\eqref{eqn:alpha_sp_bv}) and apply induction on $n$ to obtain the desired conclusion.
	\end{enumerate}
	\begin{proof}
		\hyperref[ps.1]{\textbf{Step 1:}}
		Consider the difference between the scheme~\eqref{eqn:fvm_scheme} written for $\ats{n+1}{i,j}$ and $\ats{n+1}{i-1,j}$
		\begin{align}
		\ats{n+1}{i,j} - \ats{n+1}{i-1,j} ={}& \ats{n}{i,j} - \ats{n}{i-1,j} - \left[ \mu_{i}(\mathrm{F}_{i+1/2,j} - \mathrm{F}_{i-1/2,j}) - \mu_{i-1}(\mathrm{F}_{i-1/2,j} - \mathrm{F}_{i-3/2,j}) \right]\\
		&- \left[\lambda_{j} (\mathrm{G}_{i,j+1/2} - \mathrm{G}_{i,j-1/2}) - \lambda_{j} (\mathrm{G}_{i-1,j-1/2} - \mathrm{G}_{i-1,j-3/2}) \right] \\
		=:{}& \ats{n}{i,j} - \ats{n}{i-1,j} - \mathrm{H}_{i,j} - \mathrm{J}_{i,j}.
		\label{eqn:diff_h}
		\end{align}
		\begin{subequations}
			The term $\mathrm{H}_{i,j}$ gathers the variation in the $x$--direction; use~\eqref{eqn:flux_x} to rewrite $\mathrm{H}_{i,j}$ as
			\begin{multline}
			\mathrm{H}_{i,j} ={} \mu_{i} \mathrm{M}_{i-1/2,j}^{\,x} \left( \ats{n}{i,j} - \ats{n}{i-1,j} \right) + \mu_{i}\mathrm{M}_{i+1/2,j}^{\,x} \left( \ats{n}{i,j} - \ats{n}{i+1,j} \right)  \\
			+ \mu_{i-1} \mathrm{M}_{i-1/2,j}^{\,x} \left( \ats{n}{i,j} - \ats{n}{i-1,j} \right) + \mu_{i-1}\mathrm{M}_{i-3/2,j}^{\,x} \left( \ats{n}{i-2,j} - \ats{n}{i-1,j} \right) + \mathrm{K}_{i,j}^f,\hspace{1cm}
			\label{eqn:space_var}
			\end{multline}
			where
			\begin{align}
			\mathrm{K}_{i,j}^f :=  f(\ats{n}{i,j}) \int_{t_{n}}^{t_{n+1}}\fint_{K_{i,j}}\partial_{x} u(t,\boldsymbol{x})\,\mathrm{d}\boldsymbol{x}\,\mathrm{d}t -  f(\ats{n}{i-1,j}) \int_{t_{n}}^{t_{n+1}}\fint_{K_{i-1,j}} \partial_{x} u(t,\boldsymbol{x})\,\mathrm{d}\boldsymbol{x}\,\mathrm{d}t.
			\label{eqn:kij_defn}
			\end{align}
		\end{subequations}
		\begin{figure}[htp]
			\centering
			\begin{subfigure}[b]{0.49\textwidth}
				\centering
				\includegraphics[scale=1]{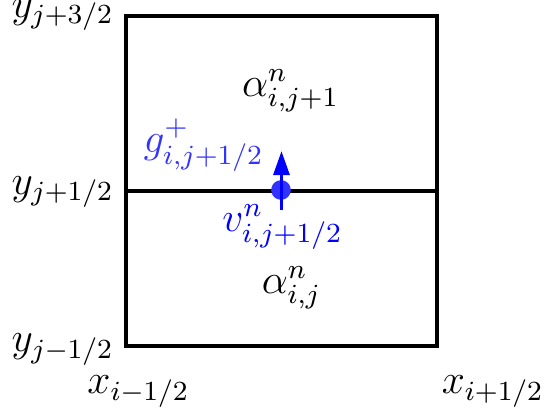}
				\caption{}
				\label{fig:v_flux_p}
			\end{subfigure}
			\begin{subfigure}[b]{0.49\textwidth}
				\centering
				\includegraphics[scale=1]{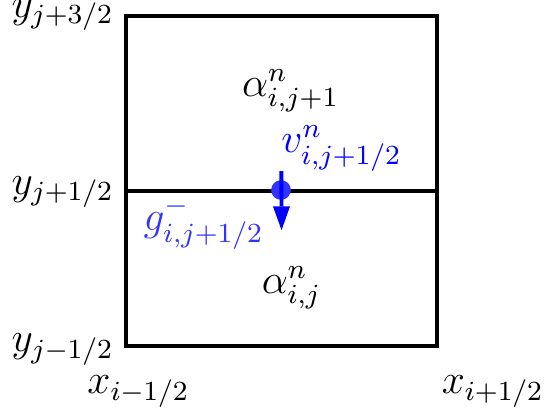}
				\caption{}
				\label{fig:v_flux_m}
			\end{subfigure}
			\caption{Spatial locations of the numerical fluxes $g_{i,j+1/2}^+$ and $g_{i,j+1/2}^{-}$.}
		\end{figure}
		\hyperref[ps.2]{\textbf{Step 2:}} The goal of this step is to transform the horizontal difference of variations between the vertical levels $(i-r,j+s)$ and $(i-r,j-s)$, where $(r,s) \in \{ (0,1/2),(-1,-1/2)\}$ appearing in  $\mathrm{J}_{i,j}$ of~\eqref{eqn:diff_h} so that the resulting terms can be combined to form a convex linear combination of differences of $\alpha_{h,\delta}(t_{n},\cdot)$ between neighbouring rectangles.   Use~\eqref{eqn:Gij-Jij} to rewrite 
		 $\mathrm{J}_{i,j} = \mathrm{J}_{i,j}^{+} - \mathrm{J}_{i,j}^{-}$, where 
		\begin{align}
		\mathrm{J}_{i,j}^{\star} := \lambda_{j}\left(v_{i,j+1/2}^{n\,\star}g_{i,j+1/2}^{\star} - v_{i,j-1/2}^{n\,\star}g_{i,j-1/2}^{\star} \right) - \lambda_{j}\left(v_{i-1,j+1/2}^{n\,\star}g_{i-1,j+1/2}^{\star} - v_{i-1,j-1/2}^{n\,\star}g_{i-1,j-1/2}^{\star} \right) 
		\end{align}
		with  $\star \in \{+,-\}$.		
		The numerical fluxes involved in $\mathrm{J}_{i,j}^{\,+}$ and $\mathrm{J}_{i,j}^{\,-}$ can be assigned with spatial locations as in~Figures~\ref{fig:v_flux_p} and~\ref{fig:v_flux_m}. A re--grouping of  $\mathrm{J}_{i,j}^{\star}/\lambda_{j}$ leads to 
		\begin{align}
		\allowdisplaybreaks
		\mathrm{J}_{i,j}^{\star} :={}& \lambda_{j}\left(v_{i,j+1/2}^{n\,\star}g_{i,j+1/2}^{\star} - v_{i-1,j+1/2}^{n\,\star}g_{i-1,j+1/2}^{\star}  \right) - \lambda_{j}\left(v_{i,j-1/2}^{n\,\star}g_{i,j-1/2}^{\star} - v_{i-1,j-1/2}^{n\,\star}g_{i-1,j-1/2}^{\star} \right) \\
		=:{}& \lambda_{j}\left( \mathfrak{T}_1^{\ast} +  \mathfrak{T}_2^{\ast} \right).
		\label{eqn:hor_vert}
		\end{align}
		We consider horizontal difference  $\mathfrak{T}_1^{+} = v_{i,j+1/2}^{n\,+}g(\ats{n}{i,j},\ats{n}{i,j+1}) - v_{i-1,j+1/2}^{n\,+}g(\ats{n}{i-1,j},\ats{n}{i-1,j+1})$ for clarity. Grouping the terms appropriately yields 
		\begin{multline}
		v_{i,j+1/2}^{n\,+}g(\ats{n}{i,j},\ats{n}{i,j+1}) - v_{i-1,j+1/2}^{n\,+}g(\ats{n}{i-1,j},\ats{n}{i-1,j+1}) = \left(v_{i,j+1/2}^{n\,+} - v_{i-1,j+1/2}^{n\,+}\right) g(\ats{n}{i,j},\ats{n}{i,j+1}) \\
		+ v_{i-1,j+1/2}^{n\,+}\left(g(\ats{n}{i,j},\ats{n}{i,j+1}) - g(\ats{n}{i-1,j},\ats{n}{i-1,j+1}) \right) := \mathrm{T}_1^{\,j+1/2} + \mathrm{T}_2^{\,j+1/2}.
		\label{eqn:flux-interm}
		\end{multline} 
Introduce an artificial nodal flux $g(\ats{n}{i-1,j},\ats{n}{i,j+1})$ arising from two diagonally opposite control volumes as in Figure~\ref{fig:nodal_flux}. 
		\begin{figure}[h!]
			\centering
			\includegraphics[scale=1]{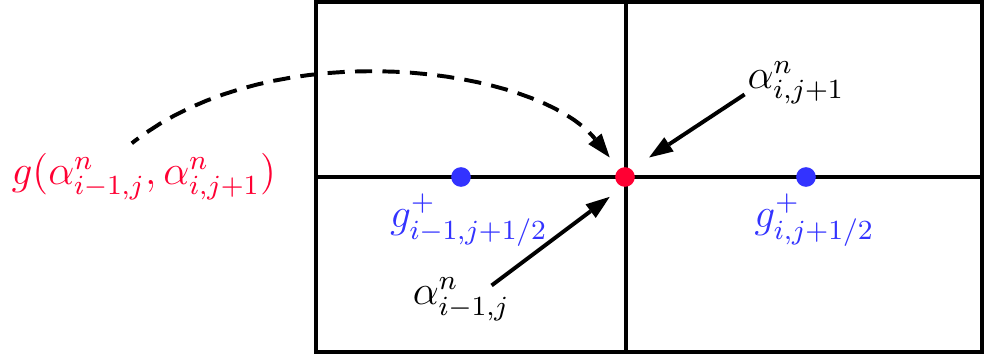}
			\caption{Intermediate nodal flux connecting the fluxes on edges.}
			\label{fig:nodal_flux}
		\end{figure}
		\begin{subequations}
			The nodal flux and some manipulations lead to 
			\begin{align}
			\mathrm{J}_{i,j}^{\star} ={}& \lambda_{j}\left( (v_{i,j+1/2}^{n\,\star} - v_{i-1,j+1/2}^{n\,\star}) g_{i,j+1/2}^{\star} - (v_{i,j-1/2}^{n\,\star} - v_{i-1,j-1/2}^{n\,\star}) g_{i,j-1/2}^{\star} \right)\\
			&+ \lambda_{j}\left[v_{i-1,j+1/2}^{n\,\star} \mathrm{E}_{\star}(\ats{n}{i,j},\ats{n}{i-1,j},\ats{n}{i,j+1}) \left( \ats{n}{i,j} - \ats{n}{i-1,j} \right)\right. \\
			&\hspace{3cm}- v_{i-1,j-1/2}^{n\,\star} \mathrm{E}_{\star}(\ats{n}{i,j-1},\ats{n}{i-1,j-1},\ats{n}{i,j}) \left(\ats{n}{i,j-1} - \ats{n}{i-1,j-1} \right)\\
			&+ v_{i-1,j+1/2}^{n\,\star} \mathrm{E}_{-\star}(\ats{n}{i,j+1},\ats{n}{i-1,j+1},\ats{n}{i-1,j})  \left(\ats{n}{i,j+1} - \ats{n}{i-1,j+1} \right) \\
			&\hspace{3cm}- \left. v_{i-1,j-1/2}^{n\,\star} \mathrm{E}_{-\star}(\ats{n}{i,j},\ats{n}{i-1,j},\ats{n}{i-1,j-1}) \left( \ats{n}{i,j} - \ats{n}{i-1,j} \right) \right],
			\label{eqn:j-split}
			\end{align}
			where the difference quotients $\mathrm{E}_{\ast} : \mathbb{R}^3 \rightarrow \mathbb{R}$ are defined by
			\begin{align}
			\hspace{-1cm}\mathrm{E}_\star(a,b,c) :={}&  \left\{ \begin{array}{c l}
			\dfrac{(1 + \star)(g(a,c) - g(b,c)) + (1 - \star)\left(g(c,a) - g(c,b)\right)}{2(a-b)} & \text{ if } a\not= b,\;\;\textrm{and}\\
			0 & \text{ if } a =  b.
			\end{array}
			\right.
			\label{eqn:E-diff}
			\end{align}
		\end{subequations}
		Note that the sums $(1 \pm (\pm))$ used in~\eqref{eqn:E-diff} are understood as $(1  \pm (\pm1))$.  Use the identity $a^{+} = a + a^{-}$ to transform the differences $(v_{i,j+1/2}^{n\,+} - v_{i-1,j+1/2}^{n\,+})$ and $(v_{i,j-1/2}^{n\,+} - v_{i-1,j-1/2}^{n\,+})$ in $\mathrm{J}_{i,j}^{+}$ and combine the resulting negative parts with the corresponding negative parts in $\mathrm{J}_{i,j}^{-}$. This yields
		\begin{align}
		&\left( \lambda_{j}(v_{i,j+1/2}^{n\,+} - v_{i-1,j+1/2}^{n\,+}) g_{i,j+1/2}^{+} - \lambda_{j}(v_{i,j-1/2}^{n\,+} - v_{i-1,j-1/2}^{n\,+}) g_{i,j-1/2}^{+} \right) \\
		&-\left( \lambda_{j}(v_{i,j+1/2}^{n\,-} - v_{i-1,j+1/2}^{n\,-}) g_{i,j+1/2}^{-} - \lambda_{j}(v_{i,j-1/2}^{n\,-} - v_{i-1,j-1/2}^{n\,-}) g_{i,j-1/2}^{-} \right)  \\
		={}&  \left.\left( \begin{array}{l}\dfrac{g_{i,j+1/2}^{+}}{h_j} \displaystyle\int_{t_n}^{t_{n+1}}\left( \displaystyle\fint_{x_{i-1/2}}^{x_{i+1/2}} v(t,s,y_{j+1/2})\,\mathrm{d}s - \displaystyle\fint_{x_{i-3/2}}^{x_{i-1/2}} v(t,s,y_{j+1/2})\,\mathrm{d}s\right)\,\mathrm{d}t \\
		- \dfrac{g_{i,j-1/2}^{+}}{h_j} \displaystyle\int_{t_n}^{t_{n+1}}\left( \displaystyle\fint_{x_{i-1/2}}^{x_{i+1/2}} v(t,s,y_{j-1/2})\,\mathrm{d}s - \displaystyle\fint_{x_{i-3/2}}^{x_{i-1/2}} v(t,s,y_{j-1/2})\,\mathrm{d}s\right)\,\mathrm{d}t \end{array} \right) \quad \right\} =: \mathrm{K}_{i,j}^g\\
		&+ \lambda_{j}(v_{i,j+1/2}^{n\,-} - v_{i-1,j+1/2}^{n\,-})(g_{i,j+1/2}^{+} - g_{i,j+1/2}^{-}) \\
		&- \lambda_{j}(v_{i,j-1/2}^{n\,-} 
		- v_{i-1,j-1/2}^{n\,-}) (g_{i,j-1/2}^{+} -g_{i,j-1/2}^{-}).
		\label{eqn:vert_interm}
		\end{align}
		\hyperref[ps.3]{\textbf{Step 3:}}
		\begin{subequations}
			Combine~\eqref{eqn:diff_h},~\eqref{eqn:space_var},~\eqref{eqn:j-split},~\eqref{eqn:vert_interm} and re-group the terms to obtain
				\begin{align} 
			\ats{n+1}{i,j} - \ats{n+1}{i-1,j} ={}& \left( \ats{n}{i,j} - \ats{n}{i-1,j} \right) (1 - c_{i,j}) - \mu_{i}\mathrm{M}_{i+1/2,j}^{\,x} \left( \ats{n}{i,j} - \ats{n}{i+1,j} \right) \\
			&- \mu_{i-1}\mathrm{M}_{i-3/2,j}^{\,x} \left( \ats{n}{i-2,j} - \ats{n}{i-1,j} \right) \\
			&+ \lambda_{j} \left[\sum_{\ast \in \{+,-\}} (\ast)v_{i-1,j-1/2}^{n\,\ast} \mathrm{E}_{\ast}(\ats{n}{i,j-1},\ats{n}{i-1,j-1},\ats{n}{i,j}) \left(\ats{n}{i,j-1} - \ats{n}{i-1,j-1} \right) \right.
	\\
	&- \left.\sum_{\ast \in \{+,-\}}(\ast)v_{i-1,j+1/2}^{n\,\ast} \mathrm{E}_{-\ast }(\ats{n}{i,j+1},\ats{n}{i-1,j+1},\ats{n}{i-1,j})  \left(\ats{n}{i,j+1} - \ats{n}{i-1,j+1} \right) \right]\\
			&- \lambda_{j} \left[(v_{i,j+1/2}^{n\,-} - v_{i-1,j+1/2}^{n\,-}) \left( g_{i,j+1/2}^{+} - g_{i,j+1/2}^{-} \right)\right. \\
			&+ \left.(v_{i,j-1/2}^{n\,-} - v_{i-1,j-1/2}^{n\,-}) \left(g_{i,j-1/2}^{+} -g_{i,j-1/2}^{-}\right) \right]- \left( \mathrm{K}_{i,j}^{f} + \mathrm{K}_{i,j}^{g} \right),
			\label{eqn:diff_i:i-1} \vspace{0.4cm} \\
			{}\text{where}\;\;c_{i,j} :={}& \mu_{i} \mathrm{M}_{i-1/2,j}^{\,x} + \mu_{i-1} \mathrm{M}_{i-1/2,j}^{\,x} +  \lambda_{j}\left[v_{i-1,j+1/2}^{n\,+} \mathrm{E}_{+}(\ats{n}{i,j},\ats{n}{i-1,j},\ats{n}{i,j+1}) \right.
			\\
			&- v_{i-1,j-1/2}^{n\,+} \mathrm{E}_{-}(\ats{n}{i,j},\ats{n}{i-1,j},\ats{n}{i-1,j-1}) 
			-  v_{i-1,j+1/2}^{n\,-} \mathrm{E}_{-}(\ats{n}{i,j},\ats{n}{i-1,j},\ats{n}{i,j+1}) 
			\\
			&+ \left. v_{i-1,j-1/2}^{n\,-} \mathrm{E}_{+}(\ats{n}{i,j},\ats{n}{i-1,j},\ats{n}{i-1,j-1})\right].
			\label{eqn:c-ij}
			\end{align}		
		\end{subequations}
\noindent Note that in~\eqref{eqn:c-ij} the terms $\mathrm{E}_{-}$ are nonpositive and $\mathrm{E}_{+}$ are nonnegative. This fact along with the CFL condition ensures that $1 - c_{i,j}$ is nonnegative. Take absolute value on both sides of~\eqref{eqn:diff_i:i-1}, multiply by $h_j$, sum on $i = 1,\ldots,I$ and $j = 0,\ldots,J$, and use the condition that $\boldsymbol{u}  = \boldsymbol{0}$ on $\partial \Omega$ to change the indices appropriately to obtain 
	\begin{align}
		\allowdisplaybreaks
		\sum_{j=0}^{J} h_{j} \sum_{i=1}^{I} |\ats{n+1}{i,j} - \ats{n+1}{i-1,j}| \le{}& \sum_{j=0}^{J} h_{j} \sum_{i=1}^{I} \left|\ats{n}{i,j} - \ats{n}{i-1,j}\right| (1 - c_{i,j}) \\
		&\hspace{-2cm}+ \sum_{j=0}^{J} h_{j} \left[\sum_{i=1}^{I}  \mu_{i-1}\mathrm{M}_{i-1/2,j}^{\,x} \left| \ats{n}{i,j} - \ats{n}{i-1,j} \right| + \ \sum_{i=1}^{I} \mu_{i}\mathrm{M}_{i-1/2,j}^{\,x} \left|\ats{n}{i,j} - \ats{n}{i-1,j} \right| \right]\\
		&\hspace{-2cm}+ \sum_{\ast \in \{+,-\}}\sum_{j=0}^{J} h_{j} \sum_{i=1}^{I} \left[\lambda_{j}v_{i-1,j+1/2}^{n\,\ast} (\ast)\mathrm{E}_{\ast}(\ats{n}{i,j},\ats{n}{i-1,j},\ats{n}{i,j+1}) \left|\ats{n}{i,j} - \ats{n}{i-1,j} \right| \right.\\
		&\hspace{0.5cm}+ \left. v_{i-1,j-1/2}^{n\,\ast} (-(\ast)\mathrm{E}_{-\ast}(\ats{n}{i,j},\ats{n}{i-1,j},\ats{n}{i-1,j-1}))  \left|\ats{n}{i,j} - \ats{n}{i-1,j} \right| \right]\\
		&\hspace{-2cm}+ \sum_{j=0}^{J-1} h_{j} \sum_{i=1}^{I} \lambda_{j}\left|v_{i,j+1/2}^{n-} - v_{i-1,j+1}^{n-}\right| \left| g(\ats{n}{i,j},\ats{n}{i,j+1}) - g(\ats{n}{i,j+1},\ats{n}{i,j})\right| \\
		&\hspace{-2cm}+ \sum_{j=1}^{J} h_{j} \sum_{i=1}^{I} \lambda_{j}\left|v_{i,j-1/2}^{n-} - v_{i-1,j-1/2}^{n-}\right|\;\left| g(\ats{n}{i,j-1},\ats{n}{i,j}) - g(\ats{n}{i,j},\ats{n}{i,j-1})\right| \\
		&\hspace{-2cm}+ \sum_{j=1}^{J} h_{j} \sum_{i=1}^{I} \left( \left|\mathrm{K}_{i,j}^{f} + \mathrm{K}_{i,j}^{g}\right| \right).
		\label{eqn:diff_i:i-1:abs}
		\end{align}	
		The term $1 - c_{i,j}$ and coefficients of $|\ats{n}{i,j} - \ats{n}{i-1,j}|$ in the second and third sum on the right hand side of~\eqref{eqn:diff_i:i-1:abs} adds up to one, and this yields
		\begin{multline}
		\sum_{j=0}^{J} h_{j} \sum_{i=1}^{I}\left|\ats{n+1}{i,j} - \ats{n+1}{i-1,j}\right| \le{} \sum_{j=0}^{J} h_{j} \sum_{i=1}^{I} \left|\ats{n}{i,j} - \ats{n}{i-1,j} \right| \\
		+ \sum_{j=0}^{J-1} h_{j} \sum_{i=1}^{I} \lambda_{j}\left|v_{i,j+1/2}^{n-} - v_{i-1,j+1/2}^{n-}\right|\;\left| g(\ats{n}{i,j},\ats{n}{i,j+1}) - g(\ats{n}{i,j+1},\ats{n}{i,j})\right| \\
		+ \sum_{j=1}^{J} h_{j} \sum_{i=1}^{I} \lambda_{j}\left|v_{i,j-1/2}^{n-} - v_{i-1,j-1/2}^{n-}\right|\; \left|g(\ats{n}{i,j-1},\ats{n}{i,j}) - g(\ats{n}{i,j},\ats{n}{i,j-1})\right| \\
		+ \delta \sum_{j=1}^{J} h_{j} \sum_{i=1}^{I} \left( \left|\mathrm{K}_{i,j}^{f} + \mathrm{K}_{i,j}^{g}\right| \right).\hspace{3cm}
		\label{eqn:var-alpha}
		\end{multline}
		Use the Lipschitz continuity of the negative part $a \rightarrow a^{-}$ (with constant 1) and $g$, Lipschitz continuity of $v$ in the $x$--direction, and grid regularity condition of Definition~\ref{defn:admis_grid} to obtain 
		\begin{align}
		&\hspace{-2cm}\lambda_{j}\left|v_{i,j-1/2}^{n\,-} - v_{i-1,j-1/2}^{n\,-}\right|\;\left|g(\ats{n}{i,j-1},\ats{n}{i,j}) - g(\ats{n}{i,j},\ats{n}{i,j-1})\right| \le \\
		&\widetilde{c}\,\left|\ats{n}{i,j} - \ats{n}{i,j-1}\right|  \mathrm{Lip}(g) \,\int_{t_n}^{t_{n+1}}\|\partial_x v(t,\cdot)\|_{L^\infty(\Omega)}\,\mathrm{d}t.
		\label{eqn:var-alpha-d_xu}
		\end{align}
		\hyperref[ps.4]{\textbf{Step 4:}} Apply~\ref{appen_id.a} on $\mathrm{K}_{i,j}^{g}$ (see~\eqref{eqn:vert_interm}) to obtain 
		\begin{align}
		\mathrm{K}_{i,j}^{g} ={}&   \dfrac{g_{i,j+1/2}^{+} - g_{i,j-1/2}^{+}}{2h_j}\left[ \int_{t_n}^{t_{n+1}}\left( \fint_{x_{i-1/2}}^{x_{i+1/2}} v(t,s,y_{j+1/2})\,\mathrm{d}s -  \fint_{x_{i-3/2}}^{x_{i-1/2}} v(t,s,y_{j+1/2})\,\mathrm{d}s\right)\,\mathrm{d}t\right. \\
		&+ \left. \int_{t_n}^{t_{n+1}}\left(\fint_{x_{i-1/2}}^{x_{i+1/2}} v(t,s,y_{j-1/2})\,\mathrm{d}s - \fint_{x_{i-3/2}}^{x_{i+1/2}} v(t,s,y_{j-1/2})\,\mathrm{d}s\,\right)\mathrm{d}t\right] \\
		&+ \dfrac{g_{i,j+1/2}^{+} + g_{i,j-1/2}^{+}}{2}\left[ \int_{t_n}^{t_{n+1}}\fint_{K_{i,j}} \partial_y v(t,\cdot)\,\mathrm{d}\boldsymbol{x}\,\mathrm{d}t - \int_{t_n}^{t_{n+1}}\fint_{K_{i-1,j}} \partial_y v(t,\cdot)\,\mathrm{d}\boldsymbol{x}\,\mathrm{d}t\right]  \\
		=:{}& \mathrm{K}_{i,j}^{g,1} + \mathrm{K}_{i,j}^{g,2}.
		\end{align}
		Write the term $\mathrm{K}_{i,j}^f$ (see~\eqref{eqn:kij_defn}) as
		\begin{align}
		\mathrm{K}_{i,j}^{f} ={}& \left(f(\ats{n}{i,j}) \int_{t_n}^{t_{n+1}} \fint_{K_{i,j}} \mathrm{div}(\boldsymbol{u})(t,\cdot)\,\mathrm{d}\boldsymbol{x}\,\mathrm{d}t - f(\ats{n}{i-1,j}) \int_{t_n}^{t_{n+1}} \fint_{K_{i-1,j}} \mathrm{div}(\boldsymbol{u})(t,\cdot)\,\mathrm{d}\boldsymbol{x}\,\mathrm{d}t \right) \\
		&- \left( f(\ats{n}{i,j}) \int_{t_n}^{t_{n+1}} \fint_{K_{i,j}} \partial_y v(t,\cdot)\,\mathrm{d}\boldsymbol{x}\,\mathrm{d}t -  f(\ats{n}{i-1,j}) \int_{t_n}^{t_{n+1}} \fint_{K_{i-1,j}} \partial_y v(t,\cdot)\,\mathrm{d}\boldsymbol{x}\,\mathrm{d}t \right) \\
		=:{}& \mathrm{K}_{i,j}^{f,1} + \mathrm{K}_{i,j}^{f,2}.
		\label{eqn:kif_rewrite}
		\end{align}
		Use the Lipschitz continuity of $g$, Lipschitz continuity of $v$ in the $x$--direction, and  Definition~\ref{defn:admis_grid} to obtain
		\begin{align}
		|\mathrm{K}_{i,j}^{g,1}| &\le \widetilde{c}\,\mathrm{Lip}(g) \left( |\ats{n}{i,j} - \ats{n}{i,j-1}| + |\ats{n}{i,j+1} - \ats{n}{i,j}| \right)\int_{t_n}^{t_{n+1}} ||\partial_x v(t,\cdot)||_{L^\infty(\Omega)}\,\mathrm{d}t.
		\label{eqn:est_kg1}
		\end{align}
		A use of~\ref{appen_id.a} on $\mathrm{K}_{i,j}^{f,1}$ yields
		\begin{align}
		\mathrm{K}_{i,j}^{f,1} ={}& \dfrac{f(\ats{n}{i,j}) - f(\ats{n}{i-1,j})}{2} \left[\int_{t_{n}}^{t_{n+1}} \fint_{K_{i,j}} \mathrm{div}(\boldsymbol{u})(t,\cdot)\,\mathrm{d}\boldsymbol{x}\,\mathrm{d}t + \int_{t_{n}}^{t_{n+1}} \fint_{K_{i-1,j}} \mathrm{div}(\boldsymbol{u})(t,\cdot)\,\mathrm{d}\boldsymbol{x}\,\mathrm{d}t\right]
		\\
	&+ \dfrac{f(\ats{n}{i,j}) + f(\ats{n}{i-1,j})}{2} \left[\int_{t_{n}}^{t_{n+1}} \fint_{K_{i,j}} \mathrm{div}(\boldsymbol{u})(t,\cdot)\,\mathrm{d}\boldsymbol{x}\,\mathrm{d}t - \int_{t_{n}}^{t_{n+1}} \fint_{K_{i-1,j}} \mathrm{div}(\boldsymbol{u})(t,\cdot)\,\mathrm{d}\boldsymbol{x}\,\mathrm{d}t\right],
		\label{eqn:sum-diff}
		\end{align} 
		Therefore, $|\mathrm{K}_{i,j}^{f,1}|$ can be bounded by 
		\begin{align}
		|\mathrm{K}_{i,j}^{f,1}| \le{}& \mathrm{Lip}(f) |\ats{n}{i,j} - \ats{n}{i-1,j}| \int_{t_n}^{t_{n+1}} ||\mathrm{div}(\boldsymbol{u})(t,\cdot)||_{L^\infty(\Omega)}\,\mathrm{d}t \\
		&\hspace{-0.5cm}+ \left( \mathrm{Lip}(f) \alpha_M + f_0 \right) \int_{t_n}^{t_{n+1}} \left|\fint_{K_{i,j}} \mathrm{div}(\boldsymbol{u})(t,\cdot)\,\mathrm{d}\boldsymbol{x} - \fint_{K_{i-1,j}} \mathrm{div}(\boldsymbol{u})(t,\cdot)\,\mathrm{d}\boldsymbol{x} \right|\,\mathrm{d}t.
		\label{eqn:est_kf1}
		\end{align}
		The sum $\mathrm{K}_{i,j}^{g,2} + \mathrm{K}_{i,j}^{f,2}$ can be written as
		\begin{align}
		\mathrm{K}_{i,j}^{g,2} + \mathrm{K}_{i,j}^{f,2} ={}&  \dfrac{-2 f(\ats{n}{i,j}) + g_{i,j+1/2}^{+} + g_{i,j-1/2}^{+} }{2} \int_{t_n}^{t_{n+1}} \fint_{K_{i,j}} \partial_y v(t,\cdot)\,\mathrm{d}\boldsymbol{x}\,\mathrm{d}t \\
		&+ \dfrac{2f(\ats{n}{i-1,j}) - g_{i,j+1/2}^{+} - g_{i,j-1/2}^{+}}{2} \int_{t_n}^{t_{n+1}} \fint_{K_{i-1,j}} \partial_y v(t,\cdot)\,\mathrm{d}\boldsymbol{x}\,\mathrm{d}t. 
		\label{eqn:est_kf2_kg2_i}
		\end{align}
		The Lipschitz continuity of $g$ and $f$ and  $g(a,a) = f(a)$ yield
		\begin{subequations}
			\begin{align}
			\label{eqn:est_kf2_kg2_ii}
			\hspace{-2cm}|-2 f(\ats{n}{i,j}) + g_{i,j+1/2}^{+} + g_{i,j-1/2}^{+}| \le{}& \mathrm{Lip}(g)|\ats{n}{i,j} - \ats{n}{i,j-1}| + \mathrm{Lip}(g)|\ats{n}{i,j} - \ats{n}{i,j+1}|,  \hspace{0.5cm}\\
			|2 f(\ats{n}{i-1,j}) + g_{i,j+1/2}^{+} + g_{i,j-1/2}^{+}| \le{}& 2 \mathrm{Lip}(f)|\ats{n}{i,j} - \ats{n}{i-1,j}| \nonumber \\
			\label{eqn:est_kf2_kg2_iii}
			&+  \mathrm{Lip}(g)|\ats{n}{i,j} - \ats{n}{i,j-1}| + \mathrm{Lip}(g)|\ats{n}{i,j} - \ats{n}{i,j+1}|.
			\end{align} 
		\end{subequations}
		Combine the bounds~\eqref{eqn:est_kg1},~\eqref{eqn:est_kf1},~\eqref{eqn:est_kf2_kg2_i},~~\eqref{eqn:est_kf2_kg2_ii}, and ~\eqref{eqn:est_kf2_kg2_iii} to obtain
		\begin{multline}
		|\mathrm{K}_{i,j}^f + \mathrm{K}_{i,j}^g| \le{} |\mathrm{K}_{i,j}^{f,1}| + |\mathrm{K}_{i,j}^{g,1}| + |\mathrm{K}_{i,j}^{f,2} + \mathrm{K}_{i,j}^{g,2}| \\
		\le{} \mathrm{Lip}(f)|\ats{n}{i,j} - \ats{n}{i-1,j}| \left( \int_{t_n}^{t_{n+1}} ||\mathrm{div}(\boldsymbol{u})(t,\cdot)||_{L^\infty(\Omega)}\,\mathrm{d}t + 2\int_{t_n}^{t_{n+1}} ||\partial_y v(t,\cdot)||_{L^\infty(\Omega)}\,\mathrm{d}t \right) \\
		+ \mathrm{Lip}(g) \left( |\ats{n}{i,j} - \ats{n}{i,j-1}| + |\ats{n}{i,j+1} - \ats{n}{i,j}| \right)\left( \widetilde{c} \int_{t_n}^{t_{n+1}} \left( ||\partial_x v(t,\cdot)||_{L^\infty(\Omega)}+ 2||\partial_y v(t,\cdot)||_{L^\infty(\Omega)} \right)\,\mathrm{d}t \right) \\
		+ \left( \mathrm{Lip}(f) \alpha_M + f_0 \right) \int_{t_n}^{t_{n+1}} \left|\fint_{K_{i,j}} \mathrm{div}(\boldsymbol{u})(t,\cdot)\,\mathrm{d}\boldsymbol{x} - \fint_{K_{i-1,j}} \mathrm{div}(\boldsymbol{u})(t,\cdot)\,\mathrm{d}\boldsymbol{x} \right|\,\mathrm{d}t.
		\label{eqn:kfg}
		\end{multline}
		\hyperref[ps.5]{\textbf{Step 5:}} Use~\eqref{eqn:var-alpha},~\eqref{eqn:var-alpha-d_xu}, and~\eqref{eqn:kfg} to obtain
		\begin{multline}
		|\alpha_{h,\delta}(t_{n+1},\cdot)|_{L^1_yBV_x} \le{} |\alpha_{h,\delta}(t_{n},\cdot)|_{L^1_yBV_x} 
		+  4\left(\widetilde{c} + 1\right)\,\mathrm{Lip}(g) |\alpha_{h,\delta}(t_{n},\cdot)|_{L^1_xBV_y} \int_{t_{n}}^{t_{n+1}}||\nabla \boldsymbol{u}||_{L^\infty(\Omega)}\,\mathrm{d}t  \\
		+ 3\mathrm{Lip}(f) |\alpha_{h,\delta}(t_{n},\cdot)|_{L^1_yBV_x}\int_{t_{n}}^{t_{n+1}}||\nabla \boldsymbol{u}||_{L^\infty(\Omega)}\,\mathrm{d}t \\
		+ \left( \mathrm{Lip}(f) \alpha_M + f_0 \right) \int_{t_{n}}^{t_{n+1}} |\Pi_{h}^0(\mathrm{div}(\boldsymbol{u}))(t,\cdot)|_{L^1_yBV_x}\,\mathrm{d}t,
		\label{eqn:bvx_l1y}
		\end{multline} 
		where the piecewise constant projection $\Pi_{h}^0 : BV_{\boldsymbol{x}}(\Omega) \rightarrow BV_{\boldsymbol{x}}(\Omega)$ for an admissible grid $\mathrm{X}_k \times \mathrm{Y}_h$ is defined by, for $\beta \in BV_{\boldsymbol{x}}(\Omega)$, $\left( \Pi_{h}^0(\beta)\right)(\boldsymbol{x}) := \fint_{K_{i,j}} \beta\,\mathrm{d}\boldsymbol{x}\;\;\forall\,\boldsymbol{x} \in K_{i,j}.$
		\begin{subequations}
			A similar argument can be obtained with $i$ and $j$ interchanged and when combined with~\eqref{eqn:bvx_l1y} yields
			\begin{align}
			|\alpha_{h,\delta}(t_{n+1},\cdot)|_{BV_{x,y}} \le{}& |\alpha_{h,\delta}(t_{n},\cdot)|_{BV_{x,y}} \left(1 + \mathscr{C}  \int_{t_n}^{t_{n+1}}||\nabla\boldsymbol{u}(t,\cdot)||_{L^\infty(\Omega)}\,\mathrm{d}t \right) \\
			&+ \mathscr{C} \int_{t_{n}}^{t_{n+1}}|\Pi_{h}^0(\mathrm{div}(\boldsymbol{u}))|_{BV_{x,y}}\,\mathrm{d}t,
			\label{eqn:alpha_sp_bv}
			\end{align}
			where $\mathscr{C} = \max\left(\mathrm{Lip}(f)\alpha_M + f_0,3\mathrm{Lip}(f) +  4\mathrm{Lip}(g)(\widetilde{c} + 1) + 1  \right)$. Apply induction on~\eqref{eqn:alpha_sp_bv} with $n$ as the index and use the fact that $|\Pi_{h}^0(\mathrm{div}(\boldsymbol{u}))|_{BV_{x,y}} \le |\mathrm{div}(\boldsymbol{u})|_{BV_{x,y}}$ to obtain 
			\begin{align}
			|\alpha_{h,\delta}(t_{n},\cdot)|_{BV_{x,y}} &{}\le  \mathrm{B}_{\boldsymbol{u}} \left( |\alpha_{h,\delta}(t_{0},\cdot)|_{BV_{x,y}} + \mathscr{C}\int_{0}^{T}|\mathrm{div}(\boldsymbol{u})|_{BV_{x,y}}\,\mathrm{d}t  \right).
			\label{eqn:bv_sum}
			\end{align}
			The desired conclusion follows from~\eqref{eqn:bv_sum} and~\eqref{eqn:alpha_ini}.
		\end{subequations}
	\end{proof}
	\begin{figure}[h!]
	\centering
	\begin{subfigure}[b]{0.49\textwidth}
		\centering
		\includegraphics[scale=1]{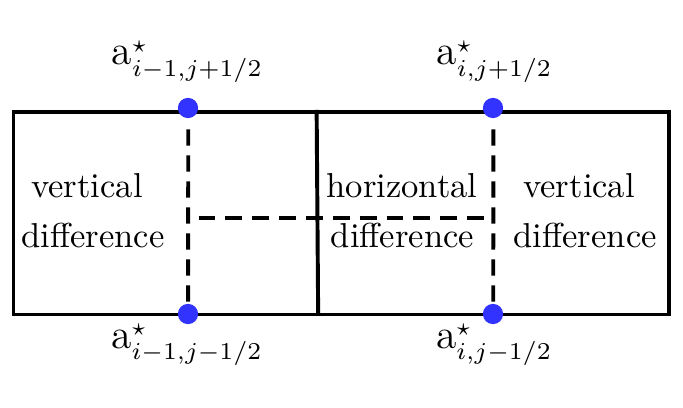}
		\caption{}
		\label{fig:vh_diff}
	\end{subfigure}
	\begin{subfigure}[b]{0.49\textwidth}
		\centering
		\includegraphics[scale=1]{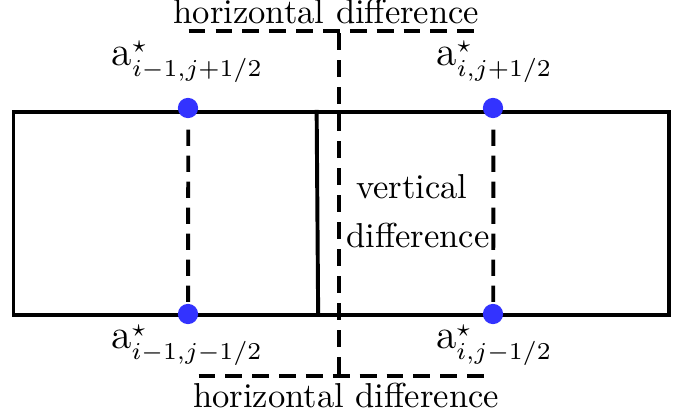}
		\caption{}
		\label{fig:hv_diff}
	\end{subfigure}
	\caption{Differences between horizontal and vertical levels. Here, $a_{i,j+1/2}^{\ast} = v_{i,j+1/2}^{n\,\ast}g_{i,j+1/2}^{\ast}$, where $g_{i,j+1/2}^{+} = g(\alpha_{i,j}^n,\alpha_{i,j+1}^n)$ and  $g_{i,j+1/2}^{-} = g(\alpha_{i,j+1}^n,\alpha_{i,j}^n)$.}
\end{figure}
\begin{remark}[Regrouping of $\mathrm{J}_{i,j}^{\ast}$ in~\eqref{eqn:hor_vert}]
Observe that  $\mathrm{J}_{i,j}^{\star}/\lambda_{j}$ is the horizontal variation between differences across two vertical levels as in Figure~\ref{fig:vh_diff}. However, this form does not yield any terms like $\ats{n}{i,r} - \ats{n}{p,r}$, where $p \in \{i + 1,i-1\}$ and $r \in \{j + 1,j,j-1\}$, and thereby annihilates any chance of expressing  $\ats{n+1}{i,j} - \ats{n+1}{i-1,j}$ as a linear combination of such terms, which is crucial in controlling the growth of spatial variation over time. 
	This problem can be fixed by considering the terms $\mathrm{J}_{i,j}^{+}$ and $\mathrm{J}_{i,j}^{-}$ as vertical variations between differences across two horizontal levels, see~\eqref{eqn:hor_vert}, as in Figure~\ref{fig:hv_diff}.
\end{remark}
	
	\begin{proposition}[temporal variation]
		\label{prop:temp_bv}
		The function $\alpha_{h,\delta}$ satisfies
		\begin{align*}
		|\alpha_{h,\delta}|_{L^1_{x,y}BV_t} \le{}& 4\mathrm{B}_{\boldsymbol{u}} \left( |\alpha_0|_{BV_{x,y}} + \mathscr{C}|\mathrm{div}(\boldsymbol{u})|_{L^1_tBV_{x,y}}  \right) \mathrm{Lip}(g) ||\nabla \boldsymbol{u}||_{L^1_tL^\infty(\Omega_T)} \\
		&+  (\mathrm{Lip}(f)\alpha_M + f_0)|\mathrm{div}(\boldsymbol{u})|_{L^1(\Omega_T)}.
		\end{align*}
	\end{proposition}
	\noindent The proof of Proposition~\ref{prop:temp_bv} is obtained by writing $\ats{n+1}{i,j} - \ats{n}{i,j}$ in terms of the differences $\ats{n}{i,j} - \ats{n}{l,m}$, where $(l,m) \in \left\{(i,j-1),(i,j+1),(i-1,j),(i+1,j) \right\}$ and by applying Proposition~\ref{prop:space_bv}. 
	\begin{proof}
		Use~\eqref{eqn:conv_comb} to write
		\begin{align}
		\ats{n+1}{i,j} -  \ats{n}{i,j} ={}& \mu_{i} \mathrm{M}_{i+1/2,j}^{\,x} \left(\ats{n}{i+1,j} - \ats{n}{i,j} \right) + \lambda_{j} \mathrm{M}_{i,j+1/2}^{\,y} \left(\ats{n}{i,j+1}-\ats{n}{i,j}\right) \\
		&+ \mu_{i} \mathrm{M}_{i-1/2,j}^{\,x} \left(\ats{n}{i-1,j}-\ats{n}{i,j}\right) + \lambda_{j} \mathrm{M}_{i,j-1/2}^{\,y}
		\left(\ats{n}{i,j-1}  - \ats{n}{i,j}\right) \\
		&- f(\ats{n}{i,j}) \left(\int_{t_{n}}^{t_{n+1}}\fint_{K_{i,j}} \mathrm{div}(\boldsymbol{u})(t,\boldsymbol{x})\,\mathrm{d}\boldsymbol{x}\,\mathrm{d}t \right).
		\label{eqn:temp-var:1}
		\end{align}
	 Multiply both sides of~\eqref{eqn:temp-var:1} by $h_j k_i$, sum over $n = 0,\ldots,N$, $i=0,\ldots,I$ and $j = 0,\ldots,J$, and use the homogeneous boundary condition on $\boldsymbol{u}$ to obtain 
	 \begin{small}
		\begin{multline}
		\sum_{j = 0}^{J}\sum_{i=0}^{I} h_j k_i \sum_{n=0}^{N} |\ats{n+1}{i,j} -  \ats{n}{i,j}| \le \sum_{n=0}^{N} \delta  \sum_{j = 0}^{J}h_j\sum_{i=0}^{I-1} \mathrm{M}_{i+1/2,j}^{\,x} |\ats{n}{i+1,j} - \ats{n}{i,j}| \\
		+ \sum_{n=0}^{N} \delta  \sum_{i = 0}^{I}k_i\sum_{j=0}^{J-1} \mathrm{M}_{i,j+1/2}^{\,y} |\ats{n}{i,j+1} - \ats{n}{i,j}| + \sum_{n=0}^{N} \delta  \sum_{j = 0}^{J}h_j\sum_{i=1}^{I} \mathrm{M}_{i-1/2,j}^{\,x} |\ats{n}{i-1,j} - \ats{n}{i,j}| \\
		+ \sum_{n=0}^{N} \delta  \sum_{i = 0}^{I}k_i\sum_{j=1}^{J} \mathrm{M}_{i,j-1/2}^{\,y} |\ats{n}{i,j-1} - \ats{n}{i,j}|
		+ \sum_{j = 0}^{J}\sum_{i=0}^{I}  \sum_{n=0}^{N} f(\ats{n}{i,j}) \left(\int_{t_{n}}^{t_{n+1}}\int_{K_{i,j}} \mathrm{div}(\boldsymbol{u})(t,\boldsymbol{x})\,\mathrm{d}\boldsymbol{x}\,\mathrm{d}t \right).
		\label{eqn:temp-var:2}
		\end{multline}
	\end{small}
		Use the Lipschitz continuity of the functions $f$ and $g$ and~\eqref{eqn:temp-var:2} to obtain
		\begin{align}
		\int_{\Omega} |\alpha_{h,\delta}(\cdot,x,y)|_{BV_t(0,T)}\,\mathrm{d}x\,\mathrm{d}y \le{}& 4\mathrm{Lip}(g)\int_{0}^T ||\nabla\boldsymbol{u}(t,\cdot)||_{L^\infty(\Omega)}|\alpha_{h,\delta}(t,\cdot)|_{BV_{x,y}} \,\mathrm{d}t \\
		&+ (\mathrm{Lip}(f)\alpha_M + f_0)||\mathrm{div}(\boldsymbol{u})||_{L^1(\Omega_T)}.
		\label{eqn:temp-var:3}
		\end{align}
		Use~\eqref{eqn:temp-var:3} and Proposition~\ref{prop:space_bv} to arrive at the desired result.
	\end{proof}

The result~\eqref{eqn:alpha-bv} in Theorem~\ref{thm:bv-alpha} follows from  Proposition~\ref{prop:space_bv},  Proposition~\ref{prop:temp_bv} and~\eqref{eqn:ts-bv-norm}. The homogeneous source term in~\eqref{eqn:cons_law} can be replaced with a function $\mathfrak{S}(t,\boldsymbol{x},\alpha)$ that satisfies the assumption:
	\begin{enumerate}[label= $\mathrm{(AS.4)}$,ref=$\mathrm{(AS.4)}$,leftmargin=\widthof{(AS.4)}+3\labelsep]
		\item\label{sa} $\mathfrak{S} \in L^1_tL^\infty(\Omega_T)$ and $\mathfrak{S}(t,\boldsymbol{x},z)$ is Lipschitz continuous with respect to $z$ (with constant $\mathrm{Lip}_z(\mathfrak{S})$), uniformly with respect to $t$ and $\boldsymbol{x}$, and is Lipschitz continuous with respect to $\boldsymbol{x}$ (with constant $\mathrm{Lip}_{\boldsymbol{x}}(\mathfrak{S})$), uniformly with respect to $t$ and $z$.
	\end{enumerate} 
	In this case, we obtain the following corollary to Theorem~\ref{thm:bv-alpha}.
	\begin{corollary}
		\label{cor:source}
		Let~\ref{as.1}--~\ref{sa} and the Courant--Friedrichs--Lewy (CFL) condition $4 \delta \max_{i,j}( \frac{1}{k_{i}} + \frac{1}{h_j} ) \mathrm{Lip}(g) ||\boldsymbol{u}||_{L^\infty(\Omega_T)} \le 1$	hold. If $\alpha_0 \in L^\infty(\Omega)\cap BV_{\boldsymbol{x}}(\Omega)$ hen, the time--reconstruct $\alpha_{h,\delta} : \Omega_T \rightarrow \mathbb{R}$ reconstructed from the values $\alpha_{i,j}^n$ obtained from the scheme 
		\begin{align}
		\ats{n+1}{i,j} = \ats{n}{i,j} - \mu_{i}(\mathrm{F}_{i+1/2,j} - \mathrm{F}_{i-1/2,j}) - \lambda_{j} (\mathrm{G}_{i,j+1/2} - \mathrm{G}_{i,j-1/2}) + \int_{t_n}^{t_{n+1}} \fint_{K_{i,j}} \mathfrak{S}(t,\boldsymbol{x},\ats{n}{i,j})\,\mathrm{d}t\,\mathrm{d}\boldsymbol{x}
		\end{align}
		satisfies $|\alpha_{h,\delta}|_{BV_{x,y,t}} \le \mathscr{C}_{\mathrm{BV}}$,
		where $\mathscr{C}_{\mathrm{BV}}$ depends on $T$, $\alpha_0$, $f$, $g$, $||\nabla\boldsymbol{u}||_{L^1_tL^\infty(\Omega_T)}$,  $|\mathrm{div}(\boldsymbol{u})|_{L^1_tBV_{x,y}}$, $\mathrm{Lip}_{\boldsymbol{x}}(\mathfrak{S})$, $\mathrm{Lip}_{z}(\mathfrak{S})$, and $|\mathfrak{S}|_{L^1_tBV_{x,y}}$.
	\end{corollary}
	\begin{proof}
		It is enough to estimate variation of the source term in the $x$ direction, which can be written as
		\begin{align}
		\mathrm{V}_{i,j} :=  \int_{t_n}^{t_{n+1}} \fint_{K_{i,j}} \mathfrak{S}(t,\boldsymbol{x},\ats{n}{i,j})\,\mathrm{d}t\,\mathrm{d}\boldsymbol{x} -  \int_{t_n}^{t_{n+1}} \fint_{K_{i-1,j}} \mathfrak{S}(t,\boldsymbol{x},\ats{n}{i-1,j})\,\mathrm{d}t\,\mathrm{d}\boldsymbol{x}.
		\label{eqn:vs1}
		\end{align}
		Add and subtract $ \int_{t_n}^{t_{n+1}} \fint_{K_{i,j}} \mathfrak{S}(t,\boldsymbol{x},\ats{n}{i-1,j})\,\mathrm{d}t\,\mathrm{d}\boldsymbol{x}$ to~\eqref{eqn:vs1} and group the terms appropriately to obtain 
		\begin{align}
		|\mathrm{V}_{i,j}| \le{}& \int_{t_n}^{t_{n+1}} \fint_{K_{i,j}} \left| \mathfrak{S}(t,\boldsymbol{x},\ats{n}{i,j}) - \mathfrak{S}(t,\boldsymbol{x},\ats{n}{i-1,j}) \right| \,\mathrm{d}t\,\mathrm{d}\boldsymbol{x} \\
		&\hspace{-1cm}+ \int_{t_n}^{t_{n+1}} \left| \fint_{K_{i,j}} \mathfrak{S}(t,\boldsymbol{x},\ats{n}{i-1,j})\,\mathrm{d}\boldsymbol{x} - \fint_{K_{i-1,j}} \mathfrak{S}(t,\boldsymbol{x},\ats{n}{i-1,j})\,\mathrm{d}\boldsymbol{x} \right|\,\mathrm{d}t =: \mathrm{V}_1 + \mathrm{V}_2.
		\label{eqn:vij}
		\end{align}
		Use the Lipschitz continuity of $\mathfrak{S}$ with respect to the third argument to bound  $\mathrm{V}_1$ by $\delta\, \mathrm{Lip}_z(\mathfrak{S})|\ats{n}{i,j} - \ats{n}{i-1,j}|$. Sum~\eqref{eqn:vij} for $i = 1,\ldots,I$ to obtain
		\begin{align}
		\sum_{i=1}^I |V_{i,j}| \le \delta\, \mathrm{Lip}_z(\mathfrak{S}) |\alpha_{h,\delta}(t_{n},\cdot)|_{BV_x} + \mathrm{Lip}_{\boldsymbol{x}}(\mathfrak{S})\int_{t_{n}}^{t_{n+1}} \left|\Pi_h^0(\mathfrak{S}) \right|_{BV_{x}}\,\mathrm{d}t.
		\label{eqn:v_addn}
		\end{align}
		Rest of the proof follows by adding the terms in the right hand side of~\eqref{eqn:v_addn} to the right hand side of \eqref{eqn:bvx_l1y} and by following the steps from there on.
	\end{proof}

	\section{$BV$ estimate for conservation laws with fully nonlinear flux}
	\label{subsec:claire}
	Theorem~\ref{thm:bv-alpha} can be extended to the case with fully  nonlinear flux such as
	\begin{align}
	\label{eqn:fully_non}
	\left. 
	\begin{array}{r l}
	\partial_t \alpha + \mathrm{div}(\boldsymbol{F}(t,\boldsymbol{x},\alpha)) &={} 0\;\;\mathrm{ in }\; \Omega_T\;\;\textmd{and} \\
	\alpha(0,\cdot) &= \alpha_0\;\;\mathrm{ in }\;\; \Omega. 
	\end{array}
	\right\}
	\end{align}
	The strong $BV$ estimate on finite volume schemes for~\eqref{eqn:fully_non} on square Cartesian grids is obtained by C. Chainais-Hilairet~\cite{claire_1999} under the assumption that $\mathrm{div}_{\boldsymbol{x}}(\boldsymbol{F}) = 0$. In this article, we relax this condition and obtain bounded variation estimates for $\alpha$ under the following assumptions.
	\begin{enumerate}[label= $\mathrm{(AS.\arabic*)}$,ref=$\mathrm{(AS.\arabic*)}$,leftmargin=\widthof{(AS.4)}+3\labelsep]
		\setcounter{enumi}{4}
		\item\label{a.1} $\boldsymbol{F}(t,\boldsymbol{x},z)$ is $\mathscr{C}^1(\Omega_T \times \mathbb{R})$ and is Lipschitz continuous with respect to $z$ (with constant $\mathrm{Lip}(\boldsymbol{F})$), uniformly with respect to $(t,\boldsymbol{x})$, and $\partial_z \boldsymbol{F}$ is Lipschitz continuous with respect to $\boldsymbol{x}$ (with constant $\mathrm{Lip}(\partial_s \boldsymbol{F})$), uniformly with respect to $t$ and $z$,
		\item\label{a.3} $|\mathrm{div}_{\boldsymbol{x}}(\boldsymbol{F})|_{L^1_tBV_{x,y}}< \infty$ and $\mathrm{div}_{\boldsymbol{x}}(\boldsymbol{F})$ is Lipschitz continuous with respect to $z$ (with constant constant $\mathrm{Lip}(\mathrm{div}_{\boldsymbol{x}}(\boldsymbol{F}))$), uniformly with respect to $t$ and $\boldsymbol{x}$.
	\end{enumerate}
Observe that assumption $\mathrm{div}_{\boldsymbol{x}}(\boldsymbol{F}) = 0$ manifests as $\mathrm{div}(\boldsymbol{u}) = 0$ in~\eqref{eqn:cons_law}, where $\boldsymbol{F}(t,\boldsymbol{x},\alpha)$ is same as $\boldsymbol{u}(t,\boldsymbol{x})f(\alpha)$. 
	Use~\ref{a.1} to write the flux $\boldsymbol{F}$ as $\boldsymbol{F} := (F_1,F_2)$, $F_1 = a + b$, and $F_2 = c + d$, where $a$ and $c$ are monotonically nondecreasing and $b$ and $d$ are monotonically nonincreasing in $z$, uniformly with respect to $t$ and $\boldsymbol{x}$. In this case, we can set the following finite volume scheme on an admissible grid $\mathrm{X}_h \times \mathrm{Y}_k$: 
	\begin{align}
	\alpha_{i,j}^{n+1} ={} \alpha_{i,j}^n -& \dfrac{1}{k_i} \left(a_{i+1/2,j}^n(\alpha_{i,j}^n) - a_{i-1/2,j}^n(\alpha_{i-1,j}^n) + b_{i+1/2,j}^n(\alpha_{i+1,j}^n) - b_{i-1/2,j}^n(\alpha_{i,j}^n)\right) \\
	-& \dfrac{1}{h_j} \left(c_{i,j+1/2}^n(\alpha_{i,j}^n) - c_{i,j-1/2}^n(\alpha_{i,j-1}^n) + d_{i,j+1/2}^n(\alpha_{i,j+1}^n) - d_{i,j-1/2}^n(\alpha_{i,j}^n)\right)
	\label{eqn:scheme_2}
	\end{align}
	with the initial condition~\ref{eqn:alpha_ini}, where the numerical fluxes are defined, for $\gamma \in \{a,b\}$, and $\varrho \in \{c,d\}$, by
	\begin{align}
	\gamma^n_{i+1/2,j}(s) ={}& \int_{t_n}^{t_{n+1}} \fint_{y_{j-1/2}}^{y_{j+1/2}} \gamma(t,x_{i+1/2},y,s)\mathrm{d}y\,\mathrm{d}t \;\;\text{ and } \\
	\varrho^n_{i,j+1/2}(s) ={}& \int_{t_n}^{t_{n+1}} \fint_{x_{i-1/2}}^{x_{i+1/2}} \varrho(t,x,y_{j+1/2},s)\mathrm{d}x\,\mathrm{d}t.
	\end{align}
	\begin{subequations}
		\begin{theorem}[bounded variation for fully nonlinear flux]
			\label{thm:non-lin:flux}
			Let the assumptions~\ref{a.1}--\ref{a.3} and the following CFL condition hold:
			$4 \delta\,\mathrm{Lip}(\boldsymbol{F}) \max_{i,j}(\frac{1}{k_i} + \frac{1}{h_j} ) \le 1.$
			Then the piecewise time--reconstruct $\alpha_{h,\delta} : \Omega_T \rightarrow \mathbb{R}$ re--constructed from the values $\alpha_{i,j}^n$ obtained from the scheme~\eqref{eqn:scheme_2} satisfies 
			$|\alpha_{h,\delta}|_{BV_{x,y,t}(\Omega_T)} \le \mathscr{C},$
			where $\mathscr{C}$ depends on $T$, $\alpha_0$, $|\mathrm{div}_{\boldsymbol{x}}(\boldsymbol{F})|_{L^1_tBV_{x,y}}$, and $\mathrm{Lip}(\mathrm{div}_{\boldsymbol{x}}(\boldsymbol{F}))$.
		\end{theorem}
	\end{subequations}
	\noindent The proof of Theorem~\ref{thm:non-lin:flux} is based on two main ideas. Firstly, the terms in the scheme~\eqref{eqn:scheme_2} are re--arranged and grouped appropriately so that the term  $\int_{t_n}^{t_{n}} \fint_{K_{i,j}} \mathrm{div}_{\boldsymbol{x}}(\boldsymbol{F})(t,\boldsymbol{x},\ats{n}{i,j})\,\mathrm{d}\boldsymbol{x}\mathrm{d}t$ can be separately estimated (see~\eqref{eqn:sch_2_mod}). Secondly, we employ the Lipschitz continuity of $\mathrm{div}_{\boldsymbol{x}}(\boldsymbol{F})$ to bound difference of the terms $\{\int_{t_n}^{t_{n+1}} \fint_{K_{l,j}}  \mathrm{div}_{\boldsymbol{x}}(\boldsymbol{F})(t,\boldsymbol{x},\ats{n+1}{l,j})\,\mathrm{d}\boldsymbol{x}\mathrm{d}t\;:\; l = i,i+1 \}$ by the $BV$ seminorms $\int_{t_n}^{t_{n+1}} |\alpha_{h,\delta}(t,\cdot)|_{BV_{x}}\,\mathrm{d}t$ and $\int_{t_n}^{t_{n+1}}|\mathrm{div}_{\boldsymbol{x}}(\boldsymbol{F})(t,\cdot,\cdot)|_{BV_{x}}\,\mathrm{d}t$.
	\begin{proof}
		Note that the scheme~\eqref{eqn:scheme_2} can be expressed as
	\begin{align}
		\ats{n+1}{i,j} ={}&  \left(\begin{array}{l}
		\ats{n}{i,j} - \Delta_{i,j}^{1,n}(\ats{n}{i,j},\ats{n}{i-1,j})(\ats{n}{i,j} - \ats{n}{i-1,j}) -  \Delta_{i,j}^{2,n}(\ats{n}{i,j},\ats{n}{i+1,j})(\ats{n}{i,j} - \ats{n}{i+1,j})\nonumber  \vspace{0.2cm}\\
		- \dfrac{1}{h_{j}} \left(c_{i,j-1/2}(\ats{n}{i,j}) - c_{i,j-1/2}(\ats{n}{i,j-1}) + d_{i,j+1/2}(\ats{n}{i,j+1}) - d_{i,j+1/2}(\ats{n}{i,j})\right) \end{array} \right) \quad  \\
		\label{eqn:sch_2_mod}
		&-  \left(\begin{array}{l} \displaystyle \int_{t_n}^{t_{n+1}} \fint_{K_{i,j}} \mathrm{div}_{\boldsymbol{x}}(\boldsymbol{F})(t,\boldsymbol{x},\ats{n}{i,j})\,\mathrm{d}\boldsymbol{x}\mathrm{d}t, \end{array} \right) =: \mathrm{T}_{1,i} - \mathrm{T}_{2,i}
		\end{align}
		where 
		\begin{align}
		\Delta_{i,j}^{1,n}(p,q)  = \dfrac{a_{i-1/2,j}^n(p) - a_{i-1/2,j}^n(q)}{p - q}\;\;\mathrm{ and }\;\;
		\Delta_{i,j}^{2,n}(p,q)  = \dfrac{b_{i+1/2,j}^n(p) - b_{i+1/2,j}^n(q)}{q - p}.
		\end{align}
			It is enough to estimate $|\alpha_{h,\delta}|_{L^1_yBV_x}$ as we did in the proof of Proposition~\ref{prop:space_bv}. Take the difference between the scheme~\eqref{eqn:sch_2_mod} written for $\ats{n+1}{i+1,j}$ and $\ats{n+1}{i,j}$. The difference $\mathrm{T}_{1,i+1} - \mathrm{T}_{1,i}$ can be estimated exactly as in the proof of~\cite[Lemma 8]{claire_1999}, wherein the CFL condition in Theorem~\ref{thm:non-lin:flux} enables us to express $\alpha_{i,j}^{n+1} - \alpha_{i-1,j}^{n+1}$ as a convex linear combination of differences at the previous time step $n$. Consider the difference $|\mathrm{T}_{2,i+1} - \mathrm{T}_{2,i}|$:
		\begin{multline}
		|\mathrm{T}_{2,i+1} - \mathrm{T}_{2,i}| \le{} \int_{t_{n}}^{t_{n+1}} \left| \fint_{K_{i+1,j}} \mathrm{div}_{\boldsymbol{x}}(\boldsymbol{F})(t,\boldsymbol{x},\ats{n}{i+1,j})\,\mathrm{d}\boldsymbol{x} - \fint_{K_{i,j}} \mathrm{div}_{\boldsymbol{x}}(\boldsymbol{F})(t,\boldsymbol{x},\ats{n}{i+1,j})\,\mathrm{d}\boldsymbol{x}  \right|\,\mathrm{d}t \\
		+ \int_{t_n}^{t_{n+1}} \left| \fint_{K_{i,j}} \mathrm{div}_{\boldsymbol{x}}(\boldsymbol{F})(t,\boldsymbol{x},\ats{n}{i+1,j})\,\mathrm{d}\boldsymbol{x} - \fint_{K_{i,j}} \mathrm{div}_{\boldsymbol{x}}(\boldsymbol{F})(t,\boldsymbol{x},\ats{n}{i,j})\,\mathrm{d}\boldsymbol{x}  \right|\,\mathrm{d}t 
		=: \mathrm{Q}_1 + \mathrm{Q}_{2}.
		\label{eqn:q1_q2}
		\end{multline}
		The term $\mathrm{Q}_2$ can be estimated as 
		\begin{align}
		\mathrm{Q}_2 \le \delta \left|\mathrm{Lip}(\mathrm{div}_{\boldsymbol{x}}(\boldsymbol{F}))\right|\,|\ats{n}{i+1,j} - \ats{n}{i,j}|.
		\label{eqn:q_2}
		\end{align}
		Follow the proof of~\cite[Lemma 8]{claire_1999} and use ~\eqref{eqn:q1_q2} and~\eqref{eqn:q_2} to obtain
		\begin{align}
		|\alpha_{h,\delta}(t_{n+1},\cdot)|_{BV_{x,y}} \le{}& |\alpha_{h,\delta}(t_{n},\cdot)|_{BV_{x,y}} \left(1 + 6\delta \mathrm{Lip}(\partial_s \boldsymbol{F}) + \delta \mathrm{Lip}(\mathrm{div}_{\boldsymbol{x}}(\boldsymbol{F}))  \right) \\
		&+ \int_{t_{n}}^{t_{n+1}}|\Pi_{h}^0(\mathrm{div}_{\boldsymbol{x}}(\boldsymbol{F}))|_{BV_{x,y}}\,\mathrm{d}t.
		\end{align}
		Apply induction on the above result and use similar arguments as in the proof of Proposition~\ref{prop:temp_bv} to obtain the desired result.
	\end{proof}
	
		\section{Numerical examples}
	\label{sec:num_tests}
	We consider three examples to demonstrate the conclusions of Theorem~\ref{thm:bv-alpha} and Theorem~\ref{thm:non-lin:flux}. In Example~\ref{eg:ex_1}, we manufacture a source term such that the conservation law~\eqref{eqn:exact_smooth} has a smooth solution. In Example~\ref{eg:ex_2}, the source term is set to be zero and a discontinuous function is chosen as the initial data, and as a result the exact solution also becomes discontinuous. Example~\ref{eg:ex_2} helps to understand how the discontinuities in the solution affect the growth of $BV$ seminorm. In Example~\ref{eg:ex.3}, we consider a conservation law with fully nonlinear flux with an exact solution and demonstrate conclusions of Theorem~\ref{thm:non-lin:flux}.
	\begin{example}[smooth solution]
		\label{eg:ex_1}
		We consider the spatial domain $\Omega = (-1,1)^2$, temporal domain $(0,1)$, velocity vector field $\boldsymbol{u} = (u,v)$ defined by
		\begin{align}
		u(t,x,y) := t\sin(\pi x)\,\cos(\pi y/2)/16\;\;\text{ and }\;\;
		v(t,x,y) := t\sin(\pi y)\,\cos(\pi x/2)/16,
		\end{align}
		initial data $\alpha_0(x,y) := 1\;\;\forall\,(x,y) \in \Omega$, and an appropriate source term $\mathfrak{S}$ such that the problem 
		\begin{align}
		\label{eqn:exact_smooth}
		\left.
		\begin{array}{r l}
		\partial_t \alpha + \mathrm{div}(\boldsymbol{u}f(\alpha)) &={} \mathfrak{S}\;\;\mathrm{ in }\; \Omega_1\;\;\textmd{and} \\
		\alpha(0,x,y) &= \alpha_0(x,y)\;\;\forall\,(x,y) \in \Omega,
		\end{array}
		\right\}
		\end{align}
		has the unique smooth solution $\alpha(t,x,y) = \exp(t(x+y))\;\;\forall\,(t,x,y) \in \Omega_1$, where $\Omega_1 = (0,1) \times \Omega$.
	\end{example}
	
	\begin{example}[discontinuous solution]
		\label{eg:ex_2} The spatial domain is $\Omega = (-3,3)^2$ and the temporal domain is $(0,2)$. If the flux function $f$ in~\eqref{eqn:exact_smooth} is linear, then we set the velocity vector field $\boldsymbol{u}$ as $(1,1)$ and the source term $\mathfrak{S}$ as zero so that the problem~\eqref{eqn:exact_smooth} has the unique solution $\alpha(t,x,y) := \alpha_0(x - t,y - t)$. The initial data considered is $\alpha_0(x,y) = \boldsymbol{1}_{[x > -1/4]}/2 + \boldsymbol{1}_{[y > -1/4]}/2$, where $\boldsymbol{1}_{A}$ is the characteristic function of the set $A$. If the flux function $f$ is nonlinear, then we set the velocity vector field $\boldsymbol{u} = (u,v)$ as 
		\begin{align}
		u(t,x,y) = \sin(\pi x)\,\cos(\pi y/2)/20\;\;\text{and}\;\;
		v(t,x,y) = \sin(\pi y)\,\cos(\pi x/2)/20.
		\end{align}
		Note that in the case of nonlinear flux, the vector $\boldsymbol{u}$ is zero on the boundary of the square $(-3,3)^2$, and as a result we can take the boundary data $(\alpha\boldsymbol{u})_{\vert \partial \Omega}\cdot\boldsymbol{n}_{\vert \partial \Omega} = 0$, where $\boldsymbol{n}_{\vert \partial \Omega}$ is the outward normal to $\partial \Omega$. This homogeneous boundary condition on $\boldsymbol{u}$ is useful since the exact solution to the problem~\eqref{eqn:exact_smooth} with a nonlinear flux is not available. The source term and the initial condition remain the same as in the case of linear flux. 
	\end{example}
	
	\begin{example}[fully nonlinear flux] \label{eg:ex.3}
		The spatial and temporal domains, initial data, and exact solution are chosen as in Example~\ref{eg:ex_1}. The nonlinear conservation law considered is 
		\begin{align}
		\label{eqn:nonlin_smooth}
		\left.
		\begin{array}{r l}
		\partial_t \alpha + \mathrm{div}\left(\sin((x-t)\alpha),\,\cos((y-t)\alpha)\right) &={} \mathfrak{S}_N\;\;\mathrm{ in }\; \Omega_1\;\;\textmd{and} \\
		\alpha(0,x,y) &= \alpha_0(x,y)\;\;\forall\,(x,y) \in \Omega.
		\end{array}
		\right\}
		\end{align}
		The source term $\mathfrak{S}_N$ is chosen such that~\eqref{eqn:nonlin_smooth} has the smooth solution $\alpha(t,x,y) = \exp(t(x+y))$. Note that the divergence of the flux $\mathrm{div}\left(\sin((x-t)\alpha),\,\cos((y-t)\alpha)\right) = \alpha\cos((x - t)\alpha) - \alpha\sin((y - t)\alpha)$ is not identically zero.
	\end{example}
	
	\noindent We consider two fluxes in the tests: (i) linear flux, $f(s) = s$ and (ii) sinusoidal flux, $f(s) = \sin(2 \pi s)$. The numerical flux used is Godunov defined by
	\begin{align}
	g(a,b) = \left\{
	\begin{array}{r l}
	\displaystyle\max_{b <  s < a}(f(s)) &\text{ if } b < a,  \\
	\displaystyle\min_{a <  s < b}(f(s)) &\text{ if } a < b.
	\end{array}
	\right.
	\end{align}
	The families of meshes considered are (a) hexagonal, (b) triangular, (c) staggered, (d) Cartesian, and (e) perturbed Cartesian (see Figures~\ref{fig:hexa}--\ref{fig:pert}). 
	\begin{figure}[htp]
		\centering
		\label{fig:grids}
		\begin{subfigure}[b]{0.2\textwidth}
			\centering
			\includegraphics[width=\textwidth]{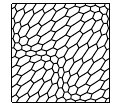}
			\caption{hexagonal}
			\label{fig:hexa}
		\end{subfigure}%
		\begin{subfigure}[b]{0.2\textwidth}
			\centering
			\includegraphics[width=\textwidth]{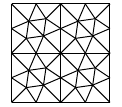}
			\caption{triangular}
			\label{fig:tria}
		\end{subfigure}%
		\begin{subfigure}[b]{0.2\textwidth}
			\centering
			\includegraphics[width=\textwidth]{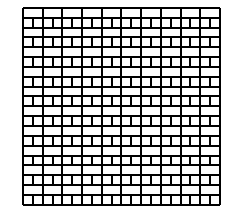}
			\caption{staggered}
			\label{fig:stag}
		\end{subfigure}%
		\begin{subfigure}[b]{0.2\textwidth}
			\centering
			\includegraphics[width=\textwidth]{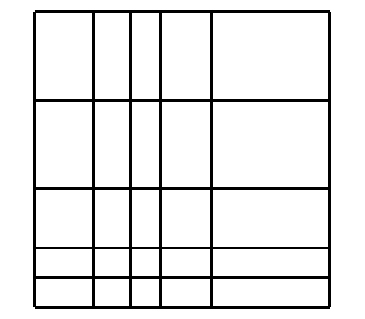}
			\caption{Cartesian}
			\label{fig:rect}
		\end{subfigure}%
		\begin{subfigure}[b]{0.2\textwidth}
			\centering
			\includegraphics[width=\textwidth]{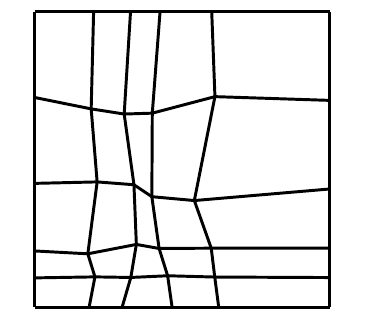}
			\caption{pert.Cartesian}
			\label{fig:pert}
		\end{subfigure}%
	\end{figure}
	
	\noindent The $BV$ and $L^1$ rates are defined by
	\begin{align}
	BV\;\mathrm{rate} =& \dfrac{\log\left(\left|\alpha_{h_{k+1},\delta_{k+1}}(T,\cdot)\right|_{BV_{x,y}}/\left|\alpha_{h_{k},\delta_{k}}(T,\cdot)\right|_{BV_{x,y}}\right)}{\log(h_{k+1}/h_{k})}\;\;\mathrm{ and } \\
	L^1\;\mathrm{rate} =& \dfrac{\log\left(\left|\alpha_{h_{k+1},\delta_{k+1}}(T,\cdot)\right|_{L^1(\Omega)}/\left|\alpha_{h_{k},\delta_{k}}(T,\cdot)\right|_{L^1(\Omega)}\right)}{\log(h_{k+1}/h_{k})}.
	\end{align}
	Discretisation factors, $BV$ norms, and $BV$ rates corresponding to Cartesian and perturbed Cartesian grids are presented in Tables~\ref{tab:tab1}--\ref{tab:tab4} and Tables~\ref{tab:tab5}--\ref{tab:tab8}. The $L^1$ errors and $L^1$ rates are also included whenever an exact solution is available. Arrangement of the contents in Tables~\ref{tab:tab1}--\ref{tab:tab8} are outlined in Table~\ref{tab:gen_ar} for clarity.
	The $BV$ rates corresponding to hexagonal, triangular, and staggered are presented in Table~\ref{tab:tab10} and Table~\ref{tab:tab9}. The quantities $L^1$ error and $L^1$ rate are omitted 
	for these three families of meshes since they follow a trend exactly similar to that of perturbed Cartesian grids. The $L^1$ and $BV$ rates of the discrete solutions obtained by applying scheme~\ref{eqn:scheme_2} to Example~\ref{eg:ex.3} is provided in Table~\ref{tab:tab11}.

	\begin{table}[htp]
		\centering
		\begin{tabular}{|c|c|c|c|}
			\hline
			\multicolumn{2}{|c|}{Tables showing $L^1$ and $BV$ rates} & \multirow{2}{*}{continuous flux} & \multirow{2}{*}{grid} \\ \cline{1-2}
			Example~\ref{eg:ex_1}   & Example~\ref{eg:ex_2}   &                                  &                       \\ \hline \hline
			Table~\ref{tab:tab1}	&     Table~\ref{tab:tab5}         & linear,  $f(s) = s$                         & Cartesian             \\ \hline
			Table~\ref{tab:tab2}	&     Table~\ref{tab:tab6}         & sinusoidal,  $f(s) = \sin(2\pi s)$                          & Cartesian             \\ \hline
			Table~\ref{tab:tab3}	&     Table~\ref{tab:tab7}         & linear,  $f(s) = s$                                & perturbed Cartesian   \\ \hline
			Table~\ref{tab:tab4}	&    Table~\ref{tab:tab8}          & sinusoidal,  $f(s) = \sin(2\pi s)$                              & perturbed Cartesian   \\ \hline 
		\end{tabular}
		\caption{Arrangement of contents in Tables~\ref{tab:tab1}--\ref{tab:tab4} and Tables~\ref{tab:tab5}--\ref{tab:tab8}.}
		\label{tab:gen_ar}
	\end{table} 
The captions of Tables~\ref{tab:tab1}--\ref{tab:tab4} and Tables~\ref{tab:tab6}--\ref{tab:tab10} are in the following format: \emph{example, continuous flux function, numerical flux function, grid type.}
	
	\subsection{Observations} 
		We recall three classical results from the theory of convergence analysis of finite volume schemes for conservation laws of the type~\eqref{eqn:cons_law}.
	\begin{enumerate}[label= $\mathrm{(R.\arabic*)}$,ref=$\mathrm{(R.\arabic*)}$,leftmargin=\widthof{(C.4)}+3\labelsep]
		\item\label{r.it.1} For a $BV$ initial data, finite volume approximations of conservation laws of the type~\eqref{eqn:cons_law} on structured Cartesian meshes converge with $h^{1/2}$ rate with respect to $L^\infty_tL^1$ norm \cite{Kuznetsov1976105}, and this result is extended to nearly Cartesian  meshes by B. Cockburn et al.~\cite{cockburn}. For generic meshes the $L^\infty_tL^1$ convergence rate is $h^{1/4}$~\cite[p.~188]{eymard}. 
		\item\label{r.it.2} The $BV$ seminorm of the finite volume solution grows with a rate not greater than $h^{-1/2}$~\cite[p.~168]{eymard}.  Further details can be found in \cite[p.~1777]{cockburn} and the references therein.
		\item \label{r.it.3} For $BV$ initial data finite volume approximations of nonlinear conservations of the type~\eqref{eqn:nonlin_smooth} converge with $h^{1/2}$ rate with respect to $L_1(\Omega_T)$ norm (see Theorem 4 and Remark 1 in~\cite{claire_1999}).
	\end{enumerate}
	
	In Tables~\ref{tab:tab1}--\ref{tab:tab7}, it can be observed that the order of convergence with respect to the $L^1$--norm is well above $1/4$. The $BV$ seminorm grows as $h$ decreases indicated by the negative values of $BV$ rate in Tables~\ref{tab:tab1}--\ref{tab:tab7}. But the growth rate is well below $h^{-1/2}$ except in the case of initial coarse meshes. The trend in  $L^1$ rate is related to the trend in $BV$ rate. A reduced $L^1$ rate is attributed to the fact that finite volume solutions on generic grids lack a uniform strong $BV$ estimate. The weak $BV$ estimate~\eqref{eqn:weak-bvrate} diminishes the $L^1$ rate from $h^{1/2}$ to $h^{1/4}$ in the case of non--Cartesian meshes.
	
  When the flux in linear, mesh is Cartesian (uniform or nonuniform), and~\eqref{eqn:cons_law} possesses a smooth solution, we obtain first order $L^1$ rate and the $BV$ rate decreases in magnitude but with oscillations. In the case of sinusoidal flux, the $L^1$ rate shows a slight reduction for coarse meshes but readily becomes well above $h^{1/2}$, which is the theoretical $L^1$ rate. Here also, $BV$ rate decreases in magnitude but with oscillations as $h$ decreases. 
	
	The numerical tests with perturbed Cartesian meshes also shows a similar behaviour. The linear flux exhibits first order $L^1$ rate and the sinusoidal flux a slightly reduced $L^1$ rate but well above $h^{1/2}$. However, the $BV$ rate shows a steady reduction in magnitude in both the linear and sinusoidal case. The $BV$ rate for hexagonal, triangular, and staggered meshes also show a steady decrease in magnitude as provided in Table~\ref{tab:tab9}.

	In Example~\ref{eg:ex_2}, we see a prominent reduction in the $L^1$ rate and this is due to the discontinuities in the weak solution to~\eqref{eqn:cons_law}. The explicit finite volume scheme introduce considerable numerical diffusion in the discrete solution by smearing out the sharp fronts, and thereby reducing the convergence rate. This reduction in the $L^1$ rate is visible in both the Cartesian and perturbed Cartesian cases (see Tables~\ref{tab:tab5} and~\ref{tab:tab7}). The  $BV$ rate is also decreasing in magnitude but with oscillations. In the sinusoidal flux case also $BV$ rates show the same pattern (see Tables~\ref{tab:tab6} and~\ref{tab:tab8}). For other non--Cartesian meshes also $BV$ rate seems to be decreasing in magnitude as presented in Table~\ref{tab:tab9}.
	
	\include{table_ex1}
	\include{table_ex2}
	
	\begin{table}[htp]
		\centering
		\begin{tabular}{|c |c H| H  c H   H | c  H | c|c|}
			\hline  
			\multirow{2}{*}{$h$} & \multirow{2}{*}{$\delta$} & \multirow{2}{*}{CFL ratio} & \multicolumn{3}{c|}{error} & \multicolumn{3}{c|}{rate}  &\multirow{2}{*}{$BV$ seminorm} &\multirow{2}{*}{$BV$ rate} \\ \cline{4-9}
			&             &                            & $L^\infty$ & $L^1$ & $L^2$ & $L^\infty$ & $L^1$ & $L^2$ &                      &     \\ \hline \hline 
			5.00$\mathrm{e}$-01  &   3.97$\mathrm{e}$-02              & 5.00$\mathrm{e}$-01                         & 3.87$\mathrm{e}$-02       & 1.47$\mathrm{e}$-01   & 3.39$\mathrm{e}$-02   & -        & -   & -   & 2.45$\mathrm{e}$+00        & -                 \\ \hline 
			2.50$\mathrm{e}$-01   &  1.98$\mathrm{e}$-02            & 5.00$\mathrm{e}$-01                         & 2.46$\mathrm{e}$-02       & 1.08$\mathrm{e}$-01   & 3.10$\mathrm{e}$-02   & 6.49$\mathrm{e}$-01        & 4.35$\mathrm{e}$-01   & 1.26$\mathrm{e}$-01   & 3.58$\mathrm{e}$+00        & -5.43$\mathrm{e}$-01                \\ \hline 
			1.25$\mathrm{e}$-01   &  9.94$\mathrm{e}$-03            & 5.00$\mathrm{e}$-01                         & 1.87$\mathrm{e}$-02       & 7.48$\mathrm{e}$-02   & 2.33$\mathrm{e}$-02   & 3.98$\mathrm{e}$-01        & 5.40$\mathrm{e}$-01   & 4.10$\mathrm{e}$-01   & 4.57$\mathrm{e}$+00         & -3.54$\mathrm{e}$-01                \\ \hline 
			6.25$\mathrm{e}$-02   &  4.97$\mathrm{e}$-03            & 5.00$\mathrm{e}$-01                         & 1.66$\mathrm{e}$-02       & 4.74$\mathrm{e}$-02   & 1.61$\mathrm{e}$-02   & 1.66$\mathrm{e}$-01        & 6.58$\mathrm{e}$-01   & 5.37$\mathrm{e}$-01   & 5.33$\mathrm{e}$+00        & -2.21$\mathrm{e}$-01                \\ \hline 
			3.12$\mathrm{e}$-02   &  2.48$\mathrm{e}$-03            & 5.00$\mathrm{e}$-01                         & 1.52$\mathrm{e}$-02       & 2.80$\mathrm{e}$-02   & 1.03$\mathrm{e}$-02   & 1.30$\mathrm{e}$-01        & 7.58$\mathrm{e}$-01   & 6.33$\mathrm{e}$-01   & 5.88$\mathrm{e}$+00         & -1.40$\mathrm{e}$-01                \\ \hline 
		\end{tabular}
		\caption{Example~\ref{eg:ex.3} -- Fully nonlinear flux and Cartesian grid.} 
		\label{tab:tab11}  
	\end{table}
	In the of conservation laws with fully nonlinear flux, it is clear from Table~\ref{tab:tab11} that the $BV$ rate is decreasing in magnitude steadily as $h$ decreases. This complements the uniform BV estimates in Theorem~\ref{thm:non-lin:flux}. The $L^1$ rate is also greater than the theoretical rate of $h^{1/2}$ except for the initial coarse mesh (see result~\ref{r.it.3}). Table~\ref{tab:tab11} also complements in~\cite[Lemma 8, Theorem 4]{claire_1999}, which provide the  boundedness of the $BV$ seminorm of discrete solutions corresponding to uniform square Cartesian grids. 
	\begin{remark}
		Choice of the functions $a$, $b$, $c$, and $d$ for the scheme~\eqref{eqn:scheme_2} is not arbitrary. It is crucial that $a$ and $c$ and nondecreasing, $b$ and $d$ are nonincreasing, and the CFL condition in Theorem~\ref{thm:non-lin:flux} holds. We use the following pairs to obtain the results provided in Table~\ref{tab:tab11}:
		\begin{align}
		a(t,x,y,z) = (\sin((x-t) z) + \mathfrak{M}z)/2, \quad b(t,x,y,z) = (\sin((x-t) z) - \mathfrak{M}z)/2, \\
		c(t,x,y,z) = (\cos((y-t) z) + \mathfrak{M}z)/2, \quad d(t,x,y,z) = (\cos((y-t) z) - \mathfrak{M}z)/2,
		\end{align}
		where $\mathfrak{M} = Lip(\boldsymbol{F})$. This choice of $\mathfrak{M}$ ensures the monotonicity conditions required by $a$, $b$, $c$, and $d$. Moreover, $a$, $b$, $c$, and $d$ become Lipschitz continuous with Lipschitz constant $Lip(\boldsymbol{F})$ so that the CFL condition in Theorem~\ref{thm:non-lin:flux} holds. 
	\end{remark}
	\begin{table}[htp]
	\centering
	\begin{tabular}{|c|c H| H H H H H H c|c|}
		\hline  
		\multirow{2}{*}{$h$} & \multirow{2}{*}{$\delta$} & \multirow{2}{*}{CFL ratio} & \multicolumn{3}{c}{} & \multicolumn{3}{c}{}  &\multirow{2}{*}{$BV$ seminorm} &\multirow{2}{*}{$BV$ rate} \\ 
		&             &                            & $L^\infty$ & $L^1$ & $L^2$ & $L^\infty$ & $L^1$ & $L^2$ &                      &     \\ \hline \hline 
		1.00E-01  &   1.00E-01              & 1.00E+00                         & 1.00E+00       & 3.33E-01   & 5.70E-01   & -        & -   & -   & 3.90E+00        & -                 \\ \hline 
		5.00E-02   &  5.00E-02            & 1.00E+00                         & 1.00E+00       & 5.00E-01   & 6.62E-01   & 0.00E+00        & -5.84E-01   & -2.16E-01   & 5.65E+00         & -5.36E-01                \\ \hline 
		2.50E-02   &  2.50E-02            & 1.00E+00                         & 1.00E+00       & 5.00E-01   & 6.75E-01   & 0.00E+00        & -1.40E-14   & -2.81E-02   & 6.86E+00         & -2.78E-01                \\ \hline 
		1.25E-02   &  1.25E-02            & 1.00E+00                         & 1.00E+00       & 5.00E-01   & 6.84E-01   & 0.00E+00        & -3.74E-14   & -1.98277E-02   & 9.00E+00         & -3.92E-01                \\ \hline 
		6.25E-03   &  6.25E-03            & 1.00E+00                         & 1.00E+00       & 5.00E-01   & 6.91E-01   & 0.00E+00        & -2.08E-13   & -1.38E-02   & 1.19E+01         & -4.14E-01                \\ \hline 
	\end{tabular}
	\caption{$BV$ seminorms of the finite volume solutions corresponding to~\eqref{eqn:despress} on staggered meshes. The parameters used are $\ell = 1$ and $T = 1/4$.} 
	\label{tab:tab_des} 
\end{table}
	\subsection{A remark on strong $BV$ estimate for non--Cartesian grids}
	In the case of Cartesian mesh, note that the $BV$ rate decreases in magnitude as $h$ decreases and the $BV$ seminorm stabilises eventually, which agrees with the conclusion of Theorem~\ref{thm:bv-alpha}. This is also supported by the higher values of $L^1$ rate than the theoretically predicted ones and the fact that the reduced convergence rate stems from lack of a strong $BV$ estimate\ (see result~\ref{r.it.1}). 
	
	Similar trends can be observed in the case of perturbed Cartesian grids also. These trends indicate there might be a possible way by which analysis in this article and in the previous works~\cite{claire_1999} could be extended to non--Cartesian grids also. Any such uniform estimate on strong $BV$ immediately provides a proof for the improved convergence rates. However, as of now any analytical proof of a strong $BV$ estimate on meshes other than nonuniform Cartesian grids is not available in the literature. A strong obstacle in this direction is the counterexample provided by B. Despr\'{e}s~\cite{despres}. This article~\cite{despres} presents an analytical proof that shows the $BV$ seminorms of finite volume solutions on a staggered grid, see Figure~\ref{fig:stag}, to the problem
	\begin{align}
	\left.
	\begin{array}{r l}
	\partial_t \alpha + u\,\partial_x \alpha &= 0,\;\;\mathrm{ for }\;\;(x,y) \in (-\ell,\ell)^2,\;\; 0 < t < T\\
	\alpha(t,x,y) &= \alpha_0(x,y)\;\;\mathrm{ for }\;\;(x,y) \in (-\ell,\ell)^2,
	\end{array}
	\right\}
	\label{eqn:despress}
	\end{align}
	with $\ell = 1$, $u = 1$, and $\alpha_0(x,y) = H(x - 1/2)$, where $H$ is the Heaviside step function blows up with an order greater than $h^{-1/2}$. This is supported by numerical experiments also. In Table~\ref{tab:tab_des} it is evident that the $BV$ seminorm is increasing and the rate of increase is also growing towards the theoretical rate of $h^{-1/2}$. Considering this result also, the uniform $BV$ estimate on non--Cartesian grids needs a deeper investigation. 
	
	\section{Extension to three spatial dimensions}
	\label{sec:3d-bv}
	An analogous result to Theorem~\ref{thm:bv-alpha} can be derived in a three spatial dimensional setting. The main result in stated in Theorem~\ref{thm:3d-bv-alpha}. The proof is omitted since it is similar to the proof of Theorem~\ref{thm:bv-alpha} and only more technical as a result of an extra spatial dimension.  Consider the partial differential equation on the time--space domain $\widehat{\Omega}_T = (0,T) \times \widehat{\Omega}$, wherein $\widehat{\Omega}  := (a_{\mathrm{L}},a_{\mathrm{R}}) \times (b_{\mathrm{L}},b_{\mathrm{R}}) \times (c_{\mathrm{L}},c_{\mathrm{R}})$ described by
	\begin{align}
	\label{eqn:3d-cons_law}\left.
	\begin{array}{r l}
	\partial_t \alpha + \mathrm{div}(\boldsymbol{u}f(\alpha)) &={} 0\;\;\mathrm{ in }\; \widehat{\Omega}_T\;\;\textmd{and} \\
	\alpha(0,\cdot) &= \alpha_0\;\;\mathrm{ in }\; \widehat{\Omega},
	\end{array}
	\right\}
	\end{align}
	where $\boldsymbol{u} = (u,v,w)$. The assumptions~\ref{as.1},~\ref{as.2}, and the condition~$\boldsymbol{u}_{\vert \partial \widehat{\Omega}} = \boldsymbol{0}$ hold. In addition to assume~\ref{as.3d-3} below.
	\begin{enumerate}[label= $\mathrm{(AS.3^{*})}$,ref=$\mathrm{(AS.3^{*})}$,leftmargin=\widthof{(AS.4)}+3\labelsep]
		\item \label{as.3d-3} There exists a generic constant $\mathscr{C} \ge 0$ such that $$\max\left( ||\boldsymbol{u}||_{L_t^1L^\infty(\widehat{\Omega}_T)}, ||\nabla \boldsymbol{u}||_{L_t^1L^\infty(\widehat{\Omega}_T)}, |\mathrm{div}(\boldsymbol{u})|_{L_t^1BV_{x,y,z}} \right) \le \mathscr{C}  < \infty.$$ 
	\end{enumerate}
	The temporal grid is same as in Section~\ref{sec:presentation}.  An admissible grid on the cube,  $\widehat{\Omega}$, is defined next.
	\begin{definition}[three dimensional admissible grid]
		\label{defn:3d-admis_grid}
		Define the one dimensional discretisations $\mathrm{X}_k := \left\{x_{-1/2},\cdots,x_{I+1/2}\right\}$, $\mathrm{Y}_h := \{y_{-1/2},\cdots,y_{J+1/2}\}$, and $\mathrm{Z}_{l}: = \{z_{-1/2},\cdots,z_{L+1/2}\}$, where $x_{-1/2} = a_{\mathrm{L}}$, $x_{I+1/2} = a_{\mathrm{R}}$, $y_{-1/2} = b_{\mathrm{L}}$, $y_{I + 1/2} = b_{\mathrm{R}}$, $z_{-1/2} = c_{\mathrm{L}}$, $z_{L + 1/2} = c_{\mathrm{R}}$, $k_{i} = x_{i+1/2} - x_{i-1/2}$, $h_{j} = y_{j+1} - y_{j-1/2}$, $l_{m} = z_{m+1} - z_{m-1/2}$, $k := \max{k_i}$, $h := \max{h_j}$, and $l := \max{l_m}$. The Cartesian grid $\mathrm{X}_k \times \mathrm{Y}_h \times \mathrm{Z}_{l}$ is said to be a three dimensional admissible grid if the following hold: for a fixed constant $\widetilde{c} > 0$, $(\widetilde{c})^{-1} \le \frac{h_j}{k_i} + \frac{k_i}{l_m} + \frac{l_m}{h_j} \le \widetilde{c}\;\;\forall i,j,l$.
	\end{definition}	
 Define the control volumes $K_{i,j,m} := (x_{i-1/2},x_{i+1/2}) \times (y_{j-1/2},y_{j+1/2}) \times (z_{m-1/2},z_{m+1/2})$, for $0 \le i \le I$, $0 \le j \le J$, and $0 \le m \le L$.

\begin{definition}[three dimensional discrete solution]
Set the discrete initial data as $\alpha_{i,j,m}^{0} := \fint_{K_{i,j,m}} \alpha_0(\boldsymbol{x})\,\mathrm{d}\boldsymbol{x}.$
\begin{subequations}
	\noindent The three dimensional discrete solution at the time step $n+1$, $\alpha_{h}^{n+1} : \widehat{\Omega} \rightarrow \mathbb{R}$, $n \ge 0$ is defined by  $\alpha_{h\vert K_{i,j,m}}^{n+1} = \ats{n+1}{i,j,m}$, where  
	\begin{align}
	\alpha_{i,j,m}^{n+1} ={}& \alpha_{i,j,m}^{n}
	- \mu_{i} \left(\mathrm{F}_{i+1/2,j,m}^{x} - \mathrm{F}_{i-1/2,j,m}^{x}\right) - \lambda_{j} \left(\mathrm{F}_{i,j+1/2,m}^{y} - \mathrm{F}_{i,j-1/2,m}^{y}\right) \nonumber \\
	&- \nu_{m} \left(\mathrm{F}_{i,j,m+1/2}^{z} - \mathrm{F}_{i,j,m-1/2}^{z}\right),
	\label{eqn:3d-fvm_scheme}
	\end{align}
	where $\mu_i = \delta/k_{i}$, $\lambda_{j} = \delta/h_{j}$, $\nu_{m} = \delta/l_{m}$, 
	\begin{align}
	\begin{array}{r l}
	\mathrm{F}_{i-1/2,j,m}^{x} :={}& \left(u_{i-1/2,j,m}^{n\,+}g(\ats{n}{i-1,j,m},\ats{n}{i,j,m}) - u_{i-1/2,j,m}^{n\,-}g(\ats{n}{i,j,m},\ats{n}{i-1,j,m}) \right),\\
	\mathrm{F}_{i,j-1/2,m}^{y} :={}& \left(v_{i,j-1/2,m}^{n\,+}g(\ats{n}{i,j-1,m},\ats{n}{i,j,m}) - v_{i,j-1/2,m}^{n\,-}g(\ats{n}{i,j,m},\ats{n}{i,j-1,m}) \right), \\
	\mathrm{F}_{i,j,m-1/2}^{z} :={}& \left(w_{i,j,m-1/2}^{n\,+}g(\ats{n}{i,j,m-1},\ats{n}{i,j,m}) - w_{i,j,m-1/2}^{n\,-}g(\ats{n}{i,j,m},\ats{n}{i,j,m-1}) \right),
	\end{array} 
	\label{eqn:3d-all_flux}
	\end{align}
	and for $a \in \mathbb{R}$,
	\begin{align}
	u_{i-1/2,j,m}^{n} =& \fint_{t_n}^{t_{n+1}}\fint_{y_{j-1/2}}^{y_{j+1/2}}\fint_{z_{m-1/2}}^{z_{m+1/2}} u(t,x_{i-1/2},s,r)\mathrm{d}r\,\mathrm{d}s\,\mathrm{d}t, \\
	v_{i,j-1/2,m}^{n} =& \fint_{t_n}^{t_{n+1}}\fint_{x_{i-1/2}}^{x_{i+1/2}}\fint_{z_{m-1/2}}^{z_{m+1/2}} v(t,s,y_{j-1/2},r)\mathrm{d}r\,\mathrm{d}s\,\mathrm{d}t,\quad \textmd{ and } \\
	w_{i,j,m-1/2}^{n} =& \fint_{t_n}^{t_{n+1}}\fint_{x_{i-1/2}}^{x_{i+1/2}}\fint_{y_{j-1/2}}^{y_{j+1/2}} w(t,s,r,z_{m-1/2})\mathrm{d}r\,\mathrm{d}s\,\mathrm{d}t.
	\end{align} 
\end{subequations}
\end{definition}
 \noindent Recall the time--reconstruct in Definition~\ref{def:time_rec}. Let $\alpha_{h,\delta} : \widehat{\Omega}_T \rightarrow \mathbb{R}$ be the time--reconstruct corresponding to the family of functions $\{\alpha_h^n\}_{n \ge 0}$. Define the $BV$ in the time--space domain $\widehat{\Omega}_T$ for a function $\beta : \widehat{\Omega}_T \rightarrow \mathbb{R}$ by
	\begin{align}
	|\beta|_{BV_{x,y,z,t}} :={} |\beta|_{L^1_{x}L^1_{y}L^1_{z}BV_t} + |\beta|_{L^1_{t}L^1_{x}L^1_{y}BV_{z}} + |\beta|_{L^1_{t}L^1_{y}L^1_{z}BV_{z}} + |\beta|_{L^1_{t}L^1_{z}L^1_{x}BV_{y}}.
	\end{align}
The next theorem shows that $\alpha_{h,\delta}$ is a function of $BV$. The proof follows analogous to Theorem~\ref{thm:bv-alpha}.
	\begin{subequations}
		\begin{theorem}[bounded variation]
			\label{thm:3d-bv-alpha}
			Let $\mathrm{X}_{k} \times \mathrm{Y}_{h} \times \mathrm{Z}_{l}$ be a three dimensional admissible grid. Assume~\ref{as.1}, \ref{as.2},~\ref{as.3d-3}, and the Courant--Friedrichs--Lewy (CFL) condition $4 \delta \max_{i,j}( \frac{1}{k_{i}} + \frac{1}{h_j} + \frac{1}{\nu_m}) \mathrm{Lip}(g) ||\boldsymbol{u}||_{L^\infty(\Omega_T)} \le \mathscr{C},$
		where $\mathscr{C} > 0$ is a constant. If $\alpha_0 \in L^\infty(\widehat{\Omega})\cap BV_{\boldsymbol{x}}(\widehat{\Omega})$, then $\alpha_{h,\delta}$ satisfies
			$|\alpha_{h,\delta}|_{BV_{x,y,z,t}} \le \widehat{\mathscr{C}}_{\mathrm{BV}},$
			 where $\widehat{\mathscr{C}}_{\mathrm{BV}}$ depends on $T$, $\alpha_0$, $f$, $g$, $||\nabla\boldsymbol{u}||_{L^1_tL^\infty(\widehat{\Omega}_T)}$, and $|\mathrm{div}(\boldsymbol{u})|_{L^1_tBV_{x,y,z}}$.  
		\end{theorem}
	\end{subequations}

	\section{Existence result for a ductal carcinoma model}
	\label{sec:ductal_tumour}
	A crucial application of Theorem~\ref{thm:bv-alpha} is that it enables us to prove the existence of a weak solution to coupled problems involving $\alpha$ and $\boldsymbol{u}$, such as~\eqref{eqn:ductal_tum} and~\eqref{eqn:cellphase}. 
In this section, we apply Theorem~\ref{thm:bv-alpha} to establish the existence of a solution to the \emph{ductal carcinoma in situ} problem~\eqref{eqn:ductal_tum}. The main idea is to combine a finite volume discretisation of~\eqref{eqn:cell_vf} and semi--discrete variational formulation of~\eqref{eqn:pv_sys}, and thereby reduce the interdependence between $\alpha$ and $\boldsymbol{u}$ to a semi--discrete relation $(\alpha^{n+1}_h,\boldsymbol{u}_h^{n+1}) = \boldsymbol{F}(\alpha_h^n, \boldsymbol{u}_h^n)$, where $(\alpha^{n}_h,\boldsymbol{u}_h^{n})$ is the discrete solution at time step $n$ and $h$ is the discretisation factor. Then, an inductive argument is used to show that the time--reconstruct $\alpha_{h,\delta}$, see Definition~\ref{def:time_rec}, constructed from $(\alpha_h^n)_{n \ge 0}$ is a function of $BV$ independent of $h$ and $\delta$. Finally, Helly's selection theorem, see Theorem~\ref{appen_id.c}, is invoked to obtain a convergent subsequence of $\{\alpha_{h,\delta} \}_{h,\delta}$ and the limit function is proved to be a weak solution of~\eqref{eqn:cell_vf}.

	\medskip
\noindent \textbf{Initial and boundary conditions}\medskip\\
 Set $\Omega = (0,1) \times (0,\ell)$ in the sequel. Fix an $\varepsilon$ such that $0 < \varepsilon < (\ell - 1)/2$ and define the auxiliary domain $\Omega(\varepsilon) := (0,1) \times (0,\ell - \varepsilon)$. Recall that for any $A  \subset \mathbb{R}^d$, the set $A_{T}$ is defined by $A_T = (0,T) \times A$.

  The initial concentration of the tumour cells and nutrient  are $\alpha(0,\boldsymbol{x}) = \alpha_0(\boldsymbol{x})$  and $c(0,\boldsymbol{x}) = c_0(\boldsymbol{x})$, respectively.  We assume that $\alpha_{0\vert (0,1) \times (1,\ell)} = 0$, which means the initial tumour occupies only a subset of $(0,1) \times (0,1)$ and later it spreads throughout the duct $\Omega$ as time evolves. In Proposition~\ref{prop:alpha_bounded}, we obtain a time $T_{\ast}$ such that the concentration of tumour cells remains zero for every $(t,x,y) \in (0,T_{\ast}) \times (0,1) \times (\ell - 2\epsilon,1)$. This temporal restriction is imperative as it enables us to obtain a uniform $BV$ estimate on the finite volume solutions from~\eqref{eqn:ductal_fvm}. The boundary conditions on~\eqref{eqn:pv_sys} and~\eqref{eqn:nutr} are as follows:
		\begin{align}
		\label{eqn:bc_1}
		\text{on } x \in \{0,1\}:&\quad \boldsymbol{u}\cdot\boldsymbol{n} = 0,\;\nabla v\cdot \boldsymbol{n} = 0,\;\nabla c\cdot \boldsymbol{n} = 0, \\
		\text{on } y = 0:&\quad \boldsymbol{u}\cdot\boldsymbol{n} = 0,\;\nabla u\cdot \boldsymbol{n} = 0,\;\nabla c\cdot \boldsymbol{n} = 0, \text{ and }\label{eqn:bc_2}
		\\
		\text{on } y = \ell:&\quad \boldsymbol{u}\cdot\boldsymbol{\tau} = 0,\;\nabla v\cdot \boldsymbol{n} =  0,\;\nabla u\cdot \boldsymbol{\tau} = \gamma,\;c = 0, p = 0,
		\label{eqn:bc_3}
		\end{align}
	where $\boldsymbol{\tau}$ and $\boldsymbol{n}$ are the unit tangent and unit normal vectors to $\partial \Omega$, respectively. The boundary condition $c = 0$ at $y = 0$ used in~\cite{Franks2003424} is replaced by  $\nabla c \cdot \boldsymbol{n} = 0$ in~\eqref{eqn:bc_2} and this 
	indicates that nutrient cannot enter or leave the interior of duct through the duct wall at $y = 0$.
	A supplementary condition $\nabla u \cdot \boldsymbol{\tau} = \gamma$ is addd in~\eqref{eqn:bc_3} and this
	manifests from~\eqref{eqn:nutr} and the boundary condition $c = 0$ at $y = \ell$. 
These changes are  reasonable from the modelling perspective as well and aid in obtaining the minimal regularity on $\boldsymbol{u}$ and $c$ that guarantees the convergence of discrete solutions.
	
	 The Sobolev spaces $W^{m,p}(\Omega)$, $H^m(\Omega) := W^{m,2}(\Omega)$, and $L^p(\Omega)$, where $1 \le p \le \infty$, are defined in the standard way. Set the product spaces $\boldsymbol{W}^{m,p}(\Omega) :={} W^{m,p}(\Omega) \times W^{m,p}(\Omega)$ and $\boldsymbol{H}^m(\Omega) :={} H^m(\Omega) \times H^m(\Omega)$.  For $\boldsymbol{u}= (u_1,\ldots,u_d) \in \Pi_{i=1}^d W^{m,p}(\Omega)$, $d \in \{1,2\}$, define the norm $\|\boldsymbol{u} \|_{m,p,\Omega} := \sum_{i=1}^d\sum_{|\boldsymbol{\beta}| \le m} \| \partial^{\boldsymbol{\beta}} u_i\|_{L^p(\Omega)}$, where $\boldsymbol{\beta} \in \mathbb{N}^d$ is a multi--index.   Let $\mathrm{X}_\mathrm{loc}(\Omega) := \{v \in L^2(\Omega):\;v_{\vert \omega} \in \mathrm{X}(\omega)\;\;\forall\,\omega  \subset\subset \Omega \}$, where $\mathrm{X} = H^m$ or $\mathrm{X} = \boldsymbol{H}^m$. Define the Hilbert spaces $\boldsymbol{H}$ and $V$ by
	\begin{align}
	\boldsymbol{H} :={}& \left\{\begin{array}{c|r}
	\boldsymbol{u} := (u,v) \in \boldsymbol{H}^1(\Omega) & \begin{minipage}{5.2cm}
	$\begin{aligned}
	\boldsymbol{u}\cdot\boldsymbol{n} ={}& 0\text{ at } x = 0,\,x = 1,\, y = 0, \\
	\text{ and }&{}\boldsymbol{u}\cdot\boldsymbol{\tau} = 0\text{ at } y = \ell
	\end{aligned}$
	\end{minipage}
	\end{array} \right\}\;\;\text{ and }\\
	V :={}& \{ v \in H^1(\Omega)\;:\; v = 0\text{ at } y = \ell \}.
	\label{eqn:cspace}
	\end{align}
	 For ease of notations, the explicit dependence of variables $(\alpha,\boldsymbol{u},p,c)$ on time is skipped. For instance, in~\eqref{eqn:u_weak}, $\boldsymbol{u}$ stands for $\boldsymbol{u}(t,\cdot)$. 
	\begin{definition}[weak solution]
		\label{def:weak_sol}
		A weak solution of the problem~\eqref{eqn:cell_vf}--\eqref{eqn:nutr} is a four tuple $(\alpha,\boldsymbol{u},p,c)$ such that the following conditions hold:
		\begin{enumerate}
			\item For $\nabla_{t,\boldsymbol{x}} = (\partial_t, \nabla )$,  the tumour cell concentration $\alpha \in L^\infty(\Omega_T)$ is such that, for every $\vartheta \in \mathscr{C}_{c}^{\infty}([0,T) \times \Omega)$,
			\begin{align}
			\int_{\Omega_T}\,\left((\alpha,\boldsymbol{u}\alpha)\cdot \nabla_{t,\boldsymbol{x}}\vartheta + \gamma \alpha(1 - c) \vartheta \right)\,\mathrm{d}\boldsymbol{x}\,\mathrm{d}t  + \int_{\Omega} \alpha_0(\boldsymbol{x}) \vartheta(0,\boldsymbol{x})\,\mathrm{d}\boldsymbol{x} = 0.
			\label{eqn:fvm_weak}
			\end{align}
			\item  The velocity--pressure system is such that $\boldsymbol{u} \in L^2(0,T;\boldsymbol{H})$, $p \in L^2(0,T;L^2(\Omega))$, and for every $\boldsymbol{\psi} := (\psi_1,\psi_2) \in L^2(0,T;\boldsymbol{H})$, $w \in L^2(0,T;L^2(\Omega))$,
			\begin{align}
			\int_{0}^T \mu\,\mathrm{a}(\boldsymbol{u},\boldsymbol{\psi})\,\mathrm{d}\boldsymbol{x}  - \int_{\Omega_T} p\,\mathrm{div}(\boldsymbol{\psi})\,\mathrm{d}\boldsymbol{x}\,\mathrm{d}t  ={}& \int_{0}^T\int_{y = \ell} \dfrac{\gamma \mu}{3} \psi_2\,\mathrm{d}s\,\mathrm{d}t, \;\;\text{ and } \label{eqn:u_weak} 
			\\
			\int_{\Omega_T} \mathrm{div}(\boldsymbol{u})\,w\,\mathrm{d}\boldsymbol{x}\,\mathrm{d}t   ={}& \int_{\Omega_T} \gamma (1 - c)\,w\,\mathrm{d}\boldsymbol{x}\,\mathrm{d}t,
			\label{eqn:p_weak}
			\end{align}
				where 
			$\mathrm{a}(\boldsymbol{v},\boldsymbol{w}) :={} \int_{\Omega} ( \nabla \boldsymbol{v}:\nabla \boldsymbol{w} + \dfrac{1}{3}\mathrm{div}(\boldsymbol{v})\mathrm{div}(\boldsymbol{w}) )\,\mathrm{d}\boldsymbol{x}$ for $\boldsymbol{v},\boldsymbol{w} \in \boldsymbol{H}^1(\Omega)$.
			\item The variable $c \in L^2(0,T;V)$ satisfies, for every $\varphi \in L^2(0,T;V)$ 
			\begin{align}
			\int_{0}^T \int_{\Omega} \nabla c\cdot \nabla  \varphi\,\mathrm{d}\boldsymbol{x}\,\mathrm{d}t = \int_{0}^T \int_{\Omega}  Q\alpha \varphi  \,\mathrm{d}\boldsymbol{x}\,\mathrm{d}t.
			\label{eqn:c_weak}
			\end{align}
		\end{enumerate}
	\end{definition}
	We define a semi--discrete scheme for~\eqref{eqn:cell_vf}--\eqref{eqn:nutr}, wherein the tumour cell concentration is discretised using a finite volume method, and velocity--pressure and nutrient concentration are obtained from the corresponding weak formulations and boundary conditions~\eqref{eqn:bc_1}--\eqref{eqn:bc_3}.
	
	\medskip
\noindent 	\textbf{Semi--discrete scheme:}\;\;\emph{
		Let $\mathrm{X}_{h} \times \mathrm{Y}_{h}$ be a uniform grid on $\Omega({\varepsilon})$ with $h < \varepsilon$ and $0 = t_0 < \cdots < T_{N} = T$ be a uniform temporal discretisation with $\delta = t_{n + 1} - t_{n}$. Set $\mu = \delta/h$. Construct a finite sequence of functions $(\alpha_h^n,\boldsymbol{u}_h^n,p_h^n,c_h^n)_{\{0 \le n < N \}}$ on $\Omega$ as follows. For $n = 0$, define $\alpha_h^0 : \Omega \rightarrow \mathbb{R}$ by $\alpha_h^0 := \alpha_{i,j}^0$, where $\alpha_{i,j}^0 := \fint_{K_{i,j}} \alpha_0(\boldsymbol{x})\,\mathrm{d}\boldsymbol{x}.$  For $0 \le n < N$, define the iterates as follows.
		\begin{enumerate}
			\item  The function $c_h^n \in V$ is defined by, for every $\varphi \in V$ it holds
			\begin{align}
			\int_{\Omega} (\nabla c_h^n \cdot \nabla \varphi - Q \alpha_h^n \varphi ) \,\mathrm{d}\boldsymbol{x} = 0.
			\label{eqn:sd_oxyg}
			\end{align}
			\item  The functions $(\boldsymbol{u}_h^n,p_h^n) \in \boldsymbol{H} \times L^2(\Omega)$ is defined by, for every $(\boldsymbol{\varphi},q) \in \boldsymbol{H} \times L^2(\Omega)$, setting $\boldsymbol{\varphi} = (\varphi_1,\varphi_2)$  it holds
				\begin{align}
			\mu\,\mathrm{a}(\boldsymbol{u}_h^n,\boldsymbol{\varphi}) - \int_{\Omega} p_h^n\,\mathrm{div}(\boldsymbol{\varphi})\,\mathrm{d}\boldsymbol{x} ={}& \int_{y = \ell} \dfrac{\gamma \mu}{3} \varphi_2\,\mathrm{d}s,\;\;\text{ and } \label{eqn:wf_cv}\\
			\int_{\Omega} \mathrm{div}(\boldsymbol{u}_h^n)q\,\mathrm{d}\boldsymbol{x} ={}& \int_{\Omega} \gamma (1 - c_h^n) \,q\,\mathrm{d}\boldsymbol{x}.
			\label{eqn:wf_p}
			\end{align}
			\item   Define $\alpha_h^{n+1}$ as the trivial extension of $\widehat{\alpha}_h^{n+1} : \Omega(\varepsilon) \rightarrow \mathbb{R}$, where $\widehat{\alpha}_h^{n+1} := \widehat{\alpha}_{i,j}^{n}$ on $K_{i,j} = (x_{i-1/2},x_{i+1/2}) \times (y_{j-1/2},y_{j+1/2})$ is obtained by
			\begin{small}
			\begin{gather}
		\hspace{-0.7cm}	\wats{n+1}{i,j} ={} \wats{n}{i,j} - \mu\left[(\widehat{\mathrm{F}}_{i+1/2,j} - \widehat{\mathrm{F}}_{i-1/2,j}) - (\widehat{\mathrm{G}}_{i,j+1/2} + \widehat{\mathrm{G}}_{i,j-1/2})\right] +  \gamma \delta \fint_{K_{i,j}}\wats{n}{i,j} (1 - c_h^n)\,\mathrm{d}\boldsymbol{x},
			\label{eqn:ductal_fvm}
			\end{gather}
			\end{small}
			where 
			\begin{align}
			\widehat{\mathrm{F}}_{i-1/2,j}^n :={} u_{i-1/2,j}^{n\,+}\ats{n}{i-1,j} - u_{i-1/2,j}^{n\,-}\ats{n}{i,j},\;\;
			\widehat{\mathrm{G}}_{i,j-1/2}^n :={} v_{i,j-1/2}^{n+}\ats{n}{i,j-1} - v_{i,j-1/2}^{n-}\ats{n}{i,j},
			\end{align}
			\begin{align}
			u_{i-1/2,j}^n = \fint_{t_{n}}^{t_{n+1}}\fint_{y_{j-1/2}}^{y_{j+1/2}} u_h^n(x_{i-1/2},s)\,\mathrm{d}s\,\mathrm{d}t,\;\;\mathrm{ and }\;\;\,v_{i,j-1/2}^n = \fint_{t_{n}}^{t_{n+1}} \fint_{x_{i-1/2}}^{x_{i+1/2}} v_h^n(s,y_{j-1/2})\,\mathrm{d}s\,\mathrm{d}t.
			\end{align}				
		\end{enumerate}}

%

	\subsection{Compactness}
	\label{sec:compact}
	
	The functions $\alpha_{h,\delta},\,\boldsymbol{u}_{h,\delta},\,p_{h,\delta},$ and $c_{h,\delta}$ are the time--reconstructs, see Definition~\ref{def:time_rec}, corresponding to the family of functions $(\alpha_h^n)_{\{n\ge 0\}}$, $(\boldsymbol{u}_h^n)_{\{n\ge 0\}}$, $(p_h^n)_{\{n\ge 0\}}$, and $(c_h^n)_{\{n\ge 0\}}$, respectively.

	\begin{theorem}[Compactness]
		\label{thm:compact}
		Fix a positive number $\alpha_M > a^0 = \sup_{\Omega}|\alpha_0|$. Assume that $\alpha_{0\vert (0,1) \times(1,\ell)} = 0$ and the following property on the discretisation factors $\delta$ and $h$:
		\begin{equation}
		\mathscr{C}_{\mathrm{ICFL}}^{\varepsilon} \le \dfrac{\delta}{h} \le \gamma\,\mathscr{C}_{\varepsilon}\,(1  + Q \mathscr{C} \sqrt{2\ell}\alpha_M).
		\label{eqn:chd_ductal}
		\end{equation}
		where the constants $\mathscr{C} > 0$ and $\mathscr{C}_{\varepsilon} > 0$ are specified in Lemma~\ref{lemma:c_apr} and Lemma~\ref{lemma:uv_reg}, respectively. Here, $\mathscr{C}_{\varepsilon}$ depends on $\varepsilon$.	The inverse CFL constant $0 < \mathscr{C}_{\mathrm{ICFL}}^{\varepsilon} < \gamma\,\mathscr{C}_{\varepsilon}\,(1  + Q \mathscr{C} \sqrt{2\ell}\alpha_M)$ depends on $\varepsilon$ but is independent of $h$ and $\delta$. Then, there exists a finite time $T_{\ast} < \infty$, a subsequence -- denoted with the same indices -- of the family of functions $\{(\alpha_{h,\delta},\boldsymbol{u}_{h,\delta},p_{h,\delta},c_{h,\delta})\}_{h,\delta}$ obtained from the semi--discrete scheme, and a four tuple of functions $(\alpha,\boldsymbol{u},p,c)$ such that 
		\begin{align}
		\alpha \in BV(\Omega_{T_\ast}),\,\boldsymbol{u} \in L^2(0,T_\ast;\boldsymbol{H}),\, p \in L^2(0,T_\ast;L^2(\Omega)),\,c\in L^2(0,T_\ast;V)
		\end{align}
		and as $h,\delta \rightarrow 0$
		\begin{itemize}
			\item $\alpha_{h,\delta} \rightarrow \alpha$ almost everywhere and in $L^\infty$ weak$-{\star}$ on $\Omega_{T_\ast}$, $\boldsymbol{u}_{h,\delta} \halfarrow \boldsymbol{u}$ weakly in $L^2(0,T_\ast;\boldsymbol{H})$,
			\item $p_{h,\delta} \halfarrow p$ weakly in $L^2(0,T_\ast;L^2(\Omega))$, and  $c_{h,\delta} \halfarrow c$ weakly in $L^2(0,T_\ast;V)$.
		\end{itemize}
	\end{theorem}
	\begin{remark}[Necessity of strong $BV$ estimate on $\alpha_{h,\delta}$]
		The uniform boundedness on $\alpha_{h,\delta}$ directly yield a subsequence that converges in weak--$\ast$ topology. However, this is not sufficient to show that the second term in the right hand side of~\eqref{eqn:ductal_fvm} converges weakly. It is shown that $c_{h,\delta}$ converges weakly in $L^2(0,T_{\ast};H^1(\Omega))$. Therefore, to establish  $\alpha_{h,\delta}(1 - c_{h,\delta})$ converges weakly to $\alpha(1 - c)$, the strong convergence of $\alpha_{h,\delta}$ is required. We employ Theorem~\ref{appen_id.c} to extract  a subsequence of $\{\alpha_{h,\delta}\}$   that converges almost everywhere and in $L^1(\Omega_{T_{\ast}})$ for which a strong uniform $BV$ estimate is necessary. 
	\end{remark}
\noindent The proof of Theorem~\ref{thm:compact} is achieved over multiple.  We establish:
	\begin{enumerate}[label= \Large${\bullet}$,ref=$\mathrm{(C.\arabic*)}$,leftmargin=\widthof{(C.4)}+3\labelsep]
		\item in Lemma~\ref{lemma:c_apr}, $c_{h}^n$  has $W^{2,p}(\Omega)$ regularity, which yields $\|c_{h}\|_{1,\infty,\Omega}$ estimate, 
		\item in Lemma~\ref{lemma:uv_reg},  $\boldsymbol{u}_h^n$ has $\boldsymbol{H}_{\mathrm{loc}}^3(\Omega)$ regularity, which yields local $\|\boldsymbol{u}_h^n\|_{1,\infty,\Omega}$ estimate,
		\item in Proposition~\ref{prop:alpha_bounded}, the finite volume solution $\alpha_{h,\delta}$ is  bounded, and
		\item in Proposition~\ref{prop:bv_ductal}, Corollary~\ref{cor:source} and the above steps are employed to prove that $\alpha_{h,\delta}$ is a function with $BV$. 
	\end{enumerate}
	Define the extended functions $\overline{c}_h^n$, $\overline{\boldsymbol{u}}_h^n := (\overline{u}_h^n,\overline{v}_h^n)$, and $\overline{p}_h^n$ on $\Omega_{\mathrm{ext}} := (-1,2) \times (-\ell,\ell)$ using even and odd reflections as follows. Let $a \in \{0,1,2\}$ and $b \in \{0,\ell\}$. Then, on $(a-1,a)\times(b-\ell,b)$ set $(\widetilde{x},\widetilde{y}) := (x(-2a^2 + 4a - 1) + (a^2 - a), (2b - \ell)y/\ell)$ and define
	\begin{gather}
	\left.
	\begin{array}{c}
	\overline{\alpha}_h^n (x,y) := \alpha_h^n(\widetilde{x},\widetilde{y}),\;\;\overline{c}_h^n (x,y) := c_h^n(\widetilde{x},\widetilde{y}),\;\;\overline{p}_h^n (x,y) := p_h^n(\widetilde{x},\widetilde{y}), \text{ and } \\
	\overline{u}_h^n (x,y) := (-2a^2 + 4a  -1)u_h^n(\widetilde{x},\widetilde{y}),\;\overline{v}_h^n (x,y) := (2b/\ell - 1)v_h^n(\widetilde{x},\widetilde{y}).
	\end{array} 
	\right\}
	\label{eqn:refl}
	\end{gather}
	In~\eqref{eqn:refl}, we have a compact representation of all reflections employed to construct the extended functions. A pictorial representation of~\eqref{eqn:refl} is provided in Figure~\ref{fig:reflections} for clarity.  We introduced three spatial domains so far and relations between them are represented in Figure~\ref{fig:domains}.
	\begin{figure}[htp]
		\centering
		\includegraphics[scale=0.8]{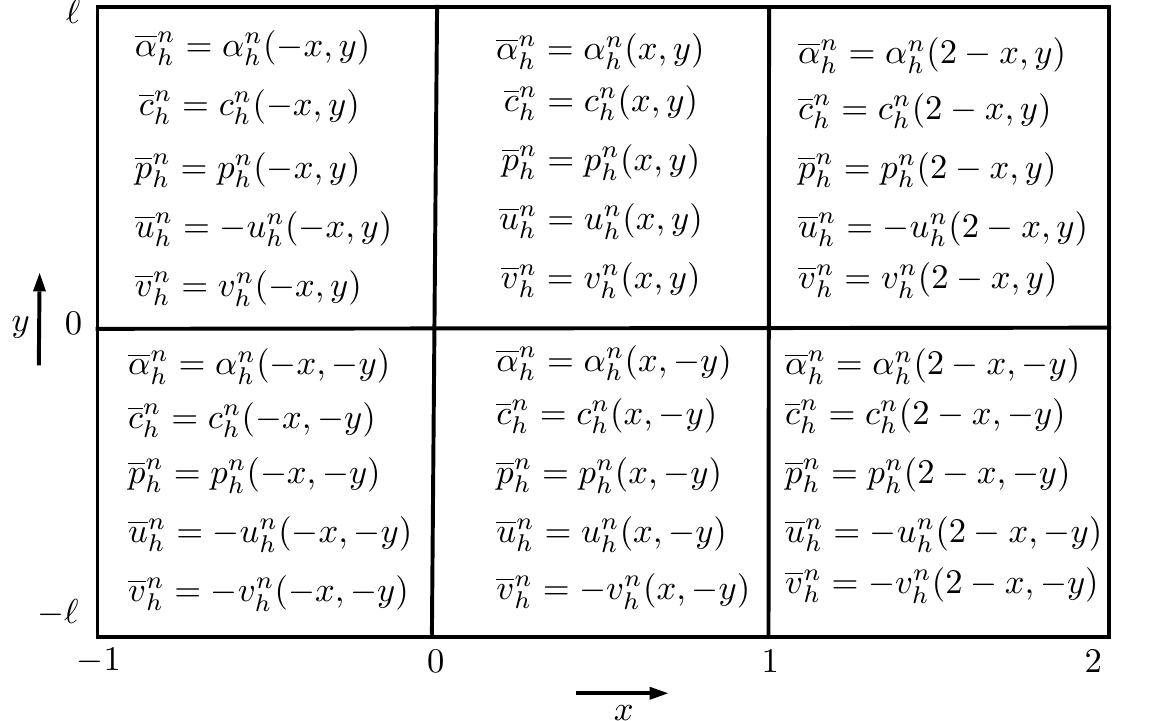}
		\caption{Extended functions on the rectangle $(-1,2) \times (-\ell,\ell)$}
		\label{fig:reflections}	
	\end{figure}
	
		\begin{remark}[auxiliary domain $\Omega(\varepsilon)$]
			The internal regularity result, see Theorem~\ref{thm:uv_reg}, only grants $\boldsymbol{u}_h^n \in \boldsymbol{H}^3(\Omega({\varepsilon}))$. The discontinuity in normal gradient of even reflection of $c$ about $y = \ell$ disables extending this local regularity of $\boldsymbol{u}_h^n$ up to $y = \ell$. As a result, it is  necessary to keep $\Omega({\varepsilon})$ to have enough regularity of $\boldsymbol{u}_h^n$ to move the analysis forward. We use the Sobolev embedding theorem to obtain $\boldsymbol{u}_h^n \in \boldsymbol{H}^{3}(\Omega({\varepsilon})) \hookrightarrow \boldsymbol{W}^{1,\infty}(\Omega({\varepsilon}))$, from which a $BV$ estimate on $\alpha_{h,\delta \vert \Omega({\varepsilon})}$, see Corollary~\ref{cor:source}, is derived. By imposing a restriction on time, the $BV$ regularity of $\alpha_{h,\delta}$ is extended to $\Omega$. 
	\end{remark}
	\begin{lemma}
		\label{lemma:c_apr}
		For every $n \ge 0$,~\eqref{eqn:sd_oxyg} has a unique solution $c_h^n \in V$. Moreover, it holds $\overline{c}_h^n \in H^{2}_{\mathrm{loc}}(\Omega_{\mathrm{ext}})$, $c_{h}^n \in W^{2,p}(\Omega)$ for any $p \ge 2$, and $\|c_h^n \|_{2,p,\Omega} \le  \mathscr{C}Q(2\ell)^{1/p}\|\alpha_h^n\|_{0,\infty,\Omega}$, where $\mathscr{C} > 0$ is a constant that depends only on $\Omega$.
	\end{lemma}
	\begin{proof}
		An application of Lax--Milgram theorem ensures the existence of a unique $c_h^n \in V$ that satisfies~\eqref{eqn:sd_oxyg}. 
	Observe that $\overline{c}_h^n \in H_{\mathrm{ext}} := \{v \in H^1(\Omega_{\mathrm{ext}}):\; v = 0 \text{ at } y = \ell, -\ell \}$. Apply change of variables to establish  $\int_{\Omega_{\mathrm{ext}}} \nabla \overline{c}_h^n \cdot \nabla v\,\mathrm{d}\boldsymbol{x} = Q \int_{\Omega_{\mathrm{ext}}} \overline{\alpha}_h^n v\,\mathrm{d}\boldsymbol{x}$ for every $v \in H_{\mathrm{ext}}$.
			Therefore,~Theorem~\ref{thm:c_reg} yields $\overline{c}_h^n \in H^2_{\mathrm{loc}}(\Omega_{\mathrm{ext}})$.
		
		 The $W^{2,p}(\Omega)$ regularity of $c_h^n$ is obtained by an application of odd reflection on $c_h^n$ about $y = \ell$. Set $\Lambda := (0,1) \times (0,2\ell)$. Define the function $\hat{c}_h^n : \Lambda \rightarrow \mathbb{R}$ by 
			\begin{align}
			\widehat{c}_h^n :={} \left\{ \begin{array}{c l}
			c_{h}^n(x,y)\;\; &\text{ if }\;\;y \le \ell, \text{ and } \\
			-c_{h}^n(x,2-y)\;\; &\text{ if }\;\;y > \ell.	    
			\end{array}
			\right.
			\end{align} 
			Let $f(x,y) = Q\alpha_h^n(x,y)$ if $y \le \ell$ and $f(x,y) = -Q\alpha_h^n(x,2-y)$ if $y \ge \ell$. Then, note that $\widehat{c}_h^n \in H^1(\Lambda)$ and $\int_{\Lambda} \nabla \widehat{c}_h^n \cdot \nabla v\,\mathrm{d}\boldsymbol{x} =  \int_{\Lambda} f v\,\mathrm{d}\boldsymbol{x}$ holds for every $v \in H^1(\Lambda)$.
			Hence, Theorem~\ref{thm:creg_local} shows that $\widehat{c}_h^n \in W^{2,p}(\Lambda)$, $p \ge 1$ and in particular, $\|c_h^n \|_{2,p,\Omega} \le  \mathscr{C}(2\ell)^{1/p}Q\|\alpha_h^n\|_{0,\infty,\Omega}$.
	\end{proof}
	\begin{figure}[htp]
	\centering
	\includegraphics[scale=0.8]{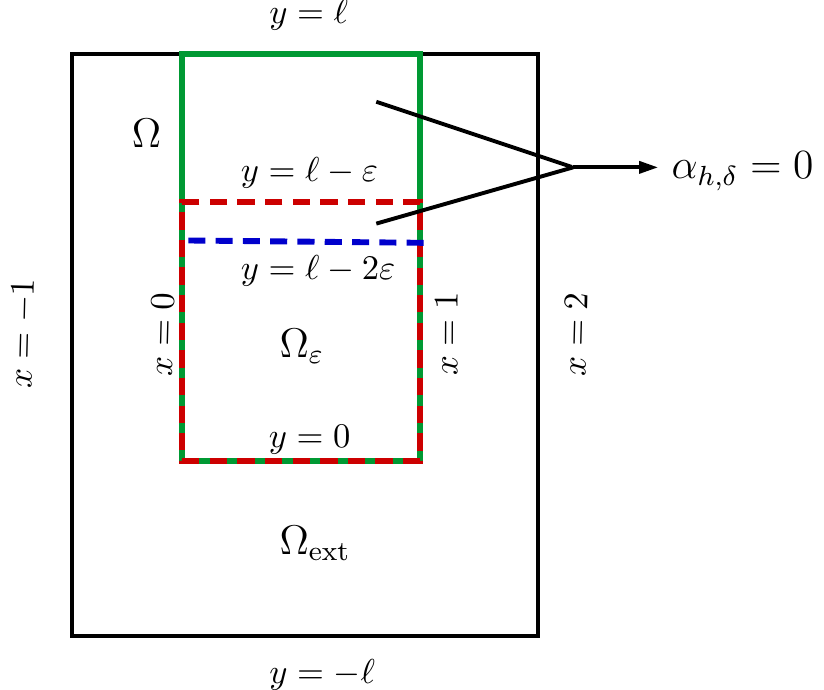}
	\caption{Relationship with domains}
	\label{fig:domains}	
\end{figure}
	\begin{lemma}
				\label{lemma:uv_reg}
		For every $n \ge 0$, there exists a unique $(\boldsymbol{u}_h^n,p_h^n) \in \boldsymbol{H} \times L^2(\Omega)$ that satisfies~\eqref{eqn:wf_cv}--\eqref{eqn:wf_p} for every $(\boldsymbol{\varphi},q) \in \boldsymbol{H} \times L^2(\Omega)$. Moreover, it holds $\overline{\boldsymbol{u}}_h^n \in \boldsymbol{H}^3_{\mathrm{loc}}(\Omega_{\mathrm{ext}})$ and for each $\varepsilon > 0$
		\begin{align}
		\| \boldsymbol{u}_h^n \|_{3,2,\Omega({\epsilon})} \le  \gamma\,\mathscr{C}_{\varepsilon}\,(1  + \mathscr{C}Q \sqrt{2\ell}\|\alpha_h^n\|_{0,\infty,\Omega}),
		\end{align} 
		where $\mathscr{C}_{\varepsilon} > 0$ depends only on $\varepsilon$.
	\end{lemma}
	\begin{proof}
		The existence of a unique solution $(\boldsymbol{u}_h^n,p_h^n) \in \boldsymbol{H} \times L^2(\Omega)$ follows from Ladyshenzkaya--Babuska--Brezzi theorem~\cite[p.~227]{Boffi}. Set the space \begin{align}
		\boldsymbol{H}_{\mathrm{ext}} :={}& \left\{\begin{array}{c|r}
		\boldsymbol{u} := (u,v) \in \boldsymbol{H}^1(\Omega_{\mathrm{ext}}) & \begin{minipage}{5.3cm}
		$\begin{aligned}
		&\boldsymbol{u}\cdot\boldsymbol{n} ={} 0\text{ at } x = -1,\,x = 2, \\
		&\text{ and }{}\boldsymbol{u}\cdot\boldsymbol{\tau} = 0\text{ at } y = \ell, \,y=-\ell
		\end{aligned}$
		\end{minipage}
		\end{array} \right\}.
		\end{align}
		Then, observe that the extended function $(\overline{\boldsymbol{u}}_h^n,\overline{p}_h^n)$ belongs to $\boldsymbol{H}_{\mathrm{ext}} \times L^2(\Omega_{\mathrm{ext}})$ and satisfies for every $(\boldsymbol{\varphi},q) \in  \boldsymbol{H}_{\mathrm{ext}} \times L^2(\Omega_{\mathrm{ext}})$
		\begin{align}
		\mu\,\int_{\Omega_{\mathrm{ext}}} ( \nabla \overline{\boldsymbol{u}}_h^n:\nabla \boldsymbol{\varphi} + \dfrac{1}{3}\mathrm{div}(\overline{\boldsymbol{u}}_h^n)\mathrm{div}(\boldsymbol{\varphi}) )\,\mathrm{d}\boldsymbol{x} - \int_{\Omega_{\mathrm{ext}}} \overline{p}_h^n\,\mathrm{div}(\boldsymbol{\varphi})\,\mathrm{d}\boldsymbol{x} ={}& \int_{y = \ell,-\ell} \dfrac{\gamma \mu}{3} \varphi_2\,\mathrm{d}s,\;\;\text{ and } \\
		\int_{\Omega_{\mathrm{ext}}} \mathrm{div}(\overline{\boldsymbol{u}}_h^n)q\,\mathrm{d}\boldsymbol{x} ={}& \int_{\Omega_{\mathrm{ext}}} \gamma (1 - \overline{c}_h^n) \,q\,\mathrm{d}\boldsymbol{x}. 
		\end{align}
		Since Lemma~\ref{lemma:c_apr} yields $\gamma (1 - \overline{c}_h^n) \in H^2_{\mathrm{loc}} (\Omega_{\mathrm{ext}})$, apply Theorem~\ref{thm:uv_reg} to conclude the proof.
	\end{proof}
	Lemmas~\ref{lemma:c_apr} and~\ref{lemma:uv_reg} are crucial in obtaining the supremum norm estimates on $c_h^n$ and $\mathrm{div}(\boldsymbol{u}_h^n)$ on $\Omega({\varepsilon})$. Since $\overline{c}_h^n \in W^{2,p}(\Omega)$ from Lemma~\ref{lemma:c_apr} and $\boldsymbol{u}_h^n \in \boldsymbol{H}^3(\Omega({\varepsilon}))$, the Sobolev embedding theorem with $p > 2$ yields
	\begin{align}
	\label{eqn:embed_c}
	\|c_h^n \|_{1,\infty,\Omega} \lesssim ||\overline{c}_h^n||_{2,p,\Omega} \le \mathscr{C}Q(2\ell)^{1/p}\|\alpha_h^n\|_{0,\infty,\Omega}, \text{ and } \\
	\label{eqn:embed_u}
	\|\boldsymbol{u}_h^n \|_{1,\infty,\Omega({\varepsilon})} \lesssim \|\overline{\boldsymbol{u}}_{h}^n \|_{3,2,\Omega({\varepsilon})} \le \gamma\,\mathscr{C}_{\varepsilon}\,(1  + \mathscr{C}Q  \sqrt{2\ell}\|\alpha_h^n\|_{0,\infty,\Omega}).
	\end{align}
	
	\begin{proposition}
		\label{prop:alpha_bounded}
		Fix a positive number $\alpha_M > a_0$. There exists a finite time $T_{\ast} > 0$ such that for every $t \le T_{\ast}$,   $\sup_{\Omega}|\alpha_{h,\delta}(t,\cdot)| \le \alpha_{M}$ holds.
	\end{proposition}
	\begin{proof}
	\textbf{Step 1:}	The proof employs strong induction on the time index $n$. Since $a^0 < \alpha_{M}$, the base case holds. To establish the inductive case, assume that $\sup_{\Omega(\varepsilon)}|\alpha_{h,\delta}(t_{k},\cdot)| \le \alpha_{M}$ for every $k \le n$. We establish that  $\sup_{\Omega(\varepsilon)}|\alpha_{h,\delta}(t_{n+1},\cdot)| \le \alpha_{M}$ holds for every $t_{n+1} < T_{1}$, for a fixed time $T_{1} > 0$.
	
		\medskip
\noindent	\textbf{Step 2:}	Recall $\|v\|_{L^1_tL^\infty(\Omega(\varepsilon)_T)} := \int_{0}^T \|v(t,\cdot)\|_{L^\infty(\Omega({\varepsilon}))}\,\mathrm{d}t$. The  results in~\eqref{eqn:embed_u} and~\eqref{eqn:chd_ductal} imply the CFL condition in Theorem~\ref{thm:bv-alpha}. Then,  Proposition~\ref{prop:boundedness} applied to~\eqref{eqn:ductal_fvm} yields, for any finite time $t < T$ 
		\begin{align}
		\|\alpha_{h,\delta}(t,\cdot)\|_{L^\infty(\Omega(\varepsilon))} \le \mathrm{B} \left(a_0 + \|\mathrm{div}(\boldsymbol{u}_h^n)\|_{L^1_tL^\infty(\Omega(\varepsilon)_T)} \right),
		\label{eqn:albnd_1}
		\end{align}
		where $\mathrm{B} = \exp(\|\mathrm{div}(\boldsymbol{u}_h^n)\|_{L^1_tL^\infty(\Omega(\varepsilon)_T)} + \gamma (T + ||c_{h}^n||_{L^1_tL^\infty(\Omega(\varepsilon)_T)}))$. Then~\eqref{eqn:embed_c},~\eqref{eqn:embed_u}, and~\eqref{eqn:albnd_1} imply
		$\|\alpha_{h,\delta}(t,\cdot)\|_{L^\infty(\Omega(\varepsilon))} \le \mathscr{F}(T)$, where
		\begin{align}
		\mathscr{F}(T) :=  \exp\left( T\gamma\,\mathscr{C}_{\varepsilon}\,(1  + \mathscr{C} Q  \sqrt{2\ell}\alpha_M) + TQ\mathscr{C}(2\ell)^{1/p}\alpha_M \right) \left(a_0 + T\gamma\,\mathscr{C}_{\varepsilon}\,(1  + Q \mathscr{C} \sqrt{2\ell}\alpha_M) \right).
		\end{align}
		Since $\mathscr{F}(0) - \alpha_{M} < 0$ and $\mathscr{F}$ is a nonnegative and monotonically increasing function, there exists a finite time $T_{1}$ such that $\|\alpha_{h,\delta}(t,\cdot)\|_{L^\infty(\Omega(\varepsilon))} \le \mathscr{F}(T_1) \le \alpha_{M}$ for every $t \in [0,T_1]$.
		
		\medskip
\noindent	\textbf{Step 3:}	Next, we need to show that $\alpha_{h,\delta}$ is bounded on $\Omega\backslash \Omega(\varepsilon)$. Note that $\alpha_{0}(x,y) = 0$ for $y \ge 1$. The finite speed of propagation of the scheme~\eqref{eqn:ductal_fvm} on $\Omega(\varepsilon)$ and ~\eqref{eqn:chd_ductal} yield $\alpha_{h,\delta} = 0$ on $(0,T_{2}) \times (\ell - 2\varepsilon,\ell)$, where   $T_{2} := (\ell - 2\varepsilon - 1) /(\gamma\,\mathscr{C}_{\varepsilon}\,(1  + Q \mathscr{C} \sqrt{2\ell}\alpha_M).$ Since $h < \varepsilon$, $\alpha_{i,j}^n  = 0$ for every $K_{i,j} \subset (\ell - 2\varepsilon,\ell)$, see Figure~\ref{fig:domains}. Define $T_{\ast} = \min(T_1,T_2)$ to obtain the conclusion.
	\end{proof}
	
  Observe that for every $(t,\boldsymbol{x},z) \in (0,T_\ast) \times \Omega \times (-\alpha_M,\alpha_M)$, the function $\mathfrak{S}(t,\boldsymbol{x},z) = \gamma (1 - c_{h,\delta}) z$  is Lipschitz continuous with respect to $z$, uniformly with respect to $t$ and $\boldsymbol{x}$ and Lipschitz continuous with respect to $\boldsymbol{x}$, uniformly with respect to $t$ and $z$. This is a direct consequence of~\eqref{eqn:embed_c}.

	\begin{proposition}
		The function $\alpha_{h,\delta} : (0,T_\ast) \times \Omega \rightarrow \mathbb{R}$ has bounded variation. Moreover, on $(0,T_\ast) \times \Omega$ it holds $|\alpha_{h,\delta}|_{BV_{x,y,t}} \le \mathscr{C}_{BV}$, where $\mathscr{C}_{BV}$ is independent of $h$ and $\delta$.
		\label{prop:bv_ductal}
	\end{proposition}
	
	The proof of Proposition~\ref{prop:bv_ductal} follows from an application of Corollary~\ref{cor:source}, the Lipschitz continuity of $\gamma (1 - c_{h,\delta}(t,\boldsymbol{x})) z$  of $(t,\boldsymbol{x},z)$ on $(0,T_\ast) \times \Omega \times (-\alpha_M,\alpha_M)$, and the fact that $\alpha_{h,\delta} = 0$ on $(0,1) \times (\ell - 2\epsilon,\ell)$, see Figure~\ref{fig:domains}.
	
	\subsubsection*{Proof of Theorem~\ref{thm:compact}}
	Recall that $\Omega_{T_{\ast}} = (0,T_{\ast}) \times \Omega$. Proposition~\ref{prop:bv_ductal} shows that $\alpha_{h,\delta} \in BV(\Omega_{T_{\ast}})$. Therefore, an application of Theorem~\ref{appen_id.c} provides the existence of subsequence of $\{\alpha_{h,\delta} \}$ -- assigned with the same indices -- and a function $\alpha \in  BV(\Omega_{T_\ast})$ such that $\alpha_{h,\delta} \rightarrow \alpha$ almost everywhere and $L^\infty$ weak$-{\star}$ on $\Omega_{T_\ast}.$ Lemma~\ref{lemma:c_apr} and Lemma~\ref{lemma:uv_reg} show that $c_{h,\delta} \in L^2(0,T_\ast;V)$ and $(\boldsymbol{u}_{h,\delta},p_{h,\delta}) \in L^2(0,T_\ast;\boldsymbol{H}) \times L^2(0,T_{\ast};L^2(\Omega))$ for every $h$ and $\delta$. Observe that $L^2(0,T_\ast;V)$ and $L^2(0,T_\ast;\boldsymbol{H}) \times L^2(0,T_{\ast};L^2(\Omega))$ are Hilbert spaces.  Hence, there exist subsequences of $\{c_{h,\delta}\}$ and $\{(\boldsymbol{u}_{h,\delta},p_{h,\delta})\}$, and functions  $c \in L^2(0,T_\ast;V)$ and $(\boldsymbol{u},p) \in L^2(0,T_\ast;\boldsymbol{H}) \times L^2(0,T_{\ast};L^2(\Omega))$ such that $c_{h,\delta} \halfarrow c$ weakly in $L^2(0,T_\ast;V)$ and $(\boldsymbol{u}_{h,\delta},p_{h,\delta}) \halfarrow (\boldsymbol{u},p)$ weakly in $L^2(0,T_\ast;\boldsymbol{H}) \times L^2(0,T_{\ast};L^2(\Omega))$.
		\subsection{Convergence}
	\label{sec:convergence}
		\begin{theorem}[Convergence]
		\label{thm:convergence}
		Let $(\alpha,\boldsymbol{u},p,\boldsymbol{c})$ be a limit of any subsequence of the family of functions $\{(\alpha_{h,\delta},\boldsymbol{u}_{h,\delta},p_{h,\delta},c_{h,\delta})\}_{h,\delta}$  obtained from the semi--discrete scheme in the sense of Theorem~\ref{thm:compact}. Then, $(\alpha,\boldsymbol{u},p,\boldsymbol{c})$ is a solution to the problem~\eqref{eqn:cell_vf}--\eqref{eqn:nutr} for the finite time $T_{\ast} < \infty.$
	\end{theorem}
	\medskip
	\noindent \textbf{Proof of Theorem~\ref{thm:convergence}.}\;The proof of Theorem~\ref{thm:convergence} has two steps.
	
	\medskip
	\noindent \textbf{Step 1. (Convergence of tumour cell concentration)}\;\;
	Let $\alpha : \Omega_{T_{\ast}} \rightarrow \mathbb{R}$ be a limit provided by Theorem~\ref{thm:compact} such that $\alpha_{h,\delta} \rightarrow \alpha$ almost everywhere in $\Omega_{T_{\ast}}$. Then, we show that $\alpha$ satisfies~\eqref{eqn:fvm_weak} for every $\vartheta \in \mathscr{C}_{c}^{\infty}([0,T_{\ast}) \times \Omega)$.

	Set $\varphi\in \mathscr{C}_c^{\infty} ([0,T_{\ast})\times \Omega)$ and $N_{\ast} = T_{\ast}/\delta$. For ease of notations, let $\varphi(t,\cdot)$ denotes its trivial extension on $\mathbb{R}^2$, for every $t \ge 0$. Multiply~\eqref{eqn:ductal_fvm} by $h^2 \vartheta_{i,j}^{n}$, $\vartheta_{i,j}^{n} := \int_{K_{i,j}} \vartheta(t_n,\cdot)\,\mathrm{d}\boldsymbol{x}$ and sum over the indices to obtain $T_1 + T_2^x + T_2^y = T_3$, where 
	\allowdisplaybreaks
	\begin{align}
	T_1 &:= h^2\sum_{n=0}^{N_{\ast}-1} \sum_{i=0}^{I} \sum_{j=0}^{J} (\ats{n+1}{i,j} - \ats{n}{i,j}) \vartheta_{i,j}^{n}, \\
	T_2^{x} &:= h^2\delta \sum_{n=0}^{N_{\ast}-1} \sum_{i=0}^{I}\sum_{j=0}^{J}  \left( \up{n}{i+1/2,j}\ats{n}{i,j} - \um{n}{i+1,j}\ats{n}{i+1/2,j} - \up{n}{i-1/2,j}\ats{n}{i-1,j} + \um{n}{i-1/2,j}\ats{n}{i,j} \right)\vartheta_{i,j}^{n},  \\
	T_2^{y} &:= h^2\delta \sum_{n=0}^{N_{\ast}-1} \sum_{i=0}^{I}\sum_{j=0}^{J}  \left( \vp{n}{i,j+1/2}\ats{n}{i,j} - \vm{n}{i,j+1/2}\ats{n}{i,j+1} - \vp{n}{i,j-1/2}\ats{n}{i,j-1} + \vm{n}{i,j-1/2}\ats{n}{i,j} \right)\vartheta_{i,j}^{n}, \text{ and }\\
	T_3 &:= h^2\delta \sum_{n=0}^{N_{\ast}-1} \sum_{i=0}^{I-1} \sum_{j=0}^{J-1} \gamma \vartheta_{i,j}^{n} \int_{t_n}^{t_{n+1}}\fint_{K_{i,j}}\alpha_{h,\delta}(1 - c_{h,\delta})\,\mathrm{d}\boldsymbol{x},\mathrm{d}t .
	\end{align}
	Define the piecewise constant function $\alpha^0_{h\vert{K_{i,j}}} := \fint_{K_{i,j}} \alpha_0(\boldsymbol{x})\,d\boldsymbol{x}$ for $0 \le i \le I$ and $0 \le j \le J$. Since $\vartheta^{N_\ast}_{i,j}=0$ for all $i,j$, use discrete integration by parts~\ref{appen_id.b} to arrive at
	\begin{align}
	T_1 = -h^2\sum_{n=0}^{N_{\ast}-1} \sum_{i=0}^{I} \sum_{j=0}^{J} (\vartheta_{i,j}^{n+1} - \vartheta_{i,j}^{n})\ats{n+1}{i,j} -  \int_{\Omega} \alpha^0_h(\boldsymbol{x}) \vartheta(0,\boldsymbol{x})\,\mathrm{d}\boldsymbol{x}. 
	\label{eqn:ca2-t1}
	\end{align}
 A direct calculation shows the first term in the right hand side of~\eqref{eqn:ca2-t1} is equal to
	\begin{align}
	-\sum_{n=0}^{N_{\ast}-1} \sum_{i=0}^{I}\sum_{j=0}^{J} \alpha_{i,j}^{n+1} \int_{t_{n}}^{t_{n+1}}\int_{K_{i,j}}  \partial_{t} \vartheta(t,\boldsymbol{x})\,d\boldsymbol{x}\,\mathrm{d}t = -\int_{\delta}^{T_{\ast} + \delta}\int_{\Omega}  \alpha_{h,\delta}(t,\boldsymbol{x}) \partial_{t} \vartheta(t-\delta,\boldsymbol{x})\,\mathrm{d}\boldsymbol{x}\,\mathrm{d}t. \nonumber 
	\end{align}
	Note that $\alpha_{h,\delta} \rightarrow \alpha$ almost everywhere (see Theorem~\ref{thm:compact}) as $h,\delta\to 0$. Then, apply Lebesgue's dominated convergence theorem to show that the first term in the right hand side of~\eqref{eqn:ca2-t1} converges to
	$-\int_{\Omega_{T_{\ast}}} \alpha(t,\boldsymbol{x}) \partial_{t} \vartheta(t,\boldsymbol{x}) \,\mathrm{d}t\,\mathrm{d}\boldsymbol{x}.$
Since $\alpha_h^0 \rightarrow \alpha_0$ in $L^2(\Omega)$, the second term in the right hand side of~\eqref{eqn:ca2-t1} converges to $-\int_{\Omega} \alpha_0(\boldsymbol{x}) \vartheta(0,\boldsymbol{x})\,\mathrm{d}\boldsymbol{x}$. 

The convergence of $T_{2}^{y}$ is shown next. The steps for $T_2^{x}$ follow similar steps. An application~\ref{appen_id.b} on $T_2^{y}$ leads to
	\begin{align}
	T_2^{y}  ={}&   \delta h^2 \sum_{n=0}^{N_{\ast}-1} \sum_{i=0}^{I}\sum_{j=0}^{J} \vartheta_{i,j}^{n} \left( |v_{i,j+1/2}^n|\dfrac{\ats{n}{i,j} - \ats{n}{i,j+1}}{2} - |v_{i,j-1/2}^n|\dfrac{\ats{n}{i,j-1} -\ats{n}{i,j}}{2} \right) \nonumber \\
	&+ \delta h^2 \sum_{n=0}^{N_{\ast}-1} \sum_{i=0}^{I} \sum_{j=0}^{J} \vartheta_{i,j}^{n} \left( v_{i,j+1/2}^n\dfrac{\ats{n}{i,j} + \ats{n}{i,j+1}}{2} - v_{i,j-1/2}^n\dfrac{\ats{n}{i,j-1} + \ats{n}{i,j}}{2} \right) \nonumber =: T_{21} + T_{22}.
	\end{align}
	Set  $\ats{n}{i,J+1} = 0$ and $\ats{n}{i,-1} = 0$. Then,  
	\begin{align}
	|T_{21}| \le&  \left| \delta h^2 \sum_{n=0}^{N_{\ast}-1} \sum_{i=0}^{I}\sum_{j=0}^{J-1} (\vartheta_{i,j+1}^{n} - \vartheta_{i,j}^n) |v_{i,j+1/2}^n|\dfrac{\ats{n}{i,j} - \ats{n}{i,j+1}}{2}\right|  + \mathcal{O}(h) \\
	&\le \dfrac{h}{2} ||\boldsymbol{u}_{h,\delta}||_{L^\infty(\Omega_{T_{\ast}})} ||\partial_x \vartheta(t,\boldsymbol{x})||_{L^\infty(\Omega_{T_{\ast}})} \sum_{n=0}^{N_{\ast}-1} \delta \sum_{i=0}^{I} h \sum_{j=0}^{J-1} |\ats{n}{i,j} - \ats{n}{i,j+1}| + \mathcal{O}(h),
	\end{align}
	and hence~\eqref{eqn:embed_u} and Proposition~\ref{prop:bv_ductal} imply $|T_{21}| \rightarrow 0$ as $h \rightarrow 0$. Use~\eqref{appen_id.b} to obtain
	\begin{align}
	T_{22} = -\delta h^2 \sum_{n=0}^{N_{\ast}-1} \sum_{i=0}^{I}\sum_{j=0}^{J} (\vartheta_{i,j+1}^n - \vartheta_{i,j}^n) v_{i,j+1/2}^n \dfrac{\alpha_{i,j}^n + \alpha_{i,j+1}^n}{2} + \mathcal{O}(h). 
	\label{eqn:ii1_3}
	\end{align}
	Add and subtract $\delta \sum_{n=0}^{N_{\ast}-1} \sum_{i=0}^{I} \sum_{j=0}^{J} (\vartheta_{i,j+1}^n - \vartheta_{i,j}^n) \frac{v_{i,j-1/2}^n}{2} \alpha_{i,j}^n$ to~\eqref{eqn:ii1_3} to arrive at
	\begin{align}
	T_{22} ={}& \delta h^2 \sum_{n=0}^{N_{\ast}-1} \sum_{i=0}^{I}\sum_{j=0}^{J} \dfrac{v_{i,j+1/2}^n\alpha_{i,j+1}^n}{2} (\vartheta_{i,j+1}^n - \vartheta_{i,j}^n - \vartheta_{i,j+2}^n + \vartheta_{i,j+1}^n) \nonumber \\
	&- \delta h^2 \sum_{n=0}^{N_{\ast}-1} \sum_{i=0}^{I} \sum_{j=0}^{J} (\vartheta_{i,j+1}^n - \vartheta_{i,j}^n) \dfrac{v_{i,j+1/2}^n + v_{i,j-1/2}^{n}}{2} \alpha_{i,j}^n. 
	\label{eqn:ca2-ta2}
	\end{align}
 Use of the definition of $\vartheta_{i,j}^n$, mean value theorem, and CFL condition~\eqref{eqn:chd_ductal} to show that the first term in the right hand side of~\eqref{eqn:ca2-ta2} converges to zero. Define $\partial_{h,\delta} \varphi : \Omega_{T_{\ast}} \rightarrow \mathbb{R}$ by $\partial_{h,\delta} \varphi := (\vartheta_{i,j+1}^n - \vartheta_{i,j}^n)/h$ on $(t_n,t_{n+1}) \times K_{i,j}$. Then the second term in the right hand side of~\eqref{eqn:ca2-ta2} can be expressed as
	\begin{align}
	-\int_{0}^{T_{\ast}}\int_{\Omega} v_{h,\delta}\alpha_{h,\delta}\partial_{h,\delta}\vartheta\,\mathrm{d}\boldsymbol{x}\,\mathrm{d}t  
	&\rightarrow -\int_{0}^{T_{\ast}}\int_{\Omega} v\,\alpha\,\partial_{x}\vartheta\,\mathrm{d}\boldsymbol{x}\,\mathrm{d}t,
	\end{align}
	where Lemmas~\ref{appen_id.e} and \ref{appen_id.f} are applied in the last step.  Follow the same steps for $T_{2}^{x}$ to obtain $T_{2} \rightarrow -\int_{0}^{T_{\ast}}\int_{\Omega} \alpha\,\boldsymbol{u}\cdot\nabla \vartheta\,\mathrm{d}\boldsymbol{x}\,\mathrm{d}t.$ Rewrite $T_{3}$ and apply Lemma~\ref{appen_id.e}
	\begin{gather}
	\int_{0}^{T_{\ast}} \int_{\Omega}  \gamma \alpha_{h,\delta} (1 - c_{h,\delta})\,\mathrm{d}\boldsymbol{x}\,\mathrm{d}t
	\rightarrow \int_{0}^{T}\int_{\Omega} \gamma \alpha (1 - c)\,\mathrm{d}\boldsymbol{x}\,\mathrm{d}t. \nonumber
	\end{gather}
 Plug the above in $T_1+T_2^x + T_2^y =T_3$ to arrive the desired conclusion.
 
 The proofs of step 2 and step 3 follows from a direct application of weak convergence of $(u_{h,\delta},p_{h,\delta})$ and $c_{h,\delta}$. Hence, we omit the proofs.
 
 \medskip
 \noindent \textbf{Step 2. (Convergence of pressure--velocity system)}\;\;Let $(\boldsymbol{u},p) : \Omega_{T_{\ast}} \rightarrow \mathbb{R}^3$ be a limit provided by Theorem~\ref{thm:compact} such that $\boldsymbol{u}_{h,\delta} \halfarrow \boldsymbol{u}$ weakly in $L^2(0,T_{\ast};\boldsymbol{H})$ and $p_{h,\delta} \halfarrow p$ weakly in $L^2(0,T_{\ast};L^2(\Omega))$. Then, $(\boldsymbol{u},p)$ satisfies~\eqref{eqn:pv_sys} for every $(\boldsymbol{\psi},q) \in L^2(0,T_{\ast};\boldsymbol{H}) \times  L^2(0,T;L^2(\Omega))$.
 
 \medskip
\noindent \textbf{Step 3. (Convergence of nutrient concentration)}\;\;Let $c : \Omega_{T_{\ast}} \rightarrow \mathbb{R}$ be a limit provided by Theorem~\ref{thm:compact} such that $c_{h,\delta} \halfarrow c$ weakly in $L^2(0,T_{\ast};V)$. Then $c$ satisfies~\eqref{eqn:c_weak} for every $\varphi \in L^2(0,T_{\ast};V)$.

	\section{Conclusions}
	\label{sec:con}
	A uniform estimate on total variation of discrete solutions obtained by applying finite volume schemes on conservation laws of the form~\eqref{eqn:cons_law} in two and three spatial dimensions for nonuniform Cartesian grids is proved. We relaxed the standard assumption that the advecting velocity vector is divergence free. This enables us to apply the finite volume scheme to problems in which the advecting velocity vector is a nonlinear function of the conserved variable. Since the underlying meshes are nonuniform Cartesian it is possible to adaptively refine the mesh on regions where the solution is expected to have sharp fronts. A uniform $BV$ estimate is also obtained for finite volume approximations of conservation laws of the type~\eqref{eqn:fully_non} that has a fully nonlinear flux on nonuniform Cartesian grids. Numerical experiments support the theoretical findings. The counterexample by B. Despr\'{e}s and numerical evidence from Table~\ref{tab:tab_des} indicate that nonuniform Cartesian grids are the current limit on which we can obtain uniform $BV$ estimates. Extending Theorem~\ref{thm:bv-alpha} to perturbed Cartesian grids (Figure~\ref{fig:pert}) might be the immediate future step. Theorem~\ref{thm:convergence}, which proves the existence of a weak solution of~\eqref{eqn:ductal_tum}, attests to the applicability of Theorem~\ref{thm:bv-alpha} in the analytical study of coupled systems involving conservation laws and elliptic equations.

\section*{Acknowledgement}
The author is grateful to Professors J\'{e}r\^{o}me Droniou, Jennifer Anne Flegg and Neela Nataraj for their valuable comments and suggestions.  The author also thanks Professors Claire Chainais-Hilairet, Thierry Gallou\"{e}t, and  Julien Vovelle for fruitful discussions. The author expresses gratitude towards ANZIAM Student Support Scheme, IIT Bombay, Professors J\'{e}r\^{o}me Droniou and Jennifer Anne Flegg for funding the travel and hospitality expenses during the author's stay at Monash University on February--March, 2020 during which this work was carried out.

	\bibliographystyle{plain}
	\bibliography{sbv_ref}
	
	\setcounter{equation}{0}
	\appendix
	\numberwithin{equation}{subsection}
	\section*{Appendix}
	\label{sec:appen}
	\renewcommand{\thesubsection}{\Roman{subsection}}
	\subsection{Identities}
	\label{appen_A}
	\renewcommand{\theequation}{\ref{appen_A}.\arabic{equation}}
	\begin{enumerate}[label= (\roman*).,ref=\ref{appen_A}.(\roman*)]
		\item\label{appen_id.a}  If $a,b,c,d \in \mathbb{R}$, then the following identities hold: $ab - cd  = \frac{(a+c)(b-d)}{2} + \frac{(a-c)(b+d)}{2} \quad \text{ and }$ and $a = a^{+} - a^{-}$, where 
	 $a^{+} = \max(a,0)$ and $a^{-} = -\min(a,0)$.
		\item\label{appen_id.b} \textbf{Discrete integration by parts formula. \cite[Section D.1.7]{droniou2018gradient}} 
		For any families $(a_n)_{n=0,\ldots,N}$ and $(b_n)_{n=0,\ldots,N}$ of real numbers, it holds
		\begin{align}
		\sum_{n=0}^{N-1} (a_{n+1} - a_{n})b_{n} = -\sum_{n=0}^{N-1} a_{n+1} (b_{n+1} - b_{n}) + a_{N}b_{N} - a_{0}b_{0}.
		\label{eqn:disc_int_parts}
		\end{align}
	\end{enumerate}
\subsection{Theorems}
\label{appen_B}
\renewcommand{\theequation}{\ref{appen_B}.\roman{equation}}
\begin{enumerate}[label= (\roman*).,ref=\ref{appen_B}.(\roman*)]
		\item\label{appen_id.c} \textbf{Helly's selection theorem. \cite[Theorem 4, p.~176]{evansmeas}.} 
		Let $\Omega \subset \mathbb{R}^d$ ($d \ge 1$) be an open and bounded set with a Lipschitz boundary $\partial \Omega$, and $(f_n)_{n\in \mathbb{N}}$ be a sequence in $BV(\Omega)$ such that $(||f_{n}||_{BV(\Omega)})_n$ is uniformly bounded. Then, there exists a subsequence $(f_n)_n$ up to re-indexing and a function $f \in BV(\Omega)$ such that as $n \rightarrow \infty$, $f_n \rightarrow f$ in $L^1(U)$ and almost everywhere in $\Omega$.
		
		\item\label{thm:c_reg} \textbf{Internal regularity of Poisson equation. {\cite[Theorem III.4.2]{Boyer2013}}} 	Let $f \in L^2(\Omega)$ and $\Omega \subset \mathbb{R}^2$ be an open and bounded set.  If $u \in H^1(\Omega)$ is a solution of the Poisson equation $-\Delta u = f$,  then $u \in H^2_{\mathrm{loc}}(\Omega)$. Also, for every bounded and open sets $\overline{\Omega_{1}} \subset \Omega_{2} \subset \overline{\Omega_{2}} \subset \Omega$ there exists a constant  $\mathscr{C}(\Omega_{1},\Omega_{2}) > 0$  independent of $u$ such that  $||u||_{2,2,\Omega_1} \le \mathscr{C} ||f||_{0,2,\Omega_1}$.
		
		\item\label{thm:creg_local} \textbf{Global regularity of Poisson equation. {\cite[Corollary 8.3.3]{mazya}}}  
		Set $m \ge 2$ and $p \ge 1$. Let $\Omega$ be a rectangle and $f \in W^{m - 2,p}(\Omega)$. If $u \in H^1(\Omega)$ is a solution of the boundary value problem $-\Delta u = f$, where $(\lambda - 1)\nabla u \cdot \boldsymbol{n} + \lambda u = 0$, $\lambda \in \{0,1\}$,   then $u \in W^{m,p}(\Omega)$.  
		
			\item\label{thm:uv_reg} \textbf{Internal regularity of Stokes equation. {\cite[Theorems IV.5.8, IV.6.1]{Boyer2013}}}  	Let $\Omega$ be an open and bounded set and $g \in H^{k+1}_{\mathrm{loc}}(\Omega)$, $k \ge 0$. Let $(\boldsymbol{u},p) \in \boldsymbol{H}_{\mathrm{loc}}^1(\Omega) \times L_{\mathrm{loc}}^2(\Omega)$ be a solution to the compressible Stokes system~\eqref{eqn:pv_sys}. Then, it holds $(\boldsymbol{u},p) \in \boldsymbol{H}_{\mathrm{loc}}^{k+2,2} \times H^{k+1}_{\mathrm{loc}}(\Omega)$. Also, for every bounded and open sets $\overline{\Omega_{1}} \subset \Omega_{2} \subset \overline{\Omega_{2}} \subset \Omega$ there exists a constant $\mathscr{C}(\Omega_{1},\Omega_{2}) > 0$ independent of $\boldsymbol{u}$ and $p$ such that $||\boldsymbol{u}||_{k+2,2,\Omega_1} + ||p||_{k+1,2,\Omega_1} \le \mathscr{C}||g||_{k+1,2,\Omega_1}.$	
			\end{enumerate}
		\subsection{Lemmas}
		\label{appen_C}
		\renewcommand{\theequation}{\ref{appen_C}.\arabic{equation}}
		 \begin{enumerate}[label= (\roman*),ref=\ref{appen_C}(\roman*)]
			\item \label{appen_id.e} \textbf{ Weak--strong convergence. \cite[Lemma D.8]{droniou2018gradient}.}
			If $p \in [0,\infty)$ and $q := p/(1-p)$ are conjugate exponents, $f_n \rightarrow f$ strongly in $L^p(X)$, and $g_n \halfarrow g$ weakly in $L^{q}(X)$, where $(X,\mu)$ is a measured space, then $\int_{X} f_{n}g_{n}\,\mathrm{d}\mu \rightarrow \int_{X} fg\,\mathrm{d}\mu.$

			The next result follows from Lebesgue's dominated convergence theorem. 
			\item\label{appen_id.f} \textbf{ Bounded--strong convergence.}
			If $f_n \rightarrow f$ in $L^2(X)$, $g_n \rightarrow g$ almost everywhere on $X$, $||g_n||_{L^\infty(X)}$ is uniformly bounded, then $f_ng_n$ converges to $fg$ in $L^2(X)$.	
		\end{enumerate}

\end{document}

%% file: table_ex1.tex
\begin{table}[htp]
	\centering
	\begin{tabular}{|c |c H| H  c H   H | c  H | c|c|}
		\hline  
		\multirow{2}{*}{$h$} & \multirow{2}{*}{$\delta$} & \multirow{2}{*}{CFL ratio} & \multicolumn{3}{c|}{error} & \multicolumn{3}{c|}{rate}  &\multirow{2}{*}{$BV$ seminorm} &\multirow{2}{*}{$BV$ rate} \\ \cline{4-9}
		&             &                            & $L^\infty$ & $L^1$ & $L^2$ & $L^\infty$ & $L^1$ & $L^2$ &                      &     \\ \hline \hline 
		5.00e-01  &   2.50e-01              & 5.00e-01                         & 1.95e-01       & 1.37e-01   & 1.71e-01   & -        & -   & -   & 2.02e+01        & -                 \\ \hline 
		2.50e-01   &  1.25e-01            & 5.00e-01                         & 1.11e-01       & 7.19e-02   & 9.29e-02   & 8.07e-01        & 9.30e-01   & 8.84e-01   & 2.31e+01         & -1.92e-01                \\ \hline 
		1.25e-01   &  6.25e-02            & 5.00e-01                         & 5.91e-02       & 3.82e-02   & 4.93e-02   & 9.19e-01        & 9.08e-01   & 9.12e-01   & 3.06e+01         & -4.07e-01                \\ \hline 
		6.25e-02   &  3.12e-02            & 5.00e-01                         & 3.01e-02       & 1.98e-02   & 2.54e-02   & 9.69e-01        & 9.50e-01   & 9.58e-01   & 3.53e+01         & -2.05e-01                \\ \hline 
		3.12e-02   &  1.56e-02            & 5.00e-01                         & 1.52e-02       & 1.01e-02   & 1.28e-02   & 9.87e-01        & 9.72e-01   & 9.78e-01   & 3.79e+01         & -1.03e-01                \\ \hline 
	\end{tabular}
\allowdisplaybreaks
\vspace{-0.2cm}
	\caption{Example~\ref{eg:ex_1}, linear, Godunov, Cartesian} 
		\label{tab:tab1} 
	\vspace{0.5cm} 
	\begin{tabular}{|c |c H| H  c H   H | c  H | c|c|}
		\hline  
		\multirow{2}{*}{$h$} & \multirow{2}{*}{$\delta$} & \multirow{2}{*}{CFL ratio} & \multicolumn{3}{c|}{error} & \multicolumn{3}{c|}{rate}  &\multirow{2}{*}{$BV$ seminorm} &\multirow{2}{*}{$BV$ rate} \\ \cline{4-9}
		&             &                            & $L^\infty$ & $L^1$ & $L^2$ & $L^\infty$ & $L^1$ & $L^2$ &                      &     \\ \hline \hline 
		5.00$\mathrm{e}$-01  &   3.97$\mathrm{e}$-02              & 5.00$\mathrm{e}$-01                         & 3.87$\mathrm{e}$-02       & 3.32$\mathrm{e}$-02   & 3.39$\mathrm{e}$-02   & -        & -   & -   & 2.42$\mathrm{e}$+01        & -                 \\ \hline 
		2.50$\mathrm{e}$-01   &  1.98$\mathrm{e}$-02            & 5.00$\mathrm{e}$-01                         & 2.46$\mathrm{e}$-02       & 3.35$\mathrm{e}$-02   & 3.10$\mathrm{e}$-02   & 6.49$\mathrm{e}$-01        & -1.02$\mathrm{e}$-02   & 1.26$\mathrm{e}$-01   & 3.20$\mathrm{e}$+01         & -3.99$\mathrm{e}$-01                \\ \hline 
		1.25$\mathrm{e}$-01   &  9.94$\mathrm{e}$-03            & 5.00$\mathrm{e}$-01                         & 1.87$\mathrm{e}$-02       & 2.59$\mathrm{e}$-02   & 2.33$\mathrm{e}$-02   & 3.98$\mathrm{e}$-01        & 3.71$\mathrm{e}$-01   & 4.10$\mathrm{e}$-01   & 3.78$\mathrm{e}$+01         & -2.41$\mathrm{e}$-01                \\ \hline 
		6.25$\mathrm{e}$-02   &  4.97$\mathrm{e}$-03            & 5.00$\mathrm{e}$-01                         & 1.66$\mathrm{e}$-02       & 1.64$\mathrm{e}$-02   & 1.61$\mathrm{e}$-02   & 1.66$\mathrm{e}$-01        & 6.53$\mathrm{e}$-01   & 5.37$\mathrm{e}$-01   & 4.03$\mathrm{e}$+01         & -9.13$\mathrm{e}$-02                \\ \hline 
		3.12$\mathrm{e}$-02   &  2.48$\mathrm{e}$-03            & 5.00$\mathrm{e}$-01                         & 1.52$\mathrm{e}$-02       & 9.58$\mathrm{e}$-03   & 1.03$\mathrm{e}$-02   & 1.30$\mathrm{e}$-01        & 7.81$\mathrm{e}$-01   & 6.33$\mathrm{e}$-01   & 4.12$\mathrm{e}$+01         & -3.13$\mathrm{e}$-02                \\ \hline 
	\end{tabular} 
\allowdisplaybreaks
\vspace{-0.2cm}
	\caption{Example~\ref{eg:ex_1}, sinusoidal, Godunov, Cartesian} 
		\label{tab:tab2}
		\vspace{0.5cm}
	\begin{tabular}{|c |c H| H  c H   H | c  H | c|c|}
		\hline  
		\multirow{2}{*}{$h$} & \multirow{2}{*}{$\delta$} & \multirow{2}{*}{CFL ratio} & \multicolumn{3}{c|}{error} & \multicolumn{3}{c|}{rate}  &\multirow{2}{*}{$BV$ seminorm} &\multirow{2}{*}{$BV$ rate} \\ \cline{4-9}
		&             &                            & $L^\infty$ & $L^1$ & $L^2$ & $L^\infty$ & $L^1$ & $L^2$ &                      &     \\ \hline \hline 
		5.70$\mathrm{e}$-01  &   2.85$\mathrm{e}$-01              & 5.00$\mathrm{e}$-01                         & 2.21$\mathrm{e}$-01       & 1.54$\mathrm{e}$-01   & 1.92$\mathrm{e}$-01   & -        & -   & -   & 1.97$\mathrm{e}$+01        & -                 \\ \hline 
		3.01$\mathrm{e}$-01   &  1.50$\mathrm{e}$-01            & 5.00$\mathrm{e}$-01                         & 1.34$\mathrm{e}$-01       & 8.76$\mathrm{e}$-02   & 1.12$\mathrm{e}$-01   & 7.79$\mathrm{e}$-01        & 8.95$\mathrm{e}$-01   & 8.42$\mathrm{e}$-01   & 2.71$\mathrm{e}$+01         & -4.99$\mathrm{e}$-01                \\ \hline 
		1.52$\mathrm{e}$-01   &  7.62$\mathrm{e}$-02            & 5.00$\mathrm{e}$-01                         & 7.18$\mathrm{e}$-02       & 4.65$\mathrm{e}$-02   & 6.00$\mathrm{e}$-02   & 9.23$\mathrm{e}$-01        & 9.26$\mathrm{e}$-01   & 9.22$\mathrm{e}$-01   & 3.30$\mathrm{e}$+01         & -2.88$\mathrm{e}$-01                \\ \hline 
		8.40$\mathrm{e}$-02   &  4.20$\mathrm{e}$-02            & 5.00$\mathrm{e}$-01                         & 4.01$\mathrm{e}$-02       & 2.61$\mathrm{e}$-02   & 3.36$\mathrm{e}$-02   & 9.75$\mathrm{e}$-01        & 9.68$\mathrm{e}$-01   & 9.68$\mathrm{e}$-01   & 3.64$\mathrm{e}$+01         & -1.65$\mathrm{e}$-01                \\ \hline 
		4.21$\mathrm{e}$-02   &  2.10$\mathrm{e}$-02            & 5.00$\mathrm{e}$-01                         & 2.03$\mathrm{e}$-02       & 1.33$\mathrm{e}$-02   & 1.71$\mathrm{e}$-02   & 9.88$\mathrm{e}$-01        & 9.80$\mathrm{e}$-01   & 9.77$\mathrm{e}$-01   & 3.84$\mathrm{e}$+01         & -7.84$\mathrm{e}$-02                \\ \hline 
	\end{tabular}
\vspace{-0.2cm}
\caption{Example~\ref{eg:ex_1}, linear, Godunov, perturbed Cartesian} 
\label{tab:tab3} 
\vspace{0.5cm}
	\begin{tabular}{|c |c H| H  c H   H | c  H | c|c|}
		\hline  
		\multirow{2}{*}{$h$} & \multirow{2}{*}{$\delta$} & \multirow{2}{*}{CFL ratio} & \multicolumn{3}{c|}{$\mathrm{e}$rror} & \multicolumn{3}{c|}{rate}  &\multirow{2}{*}{$BV$ seminorm} &\multirow{2}{*}{$BV$ rate} \\ \cline{4-9}
		&             &                            & $L^\infty$ & $L^1$ & $L^2$ & $L^\infty$ & $L^1$ & $L^2$ &                      &     \\ \hline \hline 
		5.70$\mathrm{e}$-01  &   4.54$\mathrm{e}$-02              & 5.00$\mathrm{e}$-01                         & 4.17$\mathrm{e}$-02       & 5.44$\mathrm{e}$-02   & 5.09$\mathrm{e}$-02   & -        & -   & -   & 2.42$\mathrm{e}$+01        & -                 \\ \hline 
		3.01$\mathrm{e}$-01   &  2.40$\mathrm{e}$-02            & 5.00$\mathrm{e}$-01                         & 2.83$\mathrm{e}$-02       & 3.74$\mathrm{e}$-02   & 3.51$\mathrm{e}$-02   & 6.06$\mathrm{e}$-01        & 5.87$\mathrm{e}$-01   & 5.84$\mathrm{e}$-01   & 3.31$\mathrm{e}$+01         & -4.91$\mathrm{e}$-01                \\ \hline 
		1.52$\mathrm{e}$-01   &  1.21$\mathrm{e}$-02            & 5.00$\mathrm{e}$-01                         & 1.90$\mathrm{e}$-02       & 2.69$\mathrm{e}$-02   & 2.44$\mathrm{e}$-02   & 5.85$\mathrm{e}$-01        & 4.81$\mathrm{e}$-01   & 5.30$\mathrm{e}$-01   & 3.73$\mathrm{e}$+01         & -1.75$\mathrm{e}$-01                \\ \hline 
		8.40$\mathrm{e}$-02   &  6.68$\mathrm{e}$-03            & 5.00$\mathrm{e}$-01                         & 1.67$\mathrm{e}$-02       & 1.77$\mathrm{e}$-02   & 1.71$\mathrm{e}$-02   & 2.14$\mathrm{e}$-01        & 7.06$\mathrm{e}$-01   & 5.98$\mathrm{e}$-01   & 4.13$\mathrm{e}$+01         & -1.70$\mathrm{e}$-01                \\ \hline 
		4.21$\mathrm{e}$-02   &  3.35$\mathrm{e}$-03            & 5.00$\mathrm{e}$-01                         & 1.51$\mathrm{e}$-02       & 1.01$\mathrm{e}$-02   & 1.08$\mathrm{e}$-02   & 1.45$\mathrm{e}$-01        & 8.11$\mathrm{e}$-01   & 6.63$\mathrm{e}$-01   & 4.19$\mathrm{e}$+01         & -1.84$\mathrm{e}$-02                \\ \hline 
	\end{tabular} 
\vspace{-0.2cm}
\caption{Example~\ref{eg:ex_1}, sinusoidal, Godunov, perturbed Cartesian} 
\label{tab:tab4}
\vspace{0.5cm}

	\begin{tabular}{|c|c|c|c||c|c|c|}
		\hline
		\multirow{6}{*}{\rotatebox{+90}{$\xleftarrow[]{h\mbox{\;\;\text{decreasing}}}$}} &
		\multicolumn{3}{c|}{\begin{tabular}[c]{@{}c@{}}$BV$ rate\\ (linear)\end{tabular}} &
		\multicolumn{3}{c|}{\begin{tabular}[c]{@{}c@{}}$BV$ rate\\ (sinusoidal)\end{tabular}} \\ \cline{2-7} 
		& hexagonal & triangular & staggered & hexagonal & triangular & staggered \\ \cline{2-7} 
		& -2.49E-01 & -2.43E-01  &           & -1.46E-01 & -2.39E-01  &           \\ \cline{2-7} 
		& -1.62E-01 & -1.26E-01  & -1.77E-01 & -6.87E-02 & -1.27E-01  & -9.39E-02 \\ \cline{2-7} 
		& -1.04E-01 & -6.79E-02  & -8.83E-02 & -2.41E-02 & -4.01E-02  & -3.10E-02 \\ \cline{2-7} 
		& -2.62E-01 & -3.80E-02  & -2.59E-02 & -1.28E-03 & -5.11E-03  & -1.13E-02 \\ \hline
	\end{tabular}
\vspace{-0.2cm}
	\caption{Example~\ref{eg:ex_1} -- Trend in the rate of BV norm for a smooth solution of~\eqref{eqn:cons_law}.} 
	\label{tab:tab10}
\end{table}

%% file: table_ex2.tex
\begin{table}[htp]
	\centering
	\begin{tabular}{|c |c H| H  c H   H | c  H | c|c|}
		\hline  
		\multirow{2}{*}{$h$} & \multirow{2}{*}{$\delta$} & \multirow{2}{*}{CFL ratio} & \multicolumn{3}{c|}{error} & \multicolumn{3}{c|}{rate}  &\multirow{2}{*}{$BV$ seminorm} &\multirow{2}{*}{$BV$ rate} \\ \cline{4-9}
		&             &                            & $L^\infty$ & $L^1$ & $L^2$ & $L^\infty$ & $L^1$ & $L^2$ &                      &     \\ \hline \hline 
		3.00E+00  &   9.37E-02              & 5.00E-01                         & 4.95E-01       & 4.14E-01   & 4.32E-01   & -        & -   & -   & 4.26E+00        & -                 \\ \hline 
		1.50E+00   &  4.68E-02            & 5.00E-01                         & 5.76E-01       & 8.16E-01   & 6.35E-01   & -2.17E-01        & -9.78E-01   & -5.53E-01   & 5.57E+00         & -3.85E-01                \\ \hline 
		7.50E-01   &  2.34E-02            & 5.00E-01                         & 7.81E-01       & 4.74E-01   & 4.43E-01   & -4.39E-01        & 7.81E-01   & 5.18E-01   & 6.33E+00         & -1.85E-01                \\ \hline 
		3.75E-01   &  1.17E-02            & 5.00E-01                         & 8.05E-01       & 3.70E-01   & 3.81E-01   & -4.34E-02        & 3.59E-01   & 2.15E-01   & 7.69E+00         & -2.80E-01                \\ \hline 
		1.87E-01   &  5.85E-03            & 5.00E-01                         & 1.10E+00       & 2.87E-01   & 3.36E-01   & -4.49E-01        & 3.66E-01   & 1.83E-01   & 8.75E+00         & -1.86E-01                \\ \hline 
	\end{tabular}
\vspace{-0.2cm}
	\caption{Example~\ref{eg:ex_2}, linear, Godunov, Cartesian} 
	\label{tab:tab5} 
\vspace{0.5cm}
	\begin{tabular}{|c |c  H| H  H H   H  H  H c|c|}
		\hline  
		\multirow{2}{*}{$h$} & \multirow{2}{*}{$\delta$} & \multirow{2}{*}{CFL ratio} & \multicolumn{3}{c}{} & \multicolumn{3}{c}{}  &\multirow{2}{*}{$BV$ seminorm} &\multirow{2}{*}{$BV$ rate} \\ 
		&             &                            & $L^\infty$ & $L^1$ & $L^2$ & $L^\infty$ & $L^1$ & $L^2$ &                      &     \\ \hline \hline 
		3.00E+00  &   1.49E-02              & 5.00E-01                         & 9.89E-01       & 6.81E-01   & 7.58E-01   & -        & -   & -   & 6.32E+00        & -                 \\ \hline 
		1.50E+00   &  7.46E-03            & 5.00E-01                         & 9.99E-01       & 2.21E+00   & 1.59E+00   & -1.51E-02        & -1.70E+00   & -1.07E+00   & 6.35E+00         & -7.31E-03                \\ \hline 
		7.50E-01   &  3.73E-03            & 5.00E-01                         & 1.00E+00       & 1.52E+00   & 1.27E+00   & -2.29E-05        & 5.33E-01   & 3.23E-01   & 6.60E+00         & -5.44E-02                \\ \hline 
		3.75E-01   &  1.86E-03            & 5.00E-01                         & 1.00E+00       & 1.50E+00   & 1.26E+00   & 0.00E+00        & 2.80E-02   & 1.42E-02   & 6.76E+00         & -3.52E-02                \\ \hline 
		1.87E-01   &  9.32E-04            & 5.00E-01                         & 1.00E+00       & 1.67E+00   & 1.35E+00   & 0.00E+00        & -1.60E-01   & -9.71E-02   & 7.08E+00         & -6.64E-02                \\ \hline 
	\end{tabular}
\vspace{-0.2cm}
\caption{Example~\ref{eg:ex_2}, sinusoidal, Godunov, Cartesian} 
\label{tab:tab6} 
\vspace{0.5cm}
	\begin{tabular}{|c |c H| H  c H   H | c  H | c|c|}
		\hline  
		\multirow{2}{*}{$h$} & \multirow{2}{*}{$\delta$} & \multirow{2}{*}{CFL ratio} & \multicolumn{3}{c|}{error} & \multicolumn{3}{c|}{rate}  &\multirow{2}{*}{$BV$ seminorm} &\multirow{2}{*}{$BV$ rate} \\ \cline{4-9}
		&             &                            & $L^\infty$ & $L^1$ & $L^2$ & $L^\infty$ & $L^1$ & $L^2$ &                      &     \\ \hline \hline 
		3.42E+00  &   1.06E-01              & 5.00E-01                         & 4.65E-01       & 3.98E-01   & 4.18E-01   & -        & -   & -   & 4.68E+00        & -                 \\ \hline 
		1.81E+00   &  5.65E-02            & 5.00E-01                         & 6.521E-01       & 7.24E-01   & 6.07E-01   & -5.27E-01        & -9.38E-01   & -5.86E-01   & 6.23E+00         & -4.49E-01                \\ \hline 
		9.14E-01   &  2.85E-02            & 5.00E-01                         & 6.34E-01       & 4.61E-01   & 4.26E-01   & 4.07E-02        & 6.59E-01   & 5.20E-01   & 6.48E+00         & -5.70E-02                \\ \hline 
		5.04E-01   &  1.57E-02            & 5.00E-01                         & 9.32E-01       & 3.70E-01   & 3.81E-01   & -6.47E-01        & 3.69E-01   & 1.86E-01   & 8.77E+00         & -5.07E-01                \\ \hline 
		2.53E-01   &  7.91E-03            & 5.00E-01                         & 1.05E+00       & 2.85E-01   & 3.33E-01   & -1.769E-01        & 3.78E-01   & 1.95E-01   & 9.51E+00         & -1.17E-01                \\ \hline 
	\end{tabular}
\vspace{-0.2cm}
\caption{Example~\ref{eg:ex_2}, linear, Godunov, perturbed Cartesian} 
\label{tab:tab7} 
\vspace{0.5cm}
	\begin{tabular}{|c |c  H| H  H H   H  H  H  c|c|}
		\hline  
		\multirow{2}{*}{$h$} & \multirow{2}{*}{$\delta$} & \multirow{2}{*}{CFL ratio} & \multicolumn{3}{c}{} & \multicolumn{3}{c}{}  &\multirow{2}{*}{$BV$ seminorm} &\multirow{2}{*}{$BV$ rate} \\ 
		&             &                            & $L^\infty$ & $L^1$ & $L^2$ & $L^\infty$ & $L^1$ & $L^2$ &                      &     \\ \hline \hline 
		3.42E+00  &   1.06E-01              & 5.00E-01                         & 4.65E-01       & 3.98E-01   & 4.18E-01   & -        & -   & -   & 6.32E+00        & -                 \\ \hline 
		1.81E+00   &  5.65E-02            & 5.00E-01                         & 6.52E-01       & 7.24E-01   & 6.07E-01   & -5.27E-01        & -9.38E-01   & -5.86E-01   & 6.54E+00         & -4.95E-02                \\ \hline 
		9.14E-01   &  2.85E-02            & 5.00E-01                         & 6.34E-01       & 4.61E-01   & 4.26E-01   & 4.07E-02        & 6.59E-01   & 5.20E-01   & 6.70E+00         & -3.51E-02                \\ \hline 
		5.042E-01   &  1.57E-02            & 5.00E-01                         & 9.32E-01       & 3.70E-01   & 3.81E-01   & -6.47E-01        & 3.69E-01   & 1.86E-01   & 7.57E+00         & -2.05E-01                \\ \hline 
		2.53E-01   &  7.91E-03            & 5.00E-01                         & 1.05E+00       & 2.85E-01   & 3.33E-01   & -1.76E-01        & 3.78E-01   & 1.95E-01   & 7.28E+00         & -5.75E-02                \\ \hline 
	\end{tabular}
\vspace{-0.2cm}
	\caption{Example~\ref{eg:ex_2}, sinusoidal, Godunov, perturbed Cartesian} 
	\label{tab:tab8} 
\vspace{0.5cm}
	\begin{tabular}{|c|H c|c|H c|c|c|c|}
		\hline
		\multirow{6}{*}{\rotatebox{+90}{$\xleftarrow[]{h\mbox{\;\;\text{decreasing}}}$}} &
		\multicolumn{4}{c|}{\begin{tabular}[c]{@{}c@{}}$BV$ rate\\ (linear flux)\end{tabular}} &
		\multicolumn{4}{c|}{\begin{tabular}[c]{@{}c@{}}$BV$ rate\\ (sinusoidal flux)\end{tabular}} \\ \cline{2-9} 
		& square    & hexagonal & triangular & staggered & square    & hexagonal & triangular & staggered \\ \cline{2-9} 
		& 4.66E-02  & -4.32E-01 & -3.90E-01  &           & 6.54E-02  & -7.19E-02 & -4.91E-01  &           \\ \cline{2-9} 
		& -4.44E-01 & -3.77E-01 & -3.61E-01  & -2.98E-01 & -1.89E-01 & -6.92E-02 & 1.93E-01   & 1.46E-01  \\ \cline{2-9} 
		& -2.16E-01 & -3.14E-01 & -6.60E-02  & -1.95E-01 & -4.14E-02 & -3.16E-02 & -2.39E-02  & -1.24E-01 \\ \cline{2-9} 
		& -1.37E-01 & -3.11E-01 & -1.71E-01  & -1.92E-01 & -4.88E-03 & -1.10E-03 & -9.69E-03  & -1.24E-04 \\ \hline
	\end{tabular}
\vspace{-0.2cm}
	\caption{Example~\ref{eg:ex_2} -- Trend in $BV$ rate for a discontinuous solution of~\eqref{eqn:cons_law}.} 
	\label{tab:tab9}
\end{table}